%% file: ms-spec-3d_cubic.tex
\documentclass[12pt]{amsart}

\pdfoutput=1

\usepackage{hyperref}
\usepackage{amssymb}
\usepackage{amsfonts}
\usepackage{amsmath}
\usepackage{graphicx}
\usepackage{amsthm}
\usepackage{amsxtra}
\usepackage{subfigure}
\usepackage{cite}

\include{grsmacros}

\include{jlmmacros}

\numberwithin{equation}{section}
\numberwithin{prop}{section}

\title
[Spectral Analysis for Matrix Hamiltonian Operators]
{Spectral Analysis for Matrix Hamiltonian Operators}

\author[J.L. Marzuola]
{Jeremy L. Marzuola}
\email{jm3058@columbia.edu}
\address{Applied Mathematics Department, Columbia University \\
200 S. W. Mudd, 500 W. 120th St., New York City, NY 10027, USA}

\author[G. Simpson]
{Gideon Simpson}
\email{simpson@math.toronto.edu}
\address{Mathematics Department, University of Toronto \\
Toronto, Ontario, Canada}

\begin{document}

\bibliographystyle{plain}
  
\begin{abstract}
  In this work, we study the spectral properties of matrix
  Hamiltonians generated by linearizing the nonlinear Schr\"odinger
  equation about soliton solutions.  By a numerically assisted proof,
  we show that there are no embedded eigenvalues for the three
  dimensional cubic equation.  Though we focus on a proof of the $3d$
  cubic problem, this work presents a new algorithm for verifying
  certain spectral properties needed to study soliton stability.

  Source code for verification of our comptuations, and for further
  experimentation, are available at
  \url{http://www.math.toronto.edu/simpson/files/spec_prop_code.tgz}. 
\end{abstract}

\maketitle

\tableofcontents

\section{Introduction}
\label{sec:intro}

The nonlinear Schr\"odinger equation (NLS) in $\R \times \R^d$,
\begin{equation}
  \label{eqn:nls}
  i \psi_t + \Delta \psi + g (|\psi|^2) \psi =  0, \quad \psi (0,\bx) = \psi_0 (\bx), 
\end{equation}
appears in many different contexts.  In applications, it appears as a
leading order approximation in nonlinear optics, many body quantum
systems, and hydrodynamics.  It is also intrinsically interesting as
a canonical example of the competition between nonlinearity and
dispersion.

For appropriate choices of the nonlinearity $g: \R \to \R$, the
equation is known to possess \emph{soliton} solutions, nonlinear bound
states satisfying \eqref{eqn:nls} with the \emph{ansatz}
\begin{equation*}
  \psi(t,\bx) = e^{i \lambda t} R(\bx;\lambda),
\end{equation*}
where $\lambda>0$ is the \emph{soliton} parameter. It is conjectured
that any solution of \eqref{eqn:nls} with appropriate nonlinearity
that does not disperse as $t \rightarrow \infty$ must eventually
converge to a finite sum of stable solitons.  This is referred to as the
``soliton resolution'' conjecture, a notoriously difficult problem to
formulate, see \cite{Tao2008}.

A natural property to investigate is the stability of the solitons.
In \cite{W1,W2,GSS}, a criterion for \emph{orbital} stability is
established.  Briefly, it says that if the derivative of the $L^2$ norm of
$R(\bx; \lambda)$ taken with respect to $\lambda$ is positive, then
the soliton is orbitally stable.  The perturbation 
remains small in a particular norm, $H^1$ in the case of NLS.  If this
 derivative is negative, the soliton is unstable.

Though these results on the orbital stability of solitons are very
powerful, relying on much of the variational structure of the
equations, they have three weaknesses.  The first is that they do not
say if a perturbed soliton reaches an asymptotically constant state;
orbital stability only assures us that the perturbation remains small.  The second
is that this approach provides no information if the
derivative of the $L^2$ norm vanishes.  This is the case of the $L^2$ critical focusing NLS
equation, with $g(s) = s^{d/2}$ and for \emph{saturated}
nonlinearities which possess \emph{minimal mass} solitons.  Finally,
the orbital stability fundamentally depends on the underlying equation
possessing a known variational structure.  Though this is not a valid
criticism for \eqref{eqn:nls}, it is a problem for other equations,
such as those studied in \cite{simpson08as}.

Alternatively, results such as
\cite{BusPer,BusSul,Cuc,Schlag1,KS1,RodSchSof}, and many others, prove
{\it asymptotic} stability of a soliton or a collection of solitons;
the system converges to specific solitons to as $t \to \infty$, and
the rest of the mass disperses.  Asymptotic stability is usually
proven perturbatively.  The leading order behavior of the perturbation
to the soliton is governed by the linearized operator.  First, linear
stability is proven by assesing the spectrum of the linearized
operator.   Then the nonlinearity is shown to be dominated by the
linear flow.  The spectrum of the linearized operator of \eqref{eqn:nls} with monomial
nonlinearity was studied in \cite{CGNT2007}.

Embedded eigenvalues of the linear operator are detrimental  to
proving the necessary linear estimates.  Indeed, they obstruct 
the needed dispersive estimates, as
demonstrated in \cite{CucPel}. Thus, it is standard to make the
assumption that there are no eigenvalues
embedded in the essential spectrum.  It is known that such a condition
cannot be proven directly using abstract properties of the linearized
operator, but must in fact be directly related to algebraic properties
of the soliton itself.  In this work we develop an algorithm
for studying the spectral properties of the operator appearing when
one linearizes \eqref{eqn:nls} about a soliton solution.  Furthermore,
we use this algorithm to prove the absence of \emph{embedded}
eigenvalues or resonances for four NLS problems.  We collect these
results in the following section.

\begin{rem}
  Though we only prove Theorem
  \ref{thm:speccond} for a small number of cases, our objective in
  this work is to present an approach for
  verfiying the spectral hypotheses required for  soliton stability theory.
\end{rem}

\subsection{Main Results}
\label{int:main}

Our results hinge on a so called {\it spectral property} based on
linearized matrix Schr\"odinger operators.  The specific form, and its
motivations, are developed Section \ref{sec:normal}.  In general, this
property can be formulated as:

\begin{defn}[The Generalized Spectral Property]  
\label{def:specprop}
  Let $d \geq 1$.  Given $L_\pm$ and a skew adjoint operator
  $\Lambda$, consider the two real Schr\"odinger operators
  \begin{equation*}
    \mathcal{L}_+ = -\Delta + \calV_+, \ \mathcal{L}_- = - \Delta + \calV_-,
  \end{equation*}
  defined by
  \begin{align*}
    \mathcal{L}_+ f & =\frac{1}{2}[L_+,\Lambda]f= \frac{1}{2} \left[
      {L}_+ \Lambda f -
      \Lambda {L}_+ f \right], \\
    \mathcal{L}_- f & =\frac{1}{2}[L_-,\Lambda]f= \frac{1}{2} \left[
      {L}_- \Lambda f - \Lambda {L}_- f \right]
  \end{align*}
  and
  \begin{equation*}
    \calV_\pm = \frac{1}{2}\bx \cdot \grad V_\pm.
  \end{equation*}
  Let the real quadratic form for $\mathbf{z} = (u,v)^T \in H^1\times
  H^1$ be
  \begin{align*}
    \calB(\mathbf{z},\mathbf{z}) & = \calB_+ (u,u) + \calB_- (v,v) \\
    & = \inner{\mathcal{L}_+ u}{ u} + \inner{\mathcal{L}_- v}{ v}.
  \end{align*}

  The system is said to satisfy a spectral property on the subspace $\calU
  \subseteq H^1\times H^1$ if there exists a universal constant
  $\delta_0 > 0$ such that $\forall \mathbf{z} \in \calU$,
  \begin{equation*}
    \calB(\mathbf{z},\mathbf{z}) > \delta_0 \int \paren{ | \nabla \mathbf{z} |^2 + e^{-|\by|} |\mathbf{z}|^2 } d\by.
  \end{equation*}
\end{defn}

In this work, the skew adjoint operator is
\begin{equation*}
  \Lambda f \equiv \frac{d}{2} f + \bx \cdot \nabla f  = \frac{d}{d \lambda} \left[ \lambda^{\frac{d}{2}} f( \lambda \bx) \right].
\end{equation*}
This has particular significance for the $L^2$ critical equation,
though we employ it in supercritcal problems.  This (mis)application
is discussed in Section \ref{s:discussion}.

Our results rely on the key observation of G. Perelman \cite{GPer} that
\begin{thm}
  \label{thm:coercive}
  Given the $JL$ operator, arising from the linearization of NLS about a
  soliton, where
\begin{equation*}
JL = \begin{pmatrix} 0  & L_- \\ -L_+ & 0 \end{pmatrix}
\end{equation*}
assume the $L_\pm$ operators satisfies the Spectral Property in the
  sense of Definition \ref{def:specprop}.  Then $JL$ has no
  embedded eigenvalues on the subspace $\calU$.
\end{thm}

\begin{proof}
Let us assume we have an embedded eigenstate
$\mathbf{z}_{\embd} = (u_{\embd}, v_{\embd})^T \in \calU$ corresponding to
eigenvalue $ i \tau_{\embd}$, $\tau_\embd > \lambda_0$.  Then,
\begin{align*}
  L_{-} v_{\embd} & = i \tau_{\embd} u_{\embd} , \\
  L_{+} u_{\embd} & = - i \tau_{\embd} v_{\embd}.
\end{align*}
Plugging directly into the form,
\begin{equation*}
  \begin{split}
    \calB(\mathbf{z}_{\embd},\mathbf{z}_{\embd}) & = \inner{\calL_+
      u_\embd}{u_\embd} + \inner{\calL_- v_\embd}{v_\embd}\\
    & = \frac{1}{2}\set{ \inner{\Lambda u_\embd}{L_+ u_\embd} +
      \inner{L_+ u_\embd}{\Lambda u_\embd}} \\
    &\quad+ \frac{1}{2} \set{\inner{\Lambda v_\embd}{L_- v_\embd} +
      \inner{L_- v_\embd}{\Lambda v_\embd}}\\
    & = \frac{1}{2}\set{ i \tau_\embd\inner{\Lambda u_\embd}{ v_\embd}
      -i \tau_\embd
      \inner{ v_\embd}{\Lambda u_\embd}} \\
    &\quad+ \frac{1}{2} \set{-i \tau_\embd\inner{\Lambda v_\embd}{
        u_\embd} +
      i \tau_\embd \inner{ u_\embd}{\Lambda v_\embd}}\\
    & = \frac{i \tau_\embd}{2} \set{\inner{\Lambda u_\embd}{ v_\embd}
      -\inner{ v_\embd}{\Lambda u_\embd} + \inner{v_\embd}{\Lambda
        u_\embd} - \inner{\Lambda u_\embd}{v_\embd} }\\
    &=0.
  \end{split}
\end{equation*}
\end{proof}

We remark that this holds not just for embedded eigenvalues, but for
any purely imaginary eigenvalue.  Thus, if the spectral property
holds, we are assured that there are no imaginary eigenvalues \emph{on
  the designated subspace}.  The subspace $\calU$ will be set by our analysis
of the spectrum in Section \ref{spec:lin}.  

\begin{defn}
\label{def:speccond}
Separately, we say that a linearized NLS problem satisfies of a {\it
  spectral condition} if it lacks both:
\begin{itemize}
\item Embedded eigenvalues,
\item Endpoint resonances.
\end{itemize}
\end{defn}

Our second theorem, which relies on the first is:
\begin{thm}
  \label{thm:speccond}
  The spectral condition holds for the 3d cubic equation,
  \eqref{eqn:nls} with $d=3$ and $g(s)=s$, linearized about the
  ground state soliton $R$.
\end{thm}
We adapt the methods of \cite{FMR} to give a
numerically assisted proof of this result.

\begin{rem}
  Though the main result of this paper will be to establish Theorem
  \ref{thm:speccond}, our algorithm can also be used to
  establish the spectral condition for the one dimensional equation
  with $g(s) = s^{2.5}$ and $g(s)=s^3$.  Again, this is for the
  problem linearized about the ground state soliton. 
\end{rem}

Separately, we establish that for 3d problems with nonlinearities
satisfying the conditions necessary for the existence of a soliton, as
discussed in \cite{BeLi}, one need only test for
embedded eigenvalues that are:
\begin{itemize}
\item Near the endpoints of the essential spectrum,
\item On a sufficiently low spherical harmonic.
\end{itemize}
In Appendix \ref{sec:mourre} we give a proof of the
following result using positive commutator arguments otherwise known
as Mourre estimates:

\begin{thm}
  \label{thm:spec}
  Given a Hamiltonian $\mathcal{H}$, there exists some $M > 0$ such
  that for $|\mu| > M$, there are no solutions $u_{\mu}$ such that
  \begin{eqnarray*}
    \mathcal{H} u_\mu = \mu u_\mu.
  \end{eqnarray*}
  Similarly, if $d \geq 2$, for any embedded eigenvalue, $u_{\mu}$,
  there exists a $K>0$ such that the spherical harmonic decomposition
  \begin{eqnarray*}
    u_\mu = \sum_{k = 0}^\infty \alpha_k (r,\mu) \phi_k (\phi,\theta)
  \end{eqnarray*} 
  consists only of harmonics with $k <K$.
\end{thm}

\begin{rem}
  Such results are well-known using resolvent estimate techniques;
  however, our approach provides easily computable limits on $M$ and
  $K$ in terms of the soliton solution.  
\end{rem}

\subsection{Organization of Results}
\label{sec:org}

In Sections \ref{sec:nlsprops}, \ref{sec:sollin} and
\ref{spec:lin}, we review fundamental properties of 
\eqref{eqn:nls} and the associated linearized operator.

In Sections \ref{spec:embres} and \ref{specnum:disc}, we collect
results and adapt the techniques of \cite{ES1} to prove properties of
the discrete spectrum and the absence of embedded resonances.

Finally, in Section \ref{sec:normal}, we prove an appropriate spectral property as
in Definition \ref{def:specprop}, based on the work in \cite{FMR}.
G.~Perelman's observation then rules out embedded eigenvalues.  This
proves Theorem \ref{thm:speccond}.

In Appendices \ref{spec:eig} and \ref{spec:sphhar},
we use Mourre multipliers to eliminate large embedded eigenvalues and
large spherical harmonics from the expansion of an embedded
eigenvalue.  Though these results have been known via resolvent
methods for some time, we aim to collect as much analytic
information about the spectrum as possible, providing bounds for future
estimates and computations.  An overview of our numerical methods with
benchmarks is then presented in Appendix \ref{sec:numerics}.

{\sc Acknowledgments.} This paper is an extension of a result of a
thesis done by the first author under the direction of Daniel Tataru
at the University of California, Berkeley that arose from a discussion
with Wilhelm Schlag and Galina Perelman at the Mathematisches
Forschungsinstitut Oberwolfach.  The first author is supported by an
NSF Postdoctoral Fellowship.  The second author is supported in part
by NSERC.  In addition, the authors wish to thank Gadi Fibich, Michael
Weinstein, Ian Zwiers and especially Wilhelm Schlag for many helpful
conversations throughout the development of the paper.

\section{Properties of the Nonlinear Schr\"odinger Equation}
\label{sec:nlsprops}

In this section we briefly review some important properties of
\eqref{eqn:nls}.  For additional details, we refer the reader to the
texts \cite{sulem1999nse, Caz}.

In general, for nonlinearity $g: \R \to \R$, \eqref{eqn:nls} possesses
the following invariants for data $\psi_0 \in H^1$ and $|\bx|\psi_0 \in
L^2$:
\begin{description}
\item[Conservation of Mass (or Charge)]
  \begin{equation*}
    Q (\psi) = \frac{1}{2} \int_{\R^d} |\psi|^2 d\bx = \frac{1}{2} \int_{\R^d} |\psi_0|^2 d\bx,
  \end{equation*}

\item[Conservation of Energy]
  \begin{align*}
    E(\psi) & =  \int_{\R^d} | \nabla \psi |^2 d\bx - \int_{\R^d} G(|\psi|^2) d\bx = \int_{\R^d} | \nabla \psi_0 |^2 d\bx - \int_{\R^d} G(|\psi_0|^2)
    d\bx,
  \end{align*}
  where
  \begin{equation*}
    G(t) = \int_0^t g(s) ds.
  \end{equation*}

\item[Pseudo-Conformal Conservation Law]
  \begin{equation*}
    \| (\bx + 2 i t \nabla ) \psi \|^2_{L^2} - 4 t^2 \int_{\R^d} G(|u|^2) d\bx = \|\b x \psi \|^2_{L^2} - \int_0^t \theta (s) ds,
  \end{equation*}
  where
  \begin{equation*}
    \theta (s) = \int_{\R^d} (4 (d+2) G(|\psi|^2) - 4 d g(|\psi|^2) |\psi|^2) d\bx.
  \end{equation*}
  Note that $(\bx + 2 i t \nabla )$ is the Hamilton flow of the linear
  Schr\"odinger equation, so the above identity relates how the
  solution to the nonlinear equation is effected by the linear flow.
\end{description}
Detailed proofs of these conservation laws can be arrived at easily
using energy estimates or Noether's Theorem, which relates
conservation laws to symmetries of an equation.

In this work, we restrict our attention to \emph{focusing}
nonlinearities, such that $g(s)\geq 0$ for all $ s\in \R$.  These are
the nonlinearities that can yield soliton solutions.  Often, $g(s) =
s^\sigma$ for some $\sigma>0$.  We examine one instance of the power
nonlinearity, the three dimensional cubic problem ($\sigma = 1$).

As noted, a soliton solution takes the form
\begin{equation*}
  \psi (t,\bx) = e^{i \lambda t} R(\bx;\lambda),
\end{equation*} 
where $\lambda > 0$ and $R (\bx;\lambda)$ is a positive, radially
symmetric, exponentially decaying solution of the equation:
\begin{equation}
  \label{eqn:sol}
  \Delta R - \lambda R + g(\abs{R}^2) R= 0.
\end{equation}
For power nonlinearities, the existence and uniqueness of the
ground state soliton is well known.  Additionally, the scaling
properties of this case permit us to take
$\lambda = 1$.  We do this in all that follows.

Existence of the soliton is proved by in \cite{BeLi} by minimizing the functional
\[
T(\psi) = \int | \nabla \psi |^2 d\bx
\] 
with respect to the constraint of fixed
\[
V(\psi) = \int [ G(|\psi|^2) - \frac{\lambda}{2} |\psi|^2 ] d\bx.
\]
Then, using the minimizing sequence and Schwarz symmetrization, one obtains the
existence of the nonnegative, spherically symmetric, decreasing
soliton solution.  Uniqueness is established in \cite{Mc} by ODE
methods.

An important relation is that $Q({\lambda}) = Q(R(\cdot;\lambda))$ and
$E({\lambda}) = E(R(\cdot;\lambda))$ are differentiable with respect
to $\lambda$.  This fact can be determined from the early works of
Shatah, namely \cite{Sh1}, \cite{Sh2}.  By differentiating Equation
\eqref{eqn:sol}, $Q$ and $E$ with respect to $\lambda$, we have
\begin{equation*}
  \partial_{\lambda} E = - \lambda \partial_{\lambda} Q.
\end{equation*}

Variational techniques developed in \cite{GSS} and \cite{ShSt} tell us
that when $\delta ( \lambda ) = E({\lambda}) + \lambda Q({\lambda})$
is convex, or $\delta '' (\lambda) > 0$, the soliton is orbitally stable.
For $\delta '' (\lambda) < 0$ the soliton is unstable to small perturbations.
This stability (instability) directly is closely related to the
eigenvalues of the matrix Hamiltonians resulting from linearizing NLS about a soliton.  For a brief
reference on this subject, see \cite{SS}, Chapter 4.

\section{Linearization about a Soliton}
\label{sec:sollin}

Let us write down the form of NLS linearized about a soliton
solution.  First, we assume we have a solution $\psi = e^{i
  \lambda t}(R + \phi(\bx,t))$.  Inserting this into the equation, we have
\begin{equation}
  i (\phi)_t + \Delta (\phi) =-g( R^2) \phi - 2 g'( R^2 )  R^2 \text{Re}(\phi) + \bigo(\phi^2),
\end{equation}
by splitting $\phi$ up into its real and imaginary parts, then doing a
Taylor Expansion. Hence, if $\phi = u + iv$, we get
\begin{equation}
  \partial_t \begin{pmatrix}
    u \\
    v
  \end{pmatrix} =J L \begin{pmatrix}
    u \\
    v
  \end{pmatrix},
\end{equation}
where
\begin{equation}
  JL = \begin{pmatrix}
    0 & L_{-} \\
    -L_{+} & 0
  \end{pmatrix} = \begin{pmatrix} 0 & 1 \\ -1 & 0 \end{pmatrix}\begin{pmatrix}  L_+ & 0 \\ 0 & L_- \end{pmatrix}
\end{equation}
and
\begin{align}
  L_{-} &= - \Delta + \lambda -V_-,&& V_-  =g( R ),\\
  L_{+} &= - \Delta + \lambda - V_+,&&  V_+ = g( R ) + 2 g' (R^2) R^2.
\end{align}

Alternatively, if we formulate the problem in terms of $\phi$ and
$\phi^*$,
\begin{equation}
  \partial_t \begin{pmatrix}\phi \\ \phi^* \end{pmatrix} = i \mathcal{H} \begin{pmatrix}\phi \\ \phi^* \end{pmatrix},
\end{equation}
where
\begin{align}
  \label{eqn:mathcalH}
  \mathcal{H} = \begin{pmatrix}  -\Delta + \lambda - V_1 & -V_2 \\
    V_2 &\Delta - \lambda + V_1\end{pmatrix}
\end{align}
and
\begin{equation}
  V_1  = g(R^2) + g'(R^2)R^2, \quad V_2  = g'(R^2)R^2.
\end{equation}
The potentials in the two formulations are related by $V_+ = V_1 +
V_2$ and $V_- = V_1 - V_2$.

There are many things we can immediately say about $L_{-}$, $L_{+}$,
$JL$ and $\mathcal{H}$.  For a reference on the spectral theory
involved, see Hislop-Sigal \cite{HS} or Reed-Simon \cite{RSv4}.  First
of all, both $L_{-}$ and $L_{+}$ are self-adjoint operators.  Also,
$L_{-}$ is a non-negative definite operator and its null space is
$\text{span}\{ R \}$.  Note also that the functions $\frac{\partial
  R}{\partial x_j}$ for $j = 1,2,...,d$ are in the null space of
$L_{+}$.  By comparison with the operator $\Delta + \lambda$ and using
the fact that $R$ decays exponentially, we see that the essential
spectrum of $\mathcal{H}$ is the set $(-\infty,\lambda] \cup
[\lambda,\infty)$ from Weyl's Theorem, see \cite{ES1} and \cite{RSv4}.
Equivalently, the essential spectrum of $JL$ is $(-i \infty, i
\lambda] \cup [i \lambda, i \infty)$.  Indeed, $\sigma(JL) = i
\sigma(\mathcal{H})$.  Finally, using the fact that $L_{-}$ is
non-negative definite and looking at eigenvalues $\mathcal{H}^2$, we
see
\begin{eqnarray}
  L_{-} L_{+} u = \nu^2 u.
\end{eqnarray}
However, this can be rewritten as
\begin{eqnarray*}
  T v = L_{-}^{\frac{1}{2}} L_{+} L_{-}^{\frac{1}{2}} v = \nu^2 v
\end{eqnarray*}
for $v = L_{-}^{\frac{1}{2}} u$.  Since the operator $T$ is
self-adjoint, we must have $\nu \in \R \cup i \R$.

Typically, asymptotic stability is studied with the following
assumptions made on the matrix Hamiltonian:
\begin{defn}
  \label{spec:defn1}
  A Hamiltonian, $\mathcal{H}$ is called admissible if the following
  hold:
  \begin{enumerate}
  \item There are no embedded eigenvalues in the essential spectrum.
  \item The only real eigenvalue in $[-\lambda, \lambda ]$ is $0$.
  \item The values $\pm \lambda$ are not resonances.
  \end{enumerate}
\end{defn}

\begin{defn}
  \label{spec:defn2}
  Let NLS have nonlinearity $g$.  We call $g$ admissible at $\lambda$ if there
  exists a soliton, $R_\lambda$, for NLS and the
  Hamiltonian, $\mathcal{H}$, resulting from linearization about
  $R_\lambda$ is admissible in terms of Definition \ref{spec:defn1}.
\end{defn}

\begin{rem}
For simplicity in exposition, a matrix
Hamiltonian, $\mathcal{H}$, is said to be {\it admissible} if it
satisfies several {\it spectral conditions}.  One of these properties
is a lack of embedded eigenvalues in the essential spectrum, which to establish we employ a {\it spectral property} as in \cite{FMR}.
\end{rem}

The spectral conditions of $\mathcal{H}$ from Definition \ref{spec:defn1} are generally required to
prove dispersive estimates for the evolution operator associated with
the linearized Hamiltonian equation, which in turn are required to
prove asymptotic stability results for solitons.  See
\cite{Schlag1,RodSchSof} for further discussion.  Let $P_d$ and $P_c$
be the projections onto the discrete and continuous spectrum of
$\mathcal{H}$.  

See Figure \ref{fig:spec_decomp} for a description of the spectral
decomposition for $\mathcal{H}$ resulting from linearizing about
solitons with subcritical, critical and supercritical stability
properties.  

\begin{figure}
  \centering \scalebox{0.75}{\includegraphics{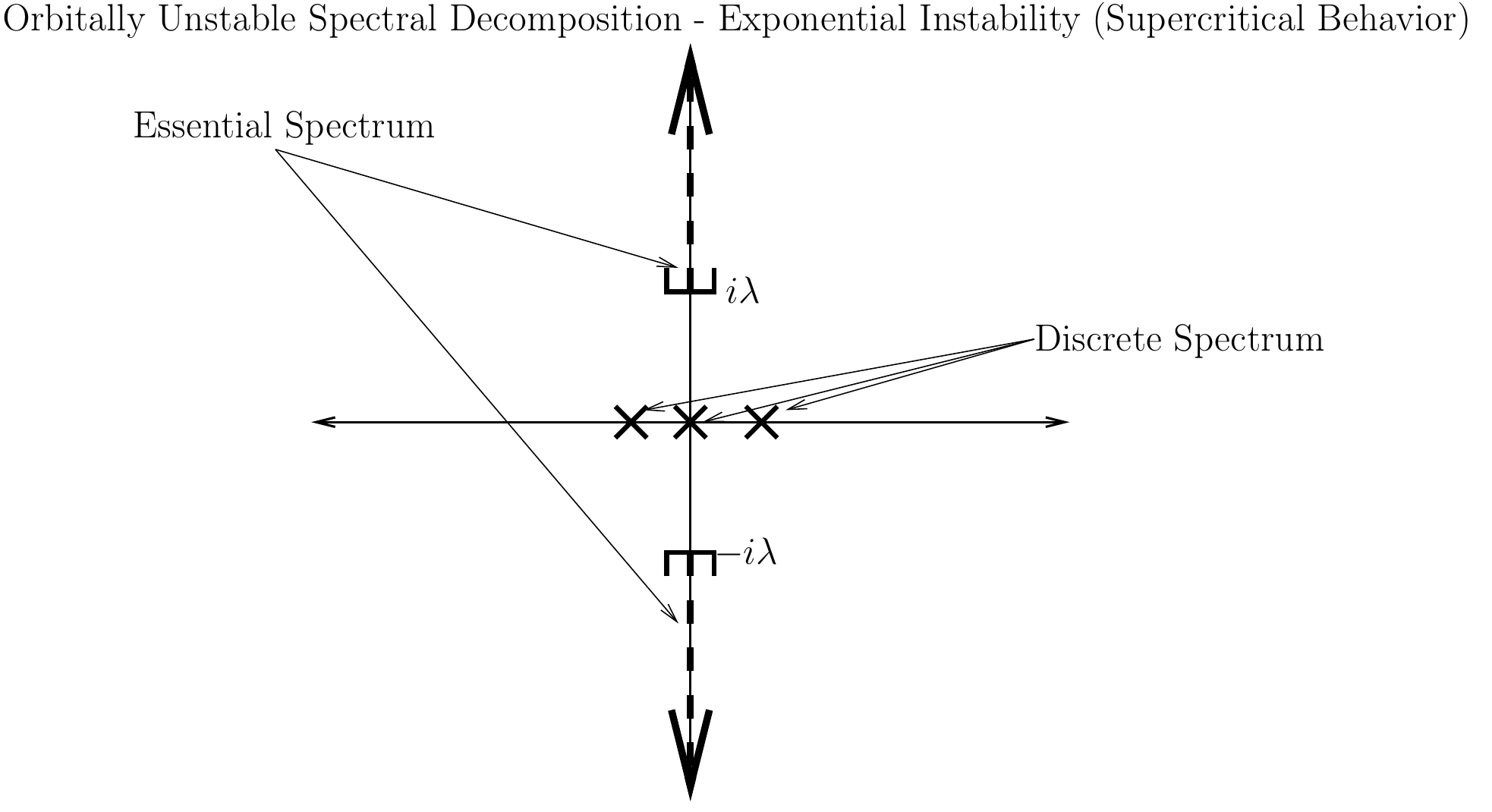}}

  \caption{Plots of the spectral decomposition for $\mathcal{H}$ when
    $R_\lambda$ is exponentially unstable (supercritical behavior).}
  \label{fig:spec_decomp}
\end{figure}

\section{Spectral Properties of the Linearized Hamiltonian}
\label{spec:lin}

We now give more detailed and formal statements on the spectral
properties of the operator under investigation.

\subsection{A Survey of Results on the Spectrum of $\calH$}
\label{sec:survey}

\subsubsection{An Analytic Result on the Spectrum of $\cal{H}$}
\label{sec:Hanalytic}

We formalize the heuristic discussion from Section \ref{sec:sollin}
with the following theorem from \cite{ES1}.  Let us write the operator as
\begin{eqnarray*}
  \mathcal{H} = \mathcal{H}_0 + V = \left[ \begin{array}{cc}
      -\Delta + \lambda & 0 \\
      0 & \Delta + \lambda
    \end{array} \right] + \left[ \begin{array}{cc}
      -V_1 & -V_2 \\
      V_2 & V_1
    \end{array} \right].
\end{eqnarray*}

In \cite{ES1}, the authors proved the following properties of the spectrum:
\begin{thm}[Erdogan-Schlag]

  Assume there are no embedded eigenvalues in the continuous spectrum
  of $\sigma ( \mathcal{H} )$.  The essential spectrum of
  $\mathcal{H}$ equals $(-\infty, -\lambda] \cup [\lambda, \infty)$.
  Moreover, $\sigma ( \mathcal{H} ) = - \sigma( \mathcal{H}) =
  \overline{\sigma( \mathcal{H} )} = \sigma( \mathcal{H}^* )$ and
  $\sigma( \mathcal{H} ) \subset \reals \cup i \reals$.  The discrete
  spectrum consists of eigenvalues $\{ z_j \}_{j=1}^N$, $0 \leq N \leq
  \infty$, of finite multiplicity.  For each $z_j \neq 0$, the
  algebraic and geometric multiplicities coincide and $\mathrm{Ran}
  (\mathcal{H} - z_j)$ is closed.  The zero eigenvalue has finite
  multiplicity.
\end{thm}

\subsubsection{Absence of Embedded Resonances}
\label{spec:embres}

This result is developed in the earlier work of Erdogan-Schlag
\cite{ES1} and Agmon \cite{Ag}.  

Define the space
\begin{eqnarray*}
  X_{\sigma} = L^{2,\sigma} \times L^{2,\sigma},
\end{eqnarray*}
where
\begin{eqnarray*}
  L^{2,\sigma} = \{ f | |\bx|^\sigma f \in L^2\}.
\end{eqnarray*}

Then, we have the following Theorem, proved in \cite{ES1}:
\begin{thm}[Erdogan-Schlag]
  Let $V_1$, $V_2$ have sufficient decay at $\infty$.  Then for any
  $\mu$ such that $|\mu| > \lambda$, $(\mathcal{H}_0 - (\mu \pm
  i0))^{-1} V:X_{-\frac{1}{2}-} \to X_{-\frac{1}{2}-}$ is a compact
  operator, and
  \begin{eqnarray*}
    I + (\mathcal{H}_0 - (\mu \pm i0))^{-1} V
  \end{eqnarray*}
  is invertible on these spaces.
\end{thm}

  The proof relies on a similar argument to a restriction theorem from
  harmonic analysis, which follows from a calculation using the specific
  structure of $R_0 (z) = (-\Delta-z)^{-1}$.  This strategy emulates
  closely that of the bootstrapping argument of Agmon, \cite{Ag}, for scalar operators.

\subsubsection{Discrete Spectrum}
\label{specnum:disc}

We wish to show that the spectrum of $\mathcal{H}$ when linearized
about the ground state has the discrete spectral
decomposition Figure \ref{fig:spec_decomp}.  

\begin{rem}
For the $3d$ cubic
nonlinearity, the structure of the discrete spectrum away from the
essential spectrum has been verified numerically in \cite{DeSc}, whose
methods we recall briefly here.
\end{rem}

In \cite{Schlag1}, using arguments derived from \cite{Per1}, it is
shown that the discrete spectrum for supercritical exponents is
determined by the discrete spectrum of $L_{\pm}$.  We present here a
slightly stronger version that works for linearizations about a
minimal mass soliton, $R=R_{\min}$, in saturated nonlinearities.
Though we will not numerically analyze any saturated nonlinearities in
the current work, we present the generalized argument saturated
nonlinearities are also of interest.

The verification of the discrete spectrum heavily relies on the
following result:
following
\begin{thm}[Schlag]
  Assume that $L_{-}$ has no discrete eigenvalues on the interval
  $(0,\lambda]$ and $\mathcal{H}$ is a Hamiltonian as in
  \eqref{eqn:mathcalH} resulting from linearizing about a minimal mass
  soliton.  Then, the only discrete eigenvalue for $\mathcal{H}$ in
  the interval $[-\lambda,\lambda]$ is $0$.
\end{thm}

\begin{rem}
  A very similar theorem appeared in \cite{Schlag1} proving the same
  result for the Hamiltonian $\mathcal{H}$ of the form
  \eqref{eqn:mathcalH} formed from $g(s) = s$ with $\lambda = 1$ and
  $\bx \in \RR^3$.  The proof follows with minimal changes and is
  adaptable to many cases, hence we include it below for completeness.
\end{rem}

\begin{proof}
  We argue by contradiction.  To this end, assume $\mathcal{H}$ has an
  eigenvalue away from $0$, say at $E$.  Let $\lambda = 1$ for
  simplicity.  Then $\mathcal{H}^2$ has an eigenvalue at some value
  $E^2 \in (0,1]$.  Hence, we have
  \begin{eqnarray*}
    L_{-} L_{+} u_E = E^2 u_E,
  \end{eqnarray*}
  for $E^2<1$.  Since $L_{-}$ is self-adjoint, we see that $u_E \perp
  \phi$.  By elliptic regularity, we have that $u_E \in H^4_{loc}$.
  Let $P$ be the projection orthogonal to $\phi$.  Let $A = P L_{+}P$.
  Using that
  \begin{equation*}
    \ker(L_{+}) = \spn\{ \partial_j \phi | 1 \leq j \leq d \}
  \end{equation*}
  and
  \begin{eqnarray}
    \label{eqn:es2}
    (\mathcal{H} - z)^{-1} & = & (\mathcal{H}_0 -z)^{-1} [I - U_1 [ I
    - U_2 J (\mathcal{H}_0 -z)^{-1} U_1]^{-1} \\ 
    && \times U_2 J (\mathcal{H}_0 -z)^{-1} U_1]^{-1}] ,  \notag
  \end{eqnarray}
  we have
  \begin{equation*}
    \ker (L_{+}) = \spn\{ \partial_\lambda R, \ \partial_j R | 1 \leq j \leq d \}
  \end{equation*}
  since
  \begin{eqnarray*}
    \langle R, \partial_\lambda R \rangle = 0.
  \end{eqnarray*}
  Take $E_0$ to be the unique negative eigenvalue for $L_{+}$.  Then,
  define
  \begin{eqnarray*}
    g(\alpha) = \langle (L_{+} - \alpha)^{-1} R, R \rangle,
  \end{eqnarray*}
  which is well-defined and differentiable on $(E_0,1)$ since $\phi$
  is orthogonal to the kernel of $L_{+}$.  We have
  \begin{eqnarray*}
    g'(\alpha) & = & \langle (L_{+} - \alpha)^{-1} R,R \rangle > 0
\end{eqnarray*}
and
\begin{eqnarray*} 
    g(0) & = & \frac12 \langle R, \partial_\lambda R \rangle = 0.
  \end{eqnarray*}
  Hence, $g(0) = 0$ is the only $0$ for $g$ in the interval $(E_0,1)$
  since
  \begin{eqnarray*}
    \lim_{\alpha \to E_0} g(\alpha) \to -\infty .
  \end{eqnarray*}

  Conversely, if $A f = \alpha f$ for some $-\infty < \alpha < 1$,
  $\alpha \neq 0$ and $f \in L^2$, then $f \perp R$ and
  \begin{eqnarray*}
    (PL_{+}P-\lambda)f = (A-\lambda) f = 0.
  \end{eqnarray*}
  Since
  \begin{eqnarray*}
    E_0 \langle f, f \rangle \leq \langle L_{+} f, f \rangle = \lambda \langle f, f \rangle,
  \end{eqnarray*}
  we know that $\lambda \geq E_0$.  If $\lambda = E_0$, then $f$ is a
  ground state of $L_{+}$ and hence not orthogonal to $R$.  However,
  $g(\lambda) = 0$, hence $\lambda = 0$.  So, $A$ has a collection of
  eigenvalues at $0$.  Define
  \begin{eqnarray*}
    \mathcal{G} = \spn\{ R, R_j, R_\lambda , u_E \}.  
  \end{eqnarray*}
  We would like to show that $\dim (\mathcal{G}) = d+3$.  Since $\phi$
  is orthogonal to all the other functions, we need only show that the
  equation
  \begin{eqnarray}
    \label{eqn:lin}
    c_1 u_E  + c_2 R_\lambda + \sum_{j =1}^d c_{j+2} R_j = 0
  \end{eqnarray}
  has only the trivial solution $c_j = 0$ for all $j$.  By applying
  $L_{+}$ to \eqref{eqn:lin}, we see
  \begin{eqnarray*}
    c_1 L_{+} u_E + c_2 R = 0.
  \end{eqnarray*}
 Taking the inner product with $u_E$, we conclude $c_1 = 0$.
  This implies that $c_2 = 0$.  As a result, $c_j = 0$ for $j =
  2,\dots,d+2$.  Now, if we can show that
  \begin{eqnarray}
    \label{eqn:max}
    \sup_{\| f \|_{L^2}, f \in \mathcal{G} } \langle A f, f \rangle < 1, 
  \end{eqnarray}
  then by the Courant minimax principle, there would be at least $d+3$
  eigenvalues less than $1$ for $A$.  However, we have shown there are
  exactly $d+2$ of them.  Note that neither the minimal eigenfunction
  for $L_{+}$ nor $\phi$ itself are eigenvalues of $A$ due to
  orthogonality arguments.  Hence, if we can prove \eqref{eqn:max}, we
  have proved the result.

 By our assumption on the spectrum of $L_{-}$, 
\begin{eqnarray*}
\langle P L_{-}^{-1} P f, f \rangle < \langle f, f \rangle
\end{eqnarray*}
 for $f \neq 0$.  Since $E \leq 1$ by assumption, we can prove the stronger result that
  \begin{eqnarray*}
   \langle A f, f \rangle \leq E^2 \langle P L_{-}^{-1} P f, f \rangle
 \end{eqnarray*}
 for all $f = a u_E + b \phi + \bold{c} \cdot \nabla R + d R_\lambda.$
  To this end, we have
  \begin{equation*}
    \begin{split}
      \langle A f, f \rangle & =  \langle L_{+} (a u_E), a u_E + \bold{c} \cdot \nabla \phi + d \phi_\lambda \rangle \\
      &\quad +  \langle L_{+} \bold{c} \cdot \nabla R, a u_E + \bold{c} \cdot \nabla R + d R_\lambda \rangle \\
      & =  E^2 \langle L_{-}^{-1} (a u_E),  a u_E + \bold{c} \cdot \nabla R + d R_\lambda \rangle \\
      &\quad +  E^2 \langle d R_\lambda, L_{-}^{-1} (a u_E) \rangle \\
      & \leq   E^2 \langle L_{-}^{-1} (a u_E),  a u_E + \bold{c} \cdot \nabla R + d R_\lambda \rangle + E^2 \langle d R_\lambda, L_{-}^{-1} (a u_E) \rangle \\
      &\quad +  E^2 \langle L_{-}^{-1} (\bold{c} \cdot \nabla R + d R_\lambda), (\bold{c} \cdot \nabla R + d R_\lambda) \rangle \\
      & \leq E^2 \langle P L_{-}^{-1} P f, f \rangle
    \end{split}
  \end{equation*}
  since $L_{-}^{-1} {u_E} \perp \nabla R$ and $L_{-}$ is positive
  definite on $\mathcal{G} \setminus R$.
\end{proof}

In order to test the discrete spectral assumptions, we
briefly recall the work of \cite{DeSc},  which requires some numerical computation.
First, let review the Birman-Schwinger method.  let $H = L_{-} - \lambda = -\Delta - V$ for $V > 0$.
Since we are looking for small, positive eigenvalues of $L_{-}$, so
take $H f = -\alpha^2 f$ for $0< \alpha < \lambda$ so we have $L_{-} f
= (\lambda - \alpha^2) f$.  Set $U = \sqrt{V}$ and $g = Uf$, then
\begin{eqnarray*}
  g = U (-\Delta + \alpha^2)^{-1} U g.
\end{eqnarray*}
In other words, $g \in L^2$ is an eigenfunction for
\begin{eqnarray*}
  K(\alpha) = U (-\Delta + \alpha^2)^{-1} U g
\end{eqnarray*}
with eigenvalue $1$ where $K > 0$, compact.  Conversely, if $g
\in L^2$ satisfies $K(\alpha) g = g$, then
\begin{eqnarray*}
  f = U^{-1} g = (-\Delta + \alpha^2)^{-1} Ug \in L^2
\end{eqnarray*}
and $Hf = -\alpha^2 f$.  The eigenvalues of $K(\alpha)$ are seen to be
strictly increasing as $\alpha \to 0$ since
\begin{eqnarray*}
  K'(\alpha) = -2 \lambda U (-\Delta + \alpha^2)^{-2} U.
\end{eqnarray*}
This implies that
\begin{eqnarray*}
  \# \left\{ \alpha : \text{Ker} (H-\alpha^2) \neq \{0\} \right\} = \# \left\{ E>1 : \text{Ker} (K(0)-E)  \neq \{0\} \right\}
\end{eqnarray*}
counted with multiplicity.

Finally, use the symmetric resolvent identity to see
\begin{eqnarray*}
  (H-z)^{-1} & = & (-\Delta -z)^{-1} + (-\Delta -z)^{-1} U \\
  && \times \left[ I - U (-\Delta -z)^{-1} U \right]^{-1} U (-\Delta -z)^{-1}.
\end{eqnarray*}
Hence, the Laurent expansion about $z=0$ does not require negative
powers for $z$ iff $I + U(-\Delta -z)^{-1} U$ is invertible at $z =0$,
i.e.
\begin{equation*}
  \ker (I - U (-\Delta)^{-1} U) = \{ 0 \}
\end{equation*}
by the Fredholm alternative since $V$ has exponential decay.  Hence,
if $H$ has no resonance or eigenvalue at $0$, then $K(0)$ will not
have an eigenvalue at $1$.  If we then count the eigenvalues
$\alpha_j$ in decreasing order for $K(0)$, then $H$ has exactly $N$
negative eigenvalues and neither an eigenvalue or resonance at $0$ iff
$\alpha_1 \geq \alpha_2 \geq \dots \geq \alpha_N > 1$ and
$\alpha_{N+1} < 1$.  Hence, we can study numerically study $K$ for a
soliton of the saturated nonlinear Schr\"odinger equation.  To do so,
we must accurately find a soliton, then use it as potential in the
truncation and discretization scheme presented in \cite{DeSc}, where
the gap condition is verified for the Hamiltonian resulting from
linearization about the $3d$ cubic ground state soliton.
 
\subsection{Generalized Kernel}
\label{s:kerg}

Let us review the generalized kernel of a Hamiltonian resulting
from linearizing about a soliton.  Following \cite{W1},
we see by direct calculation that the vectors
\begin{eqnarray*}
  \left[ \begin{array}{c}
      0 \\
      R
    \end{array} \right], \left[ \begin{array}{c}
      R_j \\
      0
    \end{array} \right]
\end{eqnarray*}
for all $j = 1, \dots, d$ are contained in $\ker(JL)$.  Now, as
$Q(R_\lambda)$ is differentiable with respect to $\lambda$, we have by
a simple calculation that $L_{+} \partial_\lambda R = - R$ and $L_{-}
(x \phi) = -2 \nabla R$.  Hence, the vectors
\begin{eqnarray*}
  \left[ \begin{array}{c}
      0 \\
      x_j R
    \end{array} \right], \left[ \begin{array}{c}
      (\partial_\lambda R)_{\lambda_0} \\
      0
    \end{array} \right]
\end{eqnarray*}
in the generalized null space of order $2$.  Notice that so far we
have constructed at $2d+2$ dimensional null space.  Since we know the
null spaces of $L_{-}$ and $L_{+}$ exactly, these are unique.  For
power nonlinearities, $g(s) = s^{\sigma}$, 
\begin{equation}
(\partial_\lambda R)_{\lambda_0=1} = \frac12 (\frac{1}{\sigma} R + \bx \cdot
\grad R).
\end{equation}
We use this explicit form in our calculations. 

As a result, we have the following
\begin{thm}
  Let $g$ be the $L^2$ supercritical monomial nonlinearity.  There exists a
  $2d+2$ dimensional null space for $\mathcal{H}$, the matrix
  Hamiltonian resulting from linearization about the ground state
  soliton ($\lambda = 1$), consisting of the span of the vectors
  \begin{eqnarray*}
    \left\{ \left[ \begin{array}{c}
          0 \\
          R
        \end{array} \right], \left[ \begin{array}{c}
          R_j \\
          0
        \end{array} \right],
      \left[ \begin{array}{c}
          0 \\
          x_j R
        \end{array} \right], \left[ \begin{array}{c}
          (\partial_\lambda R)_{\lambda_0} \\
          0
        \end{array} \right] \right\}.
  \end{eqnarray*}
\end{thm}

\begin{proof}
The generalized null space of the adjoint can be found by reversing
the location of the non-zero elements in the above vectors.

Suppose that there exists a generalized eigenspace for the eigenvalue
$E \neq 0$.  Then, there exists $\chi \neq 0$ and $\psi \neq 0$ such
that $(\mathcal{H} - E) \psi = \chi$ and $(\mathcal{H} - E) \chi = 0$.
Then, note that
\begin{eqnarray*}
  (\mathcal{H}^2 - E^2) \psi & = & (A + E) \chi = 2 E \chi, \\
  (\mathcal{H}^2 - E^2) \chi & = & 0.
\end{eqnarray*}
Hence, $\mathcal{H}^2$ has a generalized eigenspace at $E^2$.  As a
result, we see that $T = L_{+} L_{-}$ has a generalized eigenspace at
$E^2$.  Let $T \chi = E^2 \chi$ and $(T-E^2) \psi = c \chi$, for some
$c \neq 0$.  Hence,
\begin{eqnarray*}
  (L_{-}^{\frac{1}{2}} L_{+} L_{-}^{\frac{1}{2}} - E^2) L_{-}^\frac{1}{2} \psi_1 & = & c \chi_1 , \\
  (L_{-}^{\frac{1}{2}} L_{+} L_{-}^{\frac{1}{2}} - E^2)^2
  L_{-}^\frac{1}{2} \psi_1 & = & c L_{-}^{\frac{1}{2}} ( L_{+} L_{-} -
  E^2) \chi_1 = 0, 
\end{eqnarray*}
where given $P^c_R = I - P_R$, we have $\chi_1 = P^c_R \chi \neq 0$
since $T \chi = E^2 \chi$ and $\psi_1 = P^c_R \psi \neq 0$ since
$(T-E^2) \psi = c \chi$.  However, this means that the self-adjoint
operator $L_{-}^{\frac{1}{2}} L_{+} L_{-}^{\frac{1}{2}}$ has a
generalized eigenvalue, which is impossible by an orthogonality
argument.

Since we have assumed there are no eigenvalues at the endpoints of the
continuous spectrum, there can be no accumulation and the number of
discrete eigenvalues is finite.
\end{proof}

\subsection{Natural Orthogonality Conditions}
\label{sec:verification}

As noted in Definition \ref{def:specprop}, even if the spectral
property holds, it only implies the absence of embedded eigenvalues on
a subspace of $\calU \subset L^2\times L^2$.  That
we must limit ourselves to a subspace will become clear
in Section \ref{sec:index_computations}, where we demonstrate that the operators $\calL_\pm$
have negative eigenvalues. 

This subspace will be defined as the
orthogonal complement to the span of a set of vectors.  If this
collection of vectors is not chosen properly, we may find that the
spectral property holds though the operator still has embedded
eigenvalues.  Thus the constraints on the set of vectors whose
orthogonal complement will define $\calU$ are:
\begin{enumerate}
\item They must be orthogonal to any embedded eigenvalues,
\item Orthogonality with respect to them should induce positivity of
  $\calL$ on $\calU$.
\end{enumerate}

A way of meeting both of these requirements is to use the discrete spectrum of the
adjoint matrix Hamiltonian, $\mathcal{H}^*$.  To that end, we rely on the following simple results.
\begin{lem}
  If $(\lambda, \vec{u})$ is an eigenvalue, eigenvector pair for $JL$
  and $(\sigma, \vec{v})$ is an eigenvalue, eigenvector pair for
  $(JL)^\ast$, then
  \[
  (\lambda - \sigma^\ast) \inner{\vec{u}}{\vec{v}} = 0 .
  \]
  Thus, if $\lambda - \sigma^\ast \neq 0$, the states are orthogonal.
\end{lem}

\begin{cor}
  An eigenstate of $JL$ associated with an imaginary (possibly
  embedded) eigenvalue, $i \tau \neq 0$, is orthogonal to
  $\kerg((JL)^\ast)$.
\end{cor}

\begin{cor}
  Let $(i \tau \neq 0, \vec{\psi})$ and $(\lambda >0, \vec{\phi})$ be
  eigenvalue, eigenvector pairs of $JL$.  Then
  \begin{align*}
    \inner{\psi_1}{\phi_2} & = 0,\\
    \inner{\psi_2}{\phi_1} & = 0.
  \end{align*}
\end{cor}
\begin{proof}
  By the Hamiltonian symmetry of the problem, $-\lambda$ $(\phi_2,
  \phi_1)^T$ and $\lambda$, $(-\phi_2,-\phi_1)^T$ are eigenvalue pairs
  of the adjoint, $(JL)^\ast$.  Therefore,
  \begin{align*}
    \inner{\psi_1}{\phi_2} - \inner{\psi_2}{\phi_1} & = 0, \\
    -\inner{\psi_1}{\phi_2} - \inner{\psi_2}{\phi_1} & = 0.
  \end{align*}
  Adding and subtracting these equations gives the result.
\end{proof}
These trivial observations motivate using the known spectrum of the
adjoint system in constructing the orthogonal subspace.  For the 3D
cubic problem, we can thus use make use of eigenstates coming from the
origin and the two off axis, real eigenvalues.

\section{Bilinear Forms and the Spectral Property}
\label{sec:normal}

We show here how the Spectral
Property \ref{def:specprop} is a condition sufficient
for showing there are no embedded eigenvalues, thus proving Theorems
\ref{thm:coercive} and \ref{thm:speccond}.

For general nonlinearities, we assume the operator resulting from
linearizing about a soliton $R$ has as discrete spectrum that is one
of the following:
\begin{itemize}
\item $(i)$ a $2d+2$ dimensional null space given by $R, \grad
  R, \partial_\lambda R, \bx R$ plus $2$ eigenfunctions with symmetric
  discrete eigenvalues $\lambda_0$, $-\lambda_0$ such that
  $\lambda_0^2 \in \reals$,
\item $(ii)$ a $2d+4$ dimensional null space given by $R, \grad
  R, \partial_\lambda R, \bx R, \alpha, \beta$.
\end{itemize}

In the following subsections, we establish the following:
\begin{thm}
  The generalized spectral property holds for the 3d cubic problem for
  $( f, g)^T  \in \calU \subset L^2 \times L^2 $ specified by the following
  orthogonality conditions:
 \begin{equation*}
    \inner{f}{R} =0,\quad \inner{g}{R + \bx \cdot \nabla R}=0,
    \quad\inner{f}{\phi_2} =0,\quad\inner{f}{ x_j R}
    =0\quad\text{for $j=1,\ldots d$},    \end{equation*}
  where $\vec{\phi} = ( \phi_1, \phi_2)^T$ is the
  eigenstate associated with the positive eigenvalue $\sigma>0$.
\end{thm}
This subspace is motivated by the observations in Section
\ref{sec:verification}, as all of the elements we are orthogonal to
arise from the spectrum of the adjoint problem.

\begin{rem}
Note that henceforward we assume here the unstable eigenfunction for the
$3d$ cubic problem is radial.  This claim is
substantiated by direct integration in our numerical results section,
as well as the fact that the spectral decomposition for the critical
problem remains valid under the assumption of radial symmetry, hence
the multiplicity of radial eigenfunctions must be at least $4$.  Since
$R$, $R_\lambda$ are the only radial components of the kernel, the
unstable eigenmode must also be radial.
\end{rem}

Following \cite{FMR}, we computationally verify the spectral property
in the following steps.  First, the bilinear form is decomposed by
spherical harmonics into
\begin{equation}
\begin{split}
\calB(\mathbf{z},\mathbf{z}) & = \calB_+(f,f) + \calB_-(g,g) \\
&=\sum_{k=0}^\infty \calB_+^{(k)}(f^{(k)},f^{(k)}) + \sum_{k=0}^\infty
\calB_-^{(k)}(g^{(k)},g^{(k)}) ,
\end{split}
\end{equation}
where $\mathbf{z} = (f,g)^T$, and $f^{(k)}$ and $g^{(k)}$ are the
components of $f$ and $g$ in the $k$-th spherical harmonic.
We then identify the dimension of the subspace of negative eigenvalues for
$\calL_\pm^{(k)}$.  Though at first this would appear to require an
infinite number of computations, a monotonicity property of
these operators with respect to $k$ limits this to a finite number of
harmonics.  We then show that our orthogonality conditions are
sufficient to point us away from the negative directions, allowing us
to prove our result.

\subsection{The Index of an Operator}

For a bilinear form $B$ on a vector space $V$, the index of $B$ with
respect to $V$ is given by
\begin{equation*}
\begin{split}
  \text{ind}_V (B)  \equiv \max \{ k \in \mathbb{N} \mid&  \text{there exists a subspace $P$ of codimension $k$} \\
  & \text{such that $B|_{P}$ is positive} \} .
\end{split}
\end{equation*}
Our results rely on the following generalization of Theorem XIII.8 of
\cite{RSv4}, which is in turn an extension of the Sturm
Oscillation Theorem (Section XIII.7 of \cite{RSv4}):
\begin{thm}
  \label{thm:rs_idx}
  Let $U^{(k)}$ be the solution to
  \begin{equation*}
    L^{(k)} U^{(k)} = - \frac{d^2}{dr^2}U^{(k)} - \frac{d-1}{r} \ddr U^{(k)} + V(r) U^{(k)} + \frac{k(k +
      d-1)}{r^2} U^{(k)}= 0
  \end{equation*}
  with initial conditions given by the limits
  \begin{equation*}
    \lim_{r\to 0} \frac{U^{(k)}(r)}{r^k}  = 1, \quad\lim_{r\to 0} \ddr \frac{U^{(k)}(r)}{r^k}  = 0,
  \end{equation*}
  where $V$ is sufficiently smooth and decaying at $\infty$.  Then,
  the number $N(U^{(k)})$ of zeros of $U^{(k)}$ is finite and
  \begin{align*}
    \ind_{H^1_\rad} (B^{(0)}) &= N(U^{(0)}), \\
    \ind_{H^1_{\rad+}} (B^{(k)}) &= N(U^{(k)}),\quad k \geq 1,
  \end{align*}
  where $B^{(k)}$ is the bilinear form associated to $L^{(k)}$.
\end{thm}
The space $H^1_\rad$ is the set of radially symmetric $H^1(\R^d)$
functions.  The space $H^1_{\rad+}$ is the subset of $H^1_\rad$ for
which
\[
\int \frac{\abs{f}^2}{\abs{\bx}^2} d\bx < \infty.
\]
We will omit the subscript notation in our subsequent 
index computations. It will be $H^1_\rad$ for $k=0$ and $H^1_{\rad+}$ for $k\geq 1$.

If one wishes to remove the limits from the statement of the initial
conditions, let $U^{(k)}(r)= r^k \widetilde{U}^{(k)}(r)$.  Then the operator becomes
\[
\widetilde{L}^{(k)} = - \frac{d^2}{dr^2} - \frac{d-1 + 2k}{r} \frac{d}{dr} + V
\]
and the initial conditions become $\widetilde{U}^{(k)}(0) = 1$ and $\ddr
\widetilde{U}^{(k)}(0) = 0$.  Indeed, we use precisely this change of variables
when making our numerical computations; see Appendix
\ref{s:numerical_sing}.  The proof can be adapted from the proof of
Theorem XIII.8 in \cite{RSv4}.
\begin{cor}
\label{c:idx_mono}
The index is monotonic with respect to $k$,
\[
\ind(B^{(k+1)})\leq \ind(B^{(k)}) .
\]
\end{cor}
This has the useful consequence that once we find an $k$ for which
$\ind(B^{(k)}) = 0$, we can immediately conclude that $B^{(k')}\geq 0$
for all $k' \geq k$.

Once we have computed the number of directions of each
$\calL_\pm^{(k)}$ that prevents it from being positive, we can check
that we have a sufficient number of orthogonal conditions to point us
into the positive subspace.

\subsection{Numerical Estimates of the Index}
\label{sec:index_computations}
To compute the indexes of the operators, we proceed as follows.

In dimension three, we solve the initial value problems
\begin{align}
  \calL_+^{(0)} U^{(0)} &= 0, \quad U^{(0)} (0)= 1, \quad \frac{d}{dr}U^{(0)} (0)=0,  \\
  \calL_-^{(0)} Z^{(0)} &= 0, \quad Z^{(0)} (0)= 1, \quad
  \frac{d}{dr}Z^{(0)} (0)=0
\end{align}
for radially symmetric functions $U^{(0)}$ and $Z^{(0)}$.  For higher
harmonics, $k>0$, we solve the initial value problems
\begin{align}
  \calL_+^{(k)} U^{(k)} &= 0, \quad U^{(k)} (0)= 0, \quad \lim_{r\to
    0}
  r^{-k} U^{(k)} (r)=1  ,\\
  \calL_-^{(k)} Z^{(k)} &= 0, \quad Z^{(k)} (0)= 0, \quad \lim_{r\to
    0} r^{-k} Z^{(k)} (r)=1
\end{align}
for radially symmetric functions $U^{(k)}$ and $Z^{(k)}$.

\begin{prop}
  \label{prop:index_computations}
  The indexes of 3d Cubic NLS are:
  \begin{gather*}
    \ind \calL_+^{(0)} = 1, \quad \ind \calL_+^{(1)} = 1, \quad \ind
    \calL_+^{(2)} = 0,\\
    \ind \calL_-^{(0)} = 1,\quad \ind \calL_-^{(1)} = 0.
  \end{gather*}
\end{prop}
Once this proposition is established, Corollary \ref{c:idx_mono}
immediately gives us
\begin{cor}
\label{c:idx_mono_cubic}
For 3d Cubic NLS,
\begin{align*}
\ind \calB_+^{(k)} =0&\quad\textrm{for $k>2$},\\
\ind \calB_-^{(k)} =0&\quad\textrm{for $k>1$},
\end{align*}
where $\calB_{\pm}^{(k)}$ is the bilinear form associated with $\calL_{\pm}^{(k)}$.
\end{cor}

Using the method discussed in Section \ref{sec:numerics}, we compute
$U^{(\alpha)}$ and $Z^{(\alpha)}$ for each problem, $\alpha =0,1,2$.
The profiles appear in Figures
\ref{fig:3dcubic_index_k0},\ref{fig:3dcubic_index_k1},\ref{fig:3dcubic_index_k2}.
All were computed with a tolerance setting $10^{-13}$.

As a consistency check on the numerics, we note that asymptotically,
the potential vanishes, and
\begin{align}
  \calL_\pm^{(k)} &\approx -\frac{d^2}{dr^2}
  -\frac{d-1}{r}U^{(k)}+\frac{k(k+d-2)}{r^2}.
\end{align}
In the region where $r \gg 1$, and the equations are essentially free
and the solutions must behave as:
\begin{subequations}
  \begin{equation}
    \label{eq:Uk_asympt}
    U^{(k)}(r) \approx C_0^{(k)} r^k + C_1^{(k)} r^{2 - d - k}
  \end{equation}
  and
  \begin{equation}
    \label{eq:Zk_asympt}
    Z^{(k)}(r) \approx D_0^{(k)} r^k + D_1^{(k)} r^{2 - d - k}.
  \end{equation}
\end{subequations}
Estimating these constants from the numerics, we see that they have
the ``correct'' signs.  For instance in Figure
\ref{fig:3dcubic_index_k0} (a), $U^{(0)}$ clearly has one zero
crossing.  For sufficiently large $r$, the function appears to be
increasing past a local minimum.  However, since the constants appear
to have stabilized, we contend we have entered the free region; the
signs and magnitudes of the constants thus forbid another zero.

\begin{figure}
  \centering \subfigure[$k=0$ harmonic]{
    \includegraphics[width=2.5in]{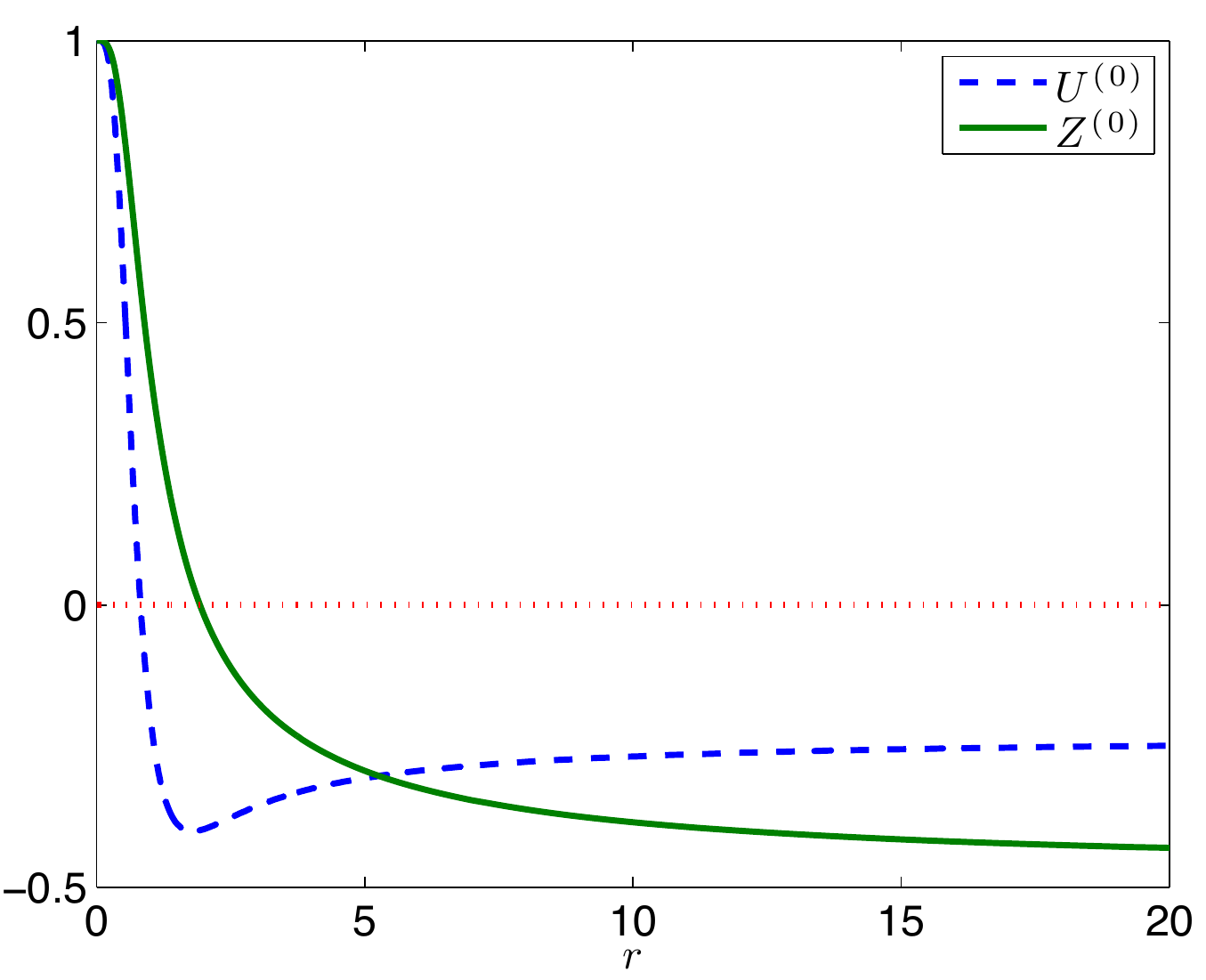}
  } \subfigure[$k=0$ harmonic asymptotics]{
    \includegraphics[width=2.5in]{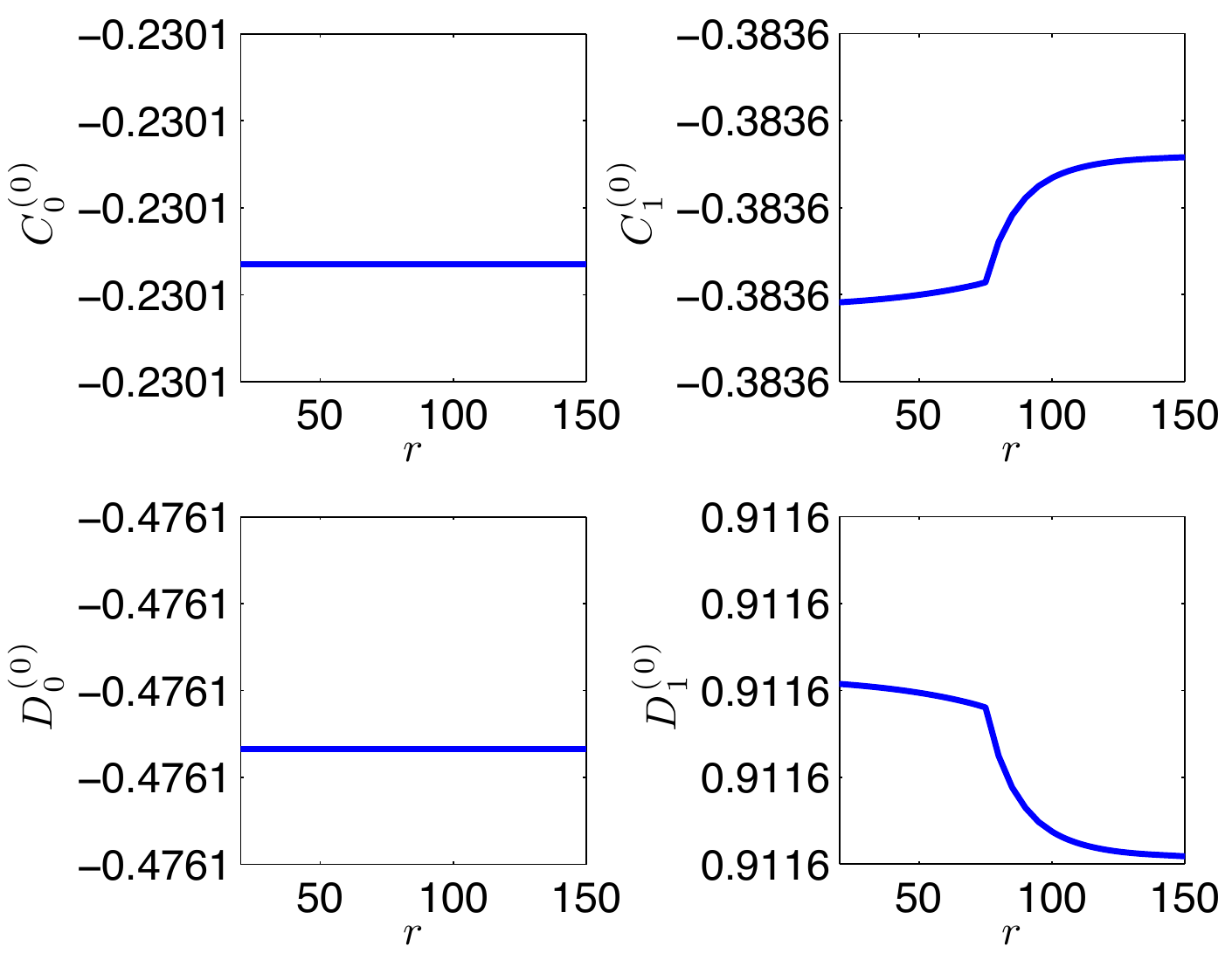}
  }
  \caption{Index computations for 3d Cubic NLS.  The number of zero
    crossings (other than $r=0$), determines the codimension of the
    subspace on which the operator $\calL_\pm^{(k)}$ is positive.}
  \label{fig:3dcubic_index_k0}
\end{figure}

\begin{figure}
  \centering \subfigure[$k=1$ harmonic]{
    \includegraphics[width=2.5in]{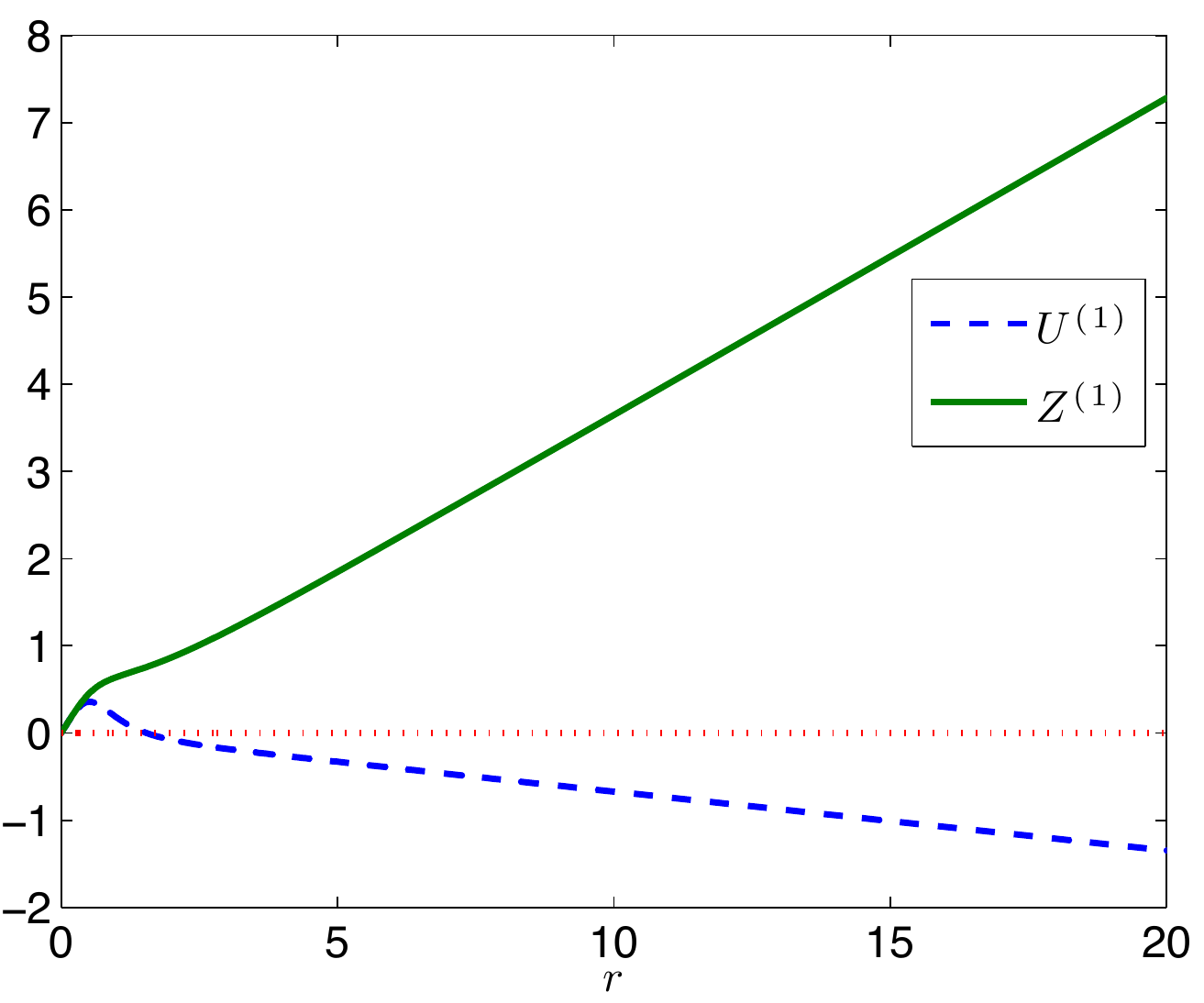}
  } \subfigure[$k=1$ harmonic asymptotics]{
    \includegraphics[width=2.5in]{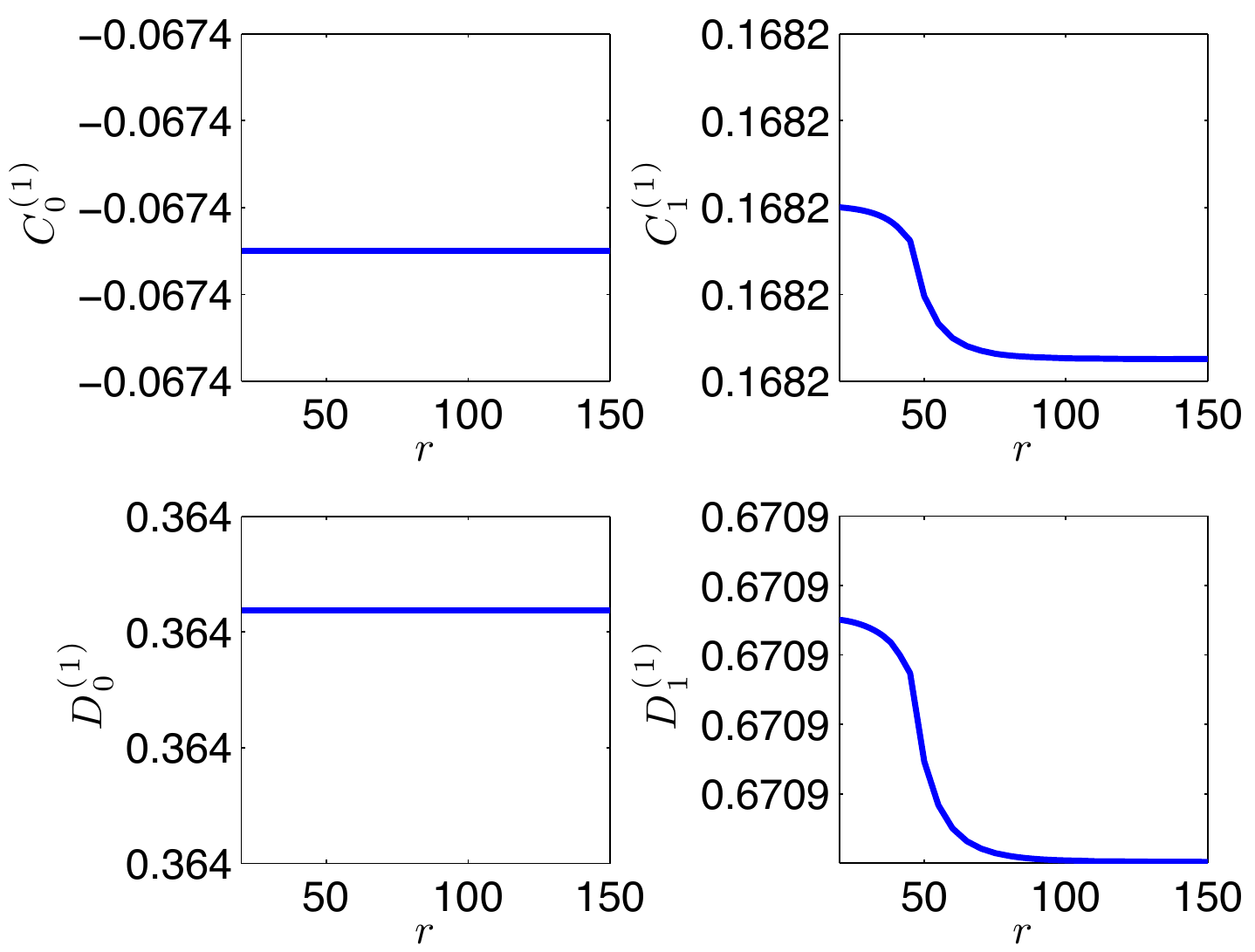}
  }
  \caption{Index computations for 3d Cubic NLS.  The number of zero
    crossings (other than $r=0$), determines the codimension of the
    subspace on which the operator $\calL_\pm^{(k)}$ is positive.}
  \label{fig:3dcubic_index_k1}
\end{figure}

\begin{figure}
  \centering \subfigure[$k=2$ harmonic]{
    \includegraphics[width=2.5in]{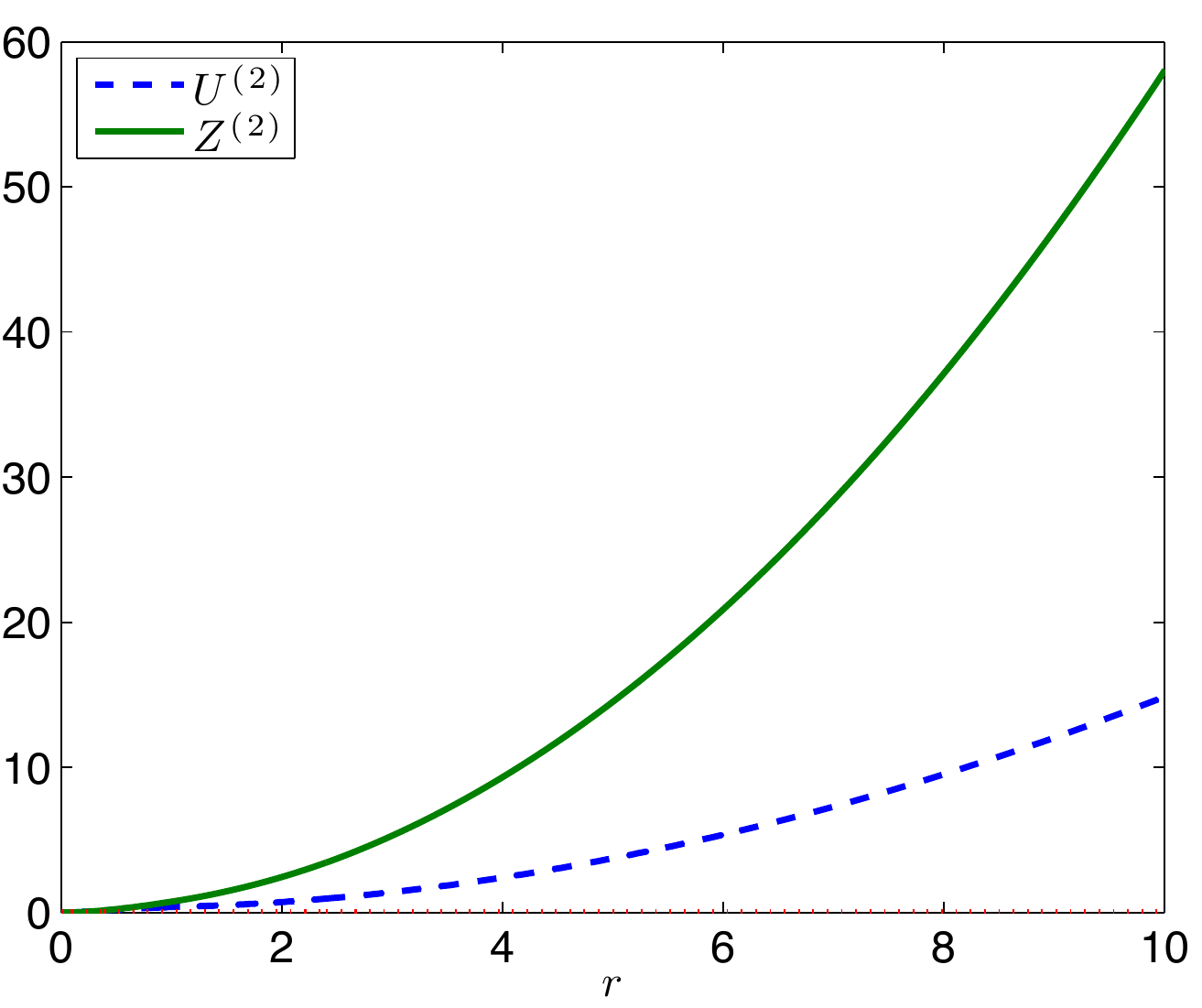}
  } \subfigure[$k=2$ harmonic asymptotics]{
    \includegraphics[width=2.5in]{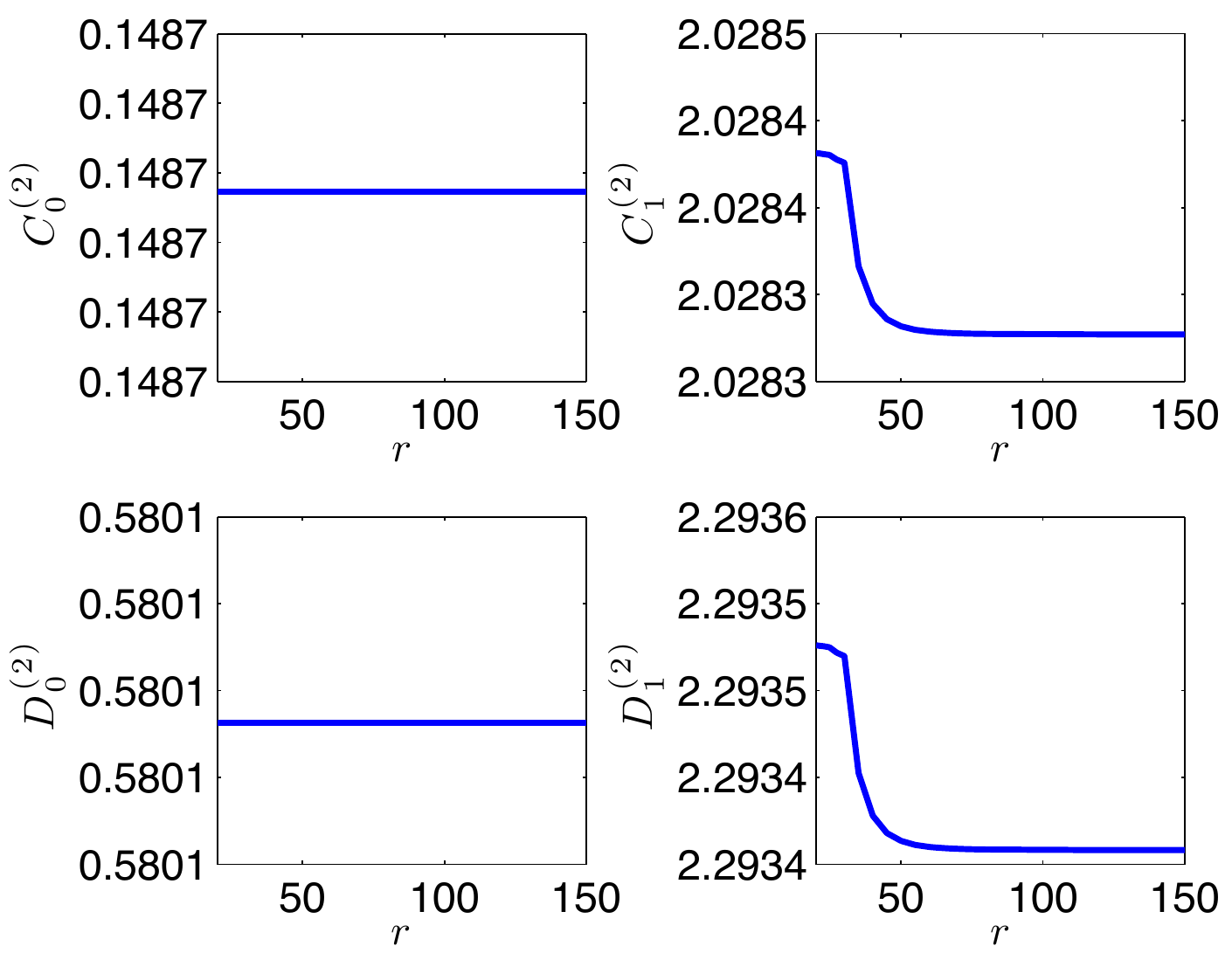}
  }
  \caption{Index computations for 3d Cubic NLS.  The number of zero
    crossings (other than $r=0$), determines the codimension of the
    subspace on which the operator $\calL_\pm^{(k)}$ is positive.}
  \label{fig:3dcubic_index_k2}
\end{figure}

\begin{prop}
  \label{prop:idx_perturbation}
  For the operators in Proposition \ref{prop:index_computations} and
  Corollary \ref{c:idx_mono_cubic} there exists a universal
  $\delta_0>0$, sufficiently small, such that for the perturbed
  operators
  \[
  \overline{\calL}_\pm^{(k)} = {\calL}_\pm^{(k)} - \delta_0 e^{-\abs{\bx}^2}
  \]
  the associated bilinear forms have the property that 
  \[
  \ind(\overline{\calB}_\pm^{(k)}) = \ind(\calB_\pm^{(k)}).
  \]
\end{prop}
\begin{proof}
  The proof here follows obviously from definition of the index of
  $B$, namely the positivity of the quadratic form $B$ on the
  subspace for individual operators.  Let
  $\delta_0$ be a sufficiently small value that it holds for 
  $\overline{\calL}_+^{(k)}$  for $k = 0, 1, 2$ and
  $\overline{\calL}_-^{(k)}$ for $k = 0,1$.  This $\delta_0$ 
  now holds for all higher values of $k$, again by a monotonicity
  argument.  Indeed, for $k > 2$,
\[
\overline{\calB}_+^{(k)}(f,f) = \overline{\calB}_+^{(2)}(f,f) + \int
\frac{k^2 + 2k - 8 }{\abs{\bx}^2}\abs{f}^2 d\bx  \geq 0.
\]
Because the form remains positive, this confirms that its index of zero is unperturbed.
\end{proof}

\subsection{Invertibility of Operators}

In conjunction with the results on the indexes of operators, we need
to compute a number of inner products of the form $\inner{\calL
  u}{u}$, where $\calL$ is one of our operators and $u$ solves $\calL
u = f$.  These are computed numerically, but we can rigorously justify
the existence and unqiuess of these solutions, $u$, for the problems
under consideration.

\begin{prop}[Numerically Verified for 3d Problems]
  \label{prop:eu_bvp_3d}

  Let $f$ be a smooth, radially symmetric, localized function
  satisfying the bound $\abs{f(r)} \leq C e^{-\kappa r}$ for some
  positive constants $C$ and $\kappa$.  There exists a unique radially
  symmetric solution
\[
(1+r^{k+1})u \in L^\infty([0,\infty))\cap C^2([0,\infty))
\]  to
  \begin{equation}
    \calL u = f,
  \end{equation}
  where $\calL = \calL^{(k)}_\pm$ for one of the 3d problems.
\end{prop}

\begin{proof}
  This is Proposition 2 and 4 of \cite{FMR}, along with our
  computations of the indexes in Proposition
  \ref{prop:index_computations}.  See Appendix \ref{sec:inverse_proof}
  for a proof in 1d.
\end{proof}
\begin{cor}
  The solutions in Proposition \ref{prop:eu_bvp_3d} are smooth and
  decay $\propto r^{-1 - k}$ as $r\to \infty$.
\end{cor}

\subsection{Estimates of Inner Products}
In order to prove the spectral property for each of these NLS
equations, we need to approximate the bilinear forms associated with
$\calL_\pm^{(\alpha)}$ on certain functions.  These particular
functions are, generically, of the form $\calL u = f$, where $f$ is
from one of the orthogonality conditions.

\begin{prop}[Numerical approximation of inner products]
  \label{prop:ip_estimates_3d}
   For the $3d$ cubic problem, let $U_1^{(0}$, $U_2^{(0)}$,
    $U_1^{(1)}$, and $Z_1^{(0)}$ solve
    \begin{align}
      \calL_+^{(0)} U_1^{(0)} &= R, \quad (1+r)U_1^{(0)}\in L^\infty,\\
      \calL_+^{(0)} U_2^{(0)} &=\phi_2, \quad (1+r)U_2^{(0)}\in L^\infty,\\
      \calL_+^{(1)} U_1^{(1)} &= r R, \quad (1+r^2)U_1^{(0)}\in L^\infty,\\
      \calL_-^{(0)} Z_1^{(0)} &= R + r R', \quad (1+r)Z_1^{(0)}\in
      L^\infty,
    \end{align}
    where $\vec{\phi}$ is the eigenstate associated with the positive
    real eigenvalue of $JL$.  Then,
    \begin{align}
      K_1^{(0)}&\equiv \inner{\calL_+^{(0)}
        U_1^{(0)}}{U_1^{(0)}}= 1.04846,\\ 
      K_2^{(0)}&\equiv \inner{\calL_+^{(0)} U_2^{(0)}}{U_2^{(0)}}=
      0.00215981,\\ 
      K_3^{(0)}&\equiv\inner{\calL_+^{(0)}
        U_1^{(0)}}{U_2^{(0)}}=-0.116369, \\ 
      K_1^{(1)}&\equiv\inner{\calL_+^{(1)} U_1^{(1)}}{U_1^{(1)}}=-0.581854,\\ 
      J_1^{(0)}&\equiv\inner{\calL_-^{(0)}
        Z_1^{(0)}}{Z_1^{(0)}}= -0.662038.
    \end{align}
\end{prop}
\begin{proof}
  These result follows from direct computation.
\end{proof}

Finally, we state the following
\begin{prop}
  \label{p:ip_est_3d_peturb}
  For each case in Proposition \ref{prop:ip_estimates_3d}, there
  exists a $\delta_0$ sufficiently small such that inner products
  associated with $\overline{\calL}$ can be made arbitrarily close to
  $\calL$.  These values will be denoted with overlines.
\end{prop}
\begin{proof}
This follows immediately from the invertibility of the operator and continuity.
\end{proof}

\subsection{Proof of the Spectral Property}
\label{sec:spec_prop}
We are now ready to prove the spectral property.  We prove positivity
of $\overline{\calB}_+^{(0)}$, the other cases are similar.  Our proof
closely follows Step 1 and Step 3 of Section 2.4 of \cite{FMR}.

Since $K_1^{(0)}$ and $K_2^{(0)}>0$, orthogonality to $R$ and $\phi_2$
will not give positivity.  However, if $f$ is orthogonal to both of
these, then it is also orthogonal to
\[
q = R - \frac{K^{(0)}_3}{K^{(0)}_2} \phi_2
\]
and
\[
\begin{split}
  \inner{\calL_+^{(0)} q}{q}&= K_1^{(0)} - 2
  \frac{K^{(0)}_3}{K^{(0)}_2} K_3^{(0)}
  + \paren{\frac{K^{(0)}_3}{K^{(0)}_2}}^2 K^{(0)}_2\\ &=
  -\frac{1}{K^{(0)}_2}((K^{(0)}_3)^2 - K^{(0)}_1 K^{(0)}_2)\\ & =
  -5.22138. 
\end{split}
\]
By Proposition \ref{p:ip_est_3d_peturb}, we can take $\delta_0$
sufficiently small such that
\[
-\frac{1}{\overline{K}^{(0)}_2}\paren{(\overline{K}^{(0)}_3)^2 -
  \overline{K}^{(0)}_1 \overline{K}^{(0)}_2}<0 .
\]
We proceed with this value of $\delta_0$.

Let $\overline{Q}$ solve 
\[
\overline{\calL}_+^{(0)} \overline{Q} = q.
\]
Obviously, 
\[
Q = \overline{U}_1^{(0)} -
\frac{K^{(0)}_3}{K^{(0)}_2} \overline{U}_2^{(0)}
\]
and
\[
\overline{\calB}_+^{(0)}(\overline{Q} ,\overline{Q} ) < 0.
\]

For a moment, suppose $\overline{Q} \in H^1_\rad$; it is not since it
decays too slowly to be in $L^2$.  We could then imagine decomposing
$H^1_\rad$ into $\spn\{\overline{Q}\}$ and its orthogonal complement,
where the orthgonalization is done with respect to the
$\overline{\calB}_+^{(0)}$ quadratic form.  Since
$\overline{\calB}_+^{(0)}(\overline{Q},\overline{Q})<0$, the form is
{\it non-degenerate} and this decomposition is well defined.  Since
$\ind_{H^1_\rad} \overline{\calB}_+^{(0)} = 1$,
$\overline{\calB}_+^{(0)}\geq 0$ on $\spn\{\overline{Q}_A\}^\perp$.
To prove this claim, we argue by contradiction.  Suppose there were an
element, $Z \in \spn\{\overline{Q}_A\}^\perp$ for which
$\overline{\calB}_+^{(0)}(Z,Z)<0$.  Then, because of our
decomposition, $\overline{\calB}_+^{(0)}<0$ on $\spn\{Z,
\overline{Q}_A\}$, a space of dimension 2.  This contradicts our index
calculation, proving the claim.

Continuing, if $u\in H^1_\rad$, $u\perp q$ (with respect
to $L^2$), then using the
hypothetical orthogonal decomposition,
\[
u = c \overline{Q} + u^\perp.
\]
If $c =0$, then $u$ lies in a subspace of $H^1_\rad$ on which
$\overline{\calB}_+^{(0)}\geq 0$, giving the desired positivity.
Indeed, the orthogonality condition, $u \perp q$, is sufficient to
ensure $u$ is orthogonal to $\overline{Q} $ with respect to the
$\overline{\calB}_+^{(0)}$ quadratic form. Taking the inner product of
$u$ with $q$,
\begin{equation*}
\begin{split}
0 &= c \inner{\overline{Q}}{ q} +
\inner{u^\perp}{q} = c
\overline{\calB}_+^{(0)}(\overline{Q},\overline{Q}) +
\inner{u^\perp}{\overline{\calL}_+^{(0)}\overline{Q}}\\
&=  c
\overline{\calB}_+^{(0)}(\overline{Q},\overline{Q}) +
\overline{\calB}_+^{(0)}(u^\perp,\overline{Q})=   c
\overline{\calB}_+^{(0)}(\overline{Q},\overline{Q}) + 0.
\end{split}
\end{equation*}
Since 
$\overline{\calB}_+^{(0)}(\overline{Q},\overline{Q})\neq 0$,
we have $c = 0$. 

Unfortunately, the above argument does not work as stated because
$\overline{Q}$ is not in $L^2$!  To get positivity of
$\overline{\calB}_+^{(0)}$, we regularize the problem and follow the
above scheme.  First, we introduce the smooth cutoff function
$\chi_A(r)=\chi(r/A)$, defined such that
\[
\chi(r)=\begin{cases}
1 & r < 1\\
0 & r\geq 2
\end{cases}
\]
and the norm
\[
\norm{f}_{\pm}^2 = \norm{\grad f}_{L^2}^2 + \int \abs{\calV_\pm} \abs{f}^2.
\]
Let $\overline{Q}_A(r) = \overline{Q}(r)\chi_A(r)$. Next, we
observe that 
\begin{equation}
\lim_{A\to +\infty} \norm{\overline{Q}_A - \overline{Q}}_+ +
\abs{\overline{\calB}_+^{(0)}(\overline{Q} ,\overline{Q} ) -
\overline{\calB}_+^{(0)}(\overline{Q}_A ,\overline{Q}_A )}=0.
\end{equation}

Since $\overline{\calB}_+^{(0)}(\overline{Q} ,\overline{Q} ) < 0
$, for sufficiently large $A$,
$\overline{\calB}_+^{(0)}(\overline{Q}_A ,\overline{Q}_A )
<0$ too.  Thus, we can legitimately decompose $H^1_\rad$ as
\begin{equation}
H^1_\rad = \spn\{\overline{Q}_A\} \oplus \spn\{\overline{Q}_A\}^\perp
\end{equation}
with the orthgonalization is done with respect to the quadratic form
$\overline{\calB}_+^{(0)}$.  

Finally, let $u \in H^1_\rad$, $u \perp R$ and $u \perp \phi_2$.  Then
$u \perp q$.  With $A$ sufficiently large to make the
above decomposition valid, 
\begin{equation}
u = c(A) \overline{Q}_A + u^\perp_A.
\end{equation}
As in the heuristic argument $\overline{\calB}_+^{(0)}(u^\perp_A,
u^\perp_A)\geq 0$.  Thus,
\[
\overline{\calB}_+^{(0)}(u,u) = c(A)^2
\overline{\calB}_+^{(0)}(\overline{Q}_A, \overline{Q}_A) +
\overline{\calB}_+^{(0)}(u^\perp_A, u^\perp_A) \geq c(A)^2
\overline{\calB}_+^{(0)}(\overline{Q}_A, \overline{Q}_A) .
\]
We will have our result if $c(A)\to 0$ as $A\to +\infty$.  

Since $\inner{u}{q}=0$, 
\[
\begin{split}
c(A) \inner{\overline{U}_A^\ast} {q} &= -
\inner{u^\perp_A}{q}\\
& = - \inner{u^\perp_A}{\overline{\calL}_+^{(0)}
  \overline{Q}}\\
& = - \inner{u^\perp_A}{\overline{\calL}_+^{(0)}\paren{
  \overline{Q}- \overline{Q}_A }}.
\end{split}
\]  
Therefore,
\[
\abs{c(A)}= \frac{\abs{\inner{u^\perp_A}{\overline{\calL}_+^{(0)}\paren{
  \overline{Q}- \overline{Q}_A
}}}}{\abs{\inner{\overline{U}_A^\ast} {q} }}\leq
\frac{\norm{u_A^\perp}_+\norm{\overline{Q} - \overline{Q}_A}_+}{\abs{\inner{\overline{U}_A^\ast} {q} }}.
\]
Also,
\[
\begin{split}
\abs{\inner{\overline{U}_A^\ast} {q} - \inner{\overline{Q}} {q}} &=
\abs{\inner{\overline{U}_A^\ast- \overline{Q}}
{\overline{\calL}_+^{(0)} \overline{Q}} }\leq \norm{\overline{Q} - \overline{Q}_A}_+\norm{\overline{Q}}_+.
\end{split}
\]
Because this vanishes as $A\to +\infty$, we have that for all $A$
sufficiently large,
\[
\abs{c(A)}\leq C \norm{u_A^\perp}_+\norm{\overline{Q} - \overline{Q}_A}_+
\]
for a constant $C$ independent of $A$.

By construction,
\[
\norm{u_A^\perp}_+ \leq C\paren{\norm{u}_+ +  c(A)
  \norm{\overline{U}_A^\ast}_+}\leq C\paren{\norm{u}_+ +  c(A)
  \norm{\overline{U}_A^\ast - \overline{Q}}_+ + \norm{\overline{Q}}_+}.
\]
Substituting into our previous estimate on $c(A)$,
\[
\abs{c(A)} \leq C \paren{\norm{u}_+ +  c(A)
  \norm{\overline{U}_A^\ast - \overline{Q}}_+ + \norm{\overline{Q}}_+}\norm{\overline{Q} - \overline{Q}_A}_+.
\]
We can clearly see that as $A\to +\infty$, $c(A)\to + 0$.  We
conclude,
\[
\overline{\calB}_+^{(0)}(u,u) \geq 0
\]
for $u \in H^1_\rad$ and $u \perp R$ and $u\perp \phi_2$.    This
yields the estimate
\[
\calB_+^{(0)}(u,u) \geq \delta_0 \int e^{-\abs{\bx}^2} \abs{u}^2dx.
\]

Following the same analysis for $\overline{\calL}_-$, we conclude
\[
\calB_-^{(0)}(g,g) \geq \delta_0 \int e^{-\abs{\bx}^2} \abs{g}^2dx
\]
for $g\in H^1_\rad$ and $g \perp R + r R'$ since $J_1^{(0)}<0$.  Repeating this again for
$f \in H^1_{\rad(1)}$, $f \perp r R$, we get
\[
\calB_-^{(1)}(f,f) \geq \delta_0 \int e^{-\abs{\bx}^2} \abs{f}^2dx
\]
because $K_1^{(1)}<0$.

Let us assume that $\delta_0$ has been taken sufficiently small such
that:
\begin{itemize}
\item The indexes of all operators are unperturbed, as in Proposition \ref{prop:idx_perturbation},
\item The above arguments on the positivity of
  $\overline{\calB}_+^{(k)}$ for $k=0,1$ and
  $\overline{\calB}_-^{(0)}$ hold.
\end{itemize}

Then for $\mathbf{z} = (f,g)^T$ satisfying the orthogonality conditions,
\begin{equation*}
\begin{split}
\overline{\calB}(\mathbf{z},\mathbf{z}) &= \overline{\calB}_+(f,f)+
\overline{\calB}_-(g,g) \\
&=\sum_{k=0}^\infty \overline{\calB}_+^{(k)}(f^{(k)},f^{(k)}) + \sum_{k=0}^\infty
\overline{\calB}_-^{(k)}(g^{(k)},g^{(k)}) \\
& = \sum_{k=0}^\infty {\calB}_+^{(k)}(f^{(k)},f^{(k)}) + \sum_{k=0}^\infty
{\calB}_-^{(k)}(g^{(k)},g^{(k)})  - \delta_0\int
e^{-\abs{\bx}^2} \paren{\abs{f}^2 + \abs{g}^2}d\bx\\
& = \calB(\mathbf{z},\mathbf{z})  - \delta_0\int
e^{-\abs{\bx}^2} \abs{\mathbf{z}}^2 d\bx\geq 0.
\end{split}
\end{equation*}

We almost have the expression in Definition \ref{def:specprop}.  To complete the proof, note that for any $\theta \in (0,1)$,
\[
(1+\theta) \calB(\bz,\bz)  \geq \theta\paren{\int \abs{\grad \bz}^2 d\bx +
  \int V_+\abs{f}^2 + \int V_- \abs{g}^2 } + \delta_0 \int e^{-\abs{\bx}^2} \abs{\bz}^2d\bx.
\]
We can take $\theta=\theta_\star$ sufficiently small such that
\[
\theta_\star\paren{
  \int V_+\abs{f}^2 + \int V_- \abs{g}^2 } + \delta_0 \int
e^{-\abs{\bx}^2} \abs{\bz}^2d\bx\geq \frac{\delta_0}{2}\int
e^{-\abs{\bx}^2} \abs{\bz}^2d\bx.
\]
Then,
\[
\calB(\bz,\bz)\geq \frac{\theta_\star}{1+\theta_\star}\int \abs{\grad \bz}^2
d\bx + \frac{\delta_0}{2(1+\theta_\star)}\int
e^{-\abs{\bx}^2} \abs{\bz}^2d\bx.
\]
Shrinking $\delta_0$ again, so that it is smaller than
\[
\min\set{\frac{\theta_\star}{1+\theta_\star}, \frac{\delta_0}{2(1+\theta_\star)}}
\]
gives us the spectral property.\qed

\section{Other Problems}
In principle, this scheme can be applied to any linearized nonlinear
Schr\"odinger equation.  One finds the indexes of the operators, picks
an appropriate subspace to project away from, and computes the
necessary inner products.  However, our experiments show that the
algorithm is not as universal as might be hoped.  In this section
we exhibit the computations for several 1d NLS equations,
\begin{equation}
\label{e:nls_1d}
i \psi_t + \psi_{xx} + \abs{\psi}^{2\sigma} \psi = 0.
\end{equation}
Sometimes our approach works, ruling out embedded eigenvalues in a range
of supercritical cases, while in others it fails, leaving a large
range of interesting problems unresolved.

\subsection{Numerical Estimates of the Index} 
\label{s:1d_idx}
As in the 3d problem, we first compute the indexes of the operators
$\calL_\pm$ to identify the number of ``bad'' directions.  In contrast
to the multidimensional problems where there are an arbitrarily high,
but finite, number of harmonics which must be examined, 1D problems
only require us to study the operators restricted to even and odd
functions.  This requires the following results, whose proofs are
quite similar to that of Theorem \ref{thm:rs_idx}:
\begin{cor}
  Let $U$ be the even solution to
  \begin{gather*}
    \left\{ \begin{array}{c}
        L U = - U'' + V(r) U = 0, \\
        U(0) = 1, U' (0) = 0,
      \end{array} \right.
  \end{gather*}
  where $V$ is sufficiently smooth and decaying at $\infty$.  Then,
  the number $N(U)$ of zeros of $U$ is finite and
  \begin{equation*}
    \ind_{H^1_e} (B) = N(U),
  \end{equation*}
  where $B$ is the bilinear form associated with $L$.
\end{cor}
\begin{cor}
  Let $U$ be the odd solution to
  \begin{gather*}
    \left\{ \begin{array}{c}
        L U = - U'' - \frac{d-1}{r} U' + V(r) U = 0, \\
        U(0) = 0, U' (0) = 1,
      \end{array} \right.
  \end{gather*}
  where $V$ is sufficiently smooth and decaying at $\infty$.  Then,
  the number $N(U)$ of zeros of $U$ is finite and
  \begin{equation*}
    \ind_{H^1_o} (B) = N(U),
  \end{equation*}
  where $B$ is the bilinear form associated with $L$.
\end{cor}
$H^1_e$ and $H^1_o$ are the subspaces of $H^1(\R)$ restricted to even
and odd functions.  In what follows, we shall omit them in the
  subscripts of the indexes.

To proceed, we numerically solve the initial value problems
\begin{align}
  \calL_+^{(e)} U^{(e)} &=0, \quad U^{(e)}(0)=1, \quad\frac{d}{dx}U^{(e)}(0)=0,\\
  \calL_-^{(e)} Z^{(e)} &=0, \quad Z^{(e)}(0)=1, \quad
  \frac{d}{dx}Z^{(e)}(0)=0
\end{align}
for even functions $U^{(e)}$ and $Z^{(e)}$.  We then solve
\begin{align}
  \calL_+^{(o)} U^{(o)} &=0, \quad U^{(o)}(0)=0, \quad\frac{d}{dx}U^{(o)}(0)=1,\\
  \calL_-^{(o)} Z^{(o)} &=0, \quad Z^{(o)}(0)=0, \quad
  \frac{d}{dx}Z^{(o)}(0)=1
\end{align}
for odd functions $U^{(o)}$ and $Z^{(o)}$.  $\calL_\pm$ have the same
definitions as before; the Laplacian is now one dimensional.  As in
the 3d case, we verify {\it a postiori} that, asymptotically, $U^{\alpha}$ and
$Z^{\alpha}$ fit
\begin{align}
U^{\alpha(x)} &\approx C^{(\alpha)}_0 + C^{(\alpha)}_1 x,\\
Z^{\alpha(x)} &\approx D^{(\alpha)}_0 + D^{(\alpha)}_1 x.
\end{align}
In addition, the constants have the right signs to ensure we have the
correct number of zero crossings.

\begin{prop}[Numerically Verified]
  \label{prop:idx_1d}

  The indexes for the 1d NLS equation with $\sigma = 2, 2.1, 2.5, 3$
  are:
  \begin{gather*}
    \ind \calL_+^{(e)} =1, \quad \ind \calL_-^{(e)} =1,\\
    \ind \calL_+^{(o)} =1, \quad \ind \calL_-^{(o)} =0.
  \end{gather*}
\end{prop}
\begin{proof}
  Using the method discussed in Section \ref{sec:numerics}, we compute
  $U^{(\alpha)}$ and $Z^{(\alpha)}$ for each problem, $\alpha = e, o$.
  The profiles appear in Figures \ref{fig:1d_sig21_idx_even},
  \ref{fig:1d_sig21_idx_odd}.  All are computed with a relative
  tolerance of $10^{-10}$ and an absolute tolerance of $10^{-12}$
  using \matlab.

\begin{figure}
  \centering \subfigure[Even functions]{
    \includegraphics[width=2.5in]{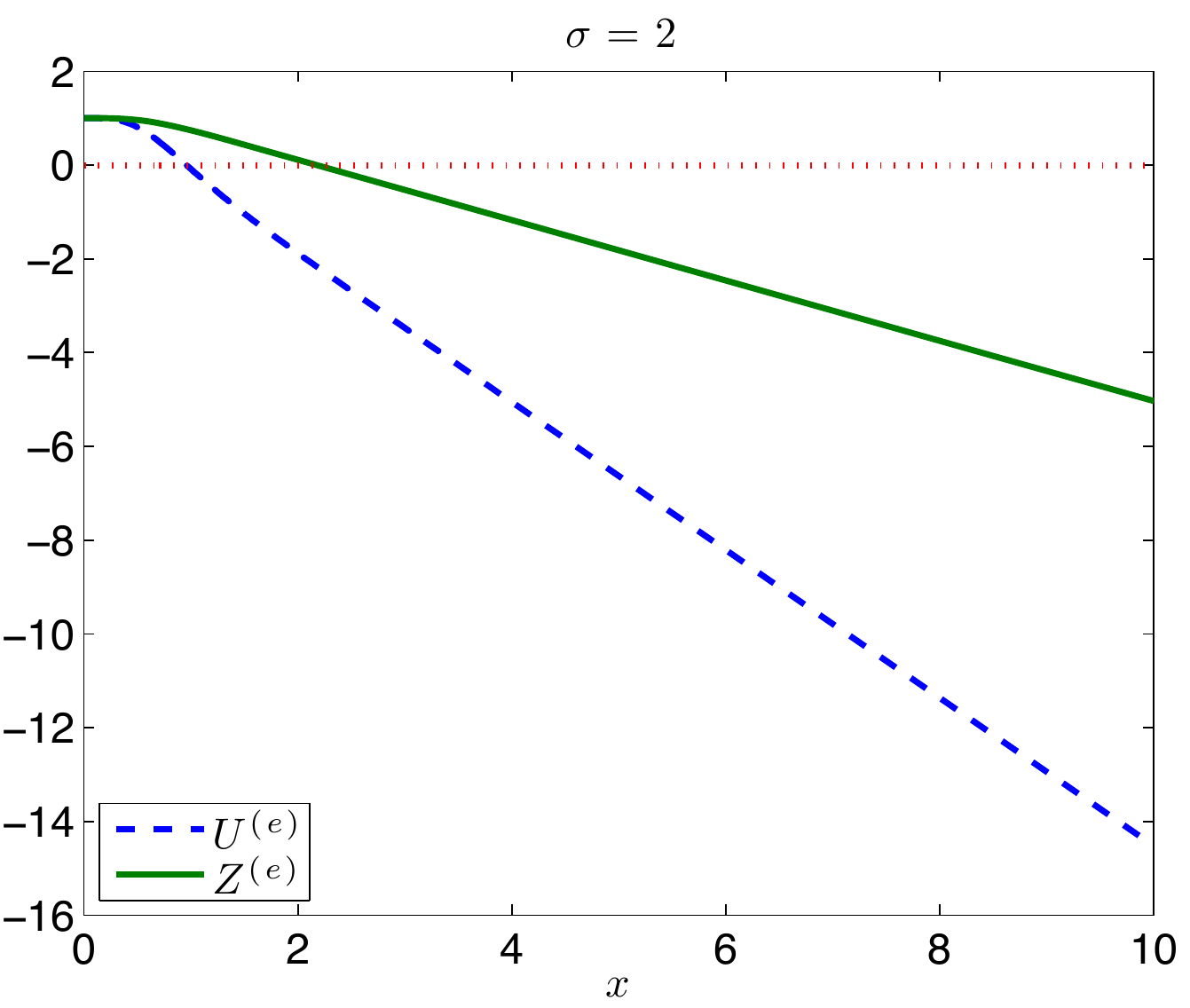}
  } \subfigure[Even function asymptotics]{
    \includegraphics[width=2.5in]{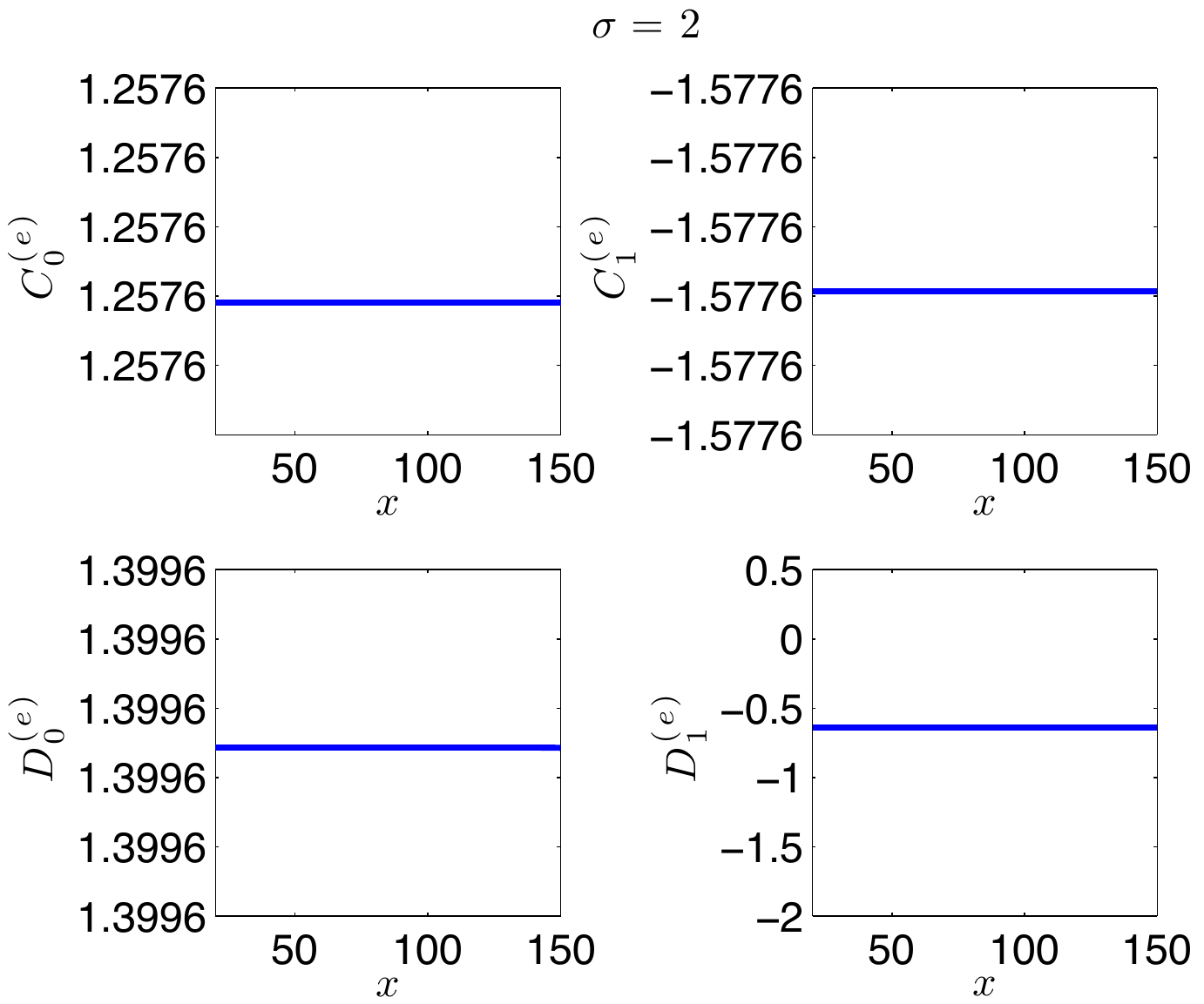}
  }
  \caption{Index computations for critical 1d NLS.  The number of zero
    crossings (other than $x=0$), determines the codimension of the
    subspace on which the operator $\calL_\pm^{(e/o)}$ is positive.}
  \label{fig:1d_sigcrit_idx_even}
\end{figure}

\begin{figure}
  \centering \subfigure[Odd functions]{
    \includegraphics[width=2.5in]{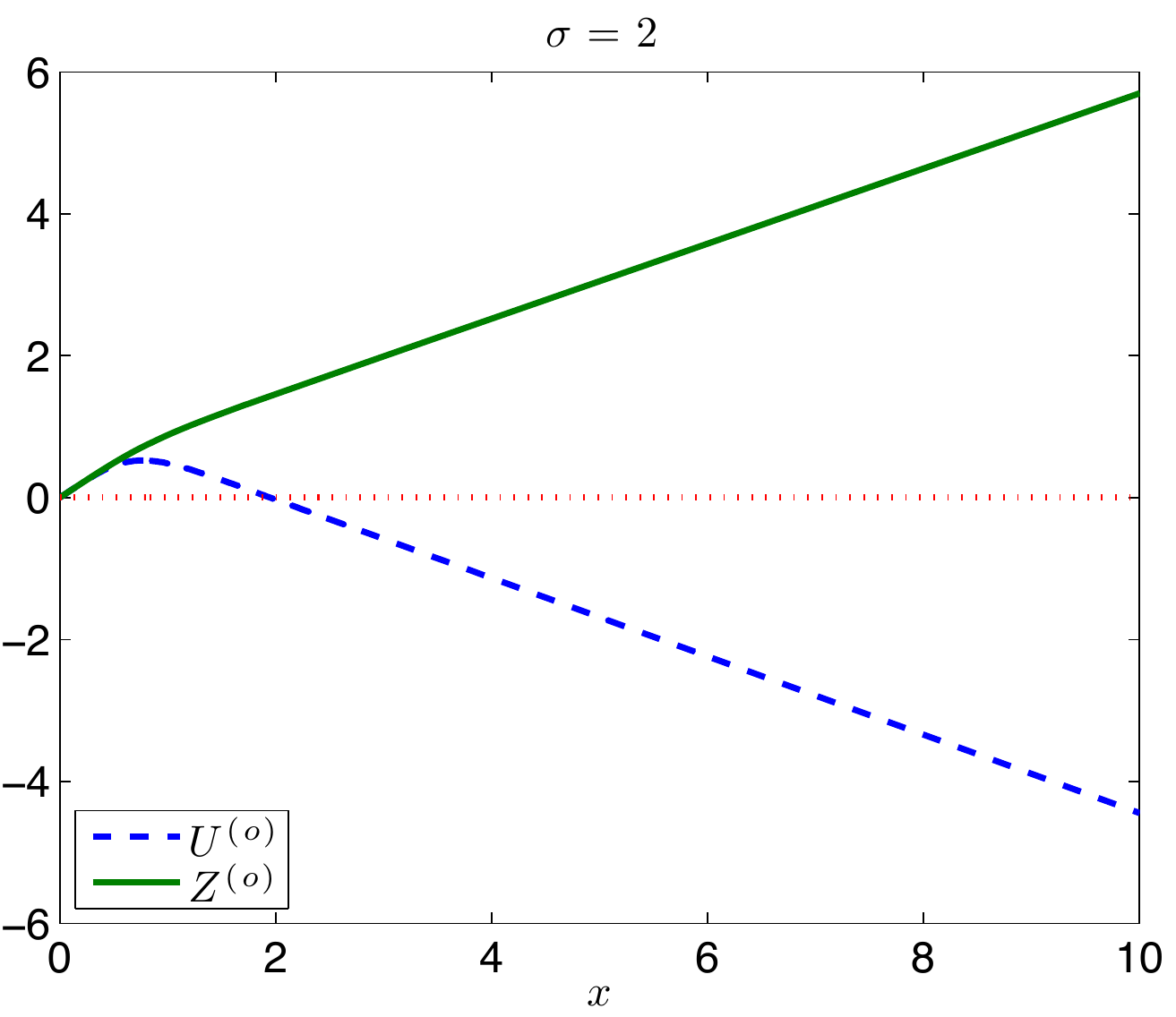}
  } \subfigure[Odd function asymptotics]{
    \includegraphics[width=2.5in]{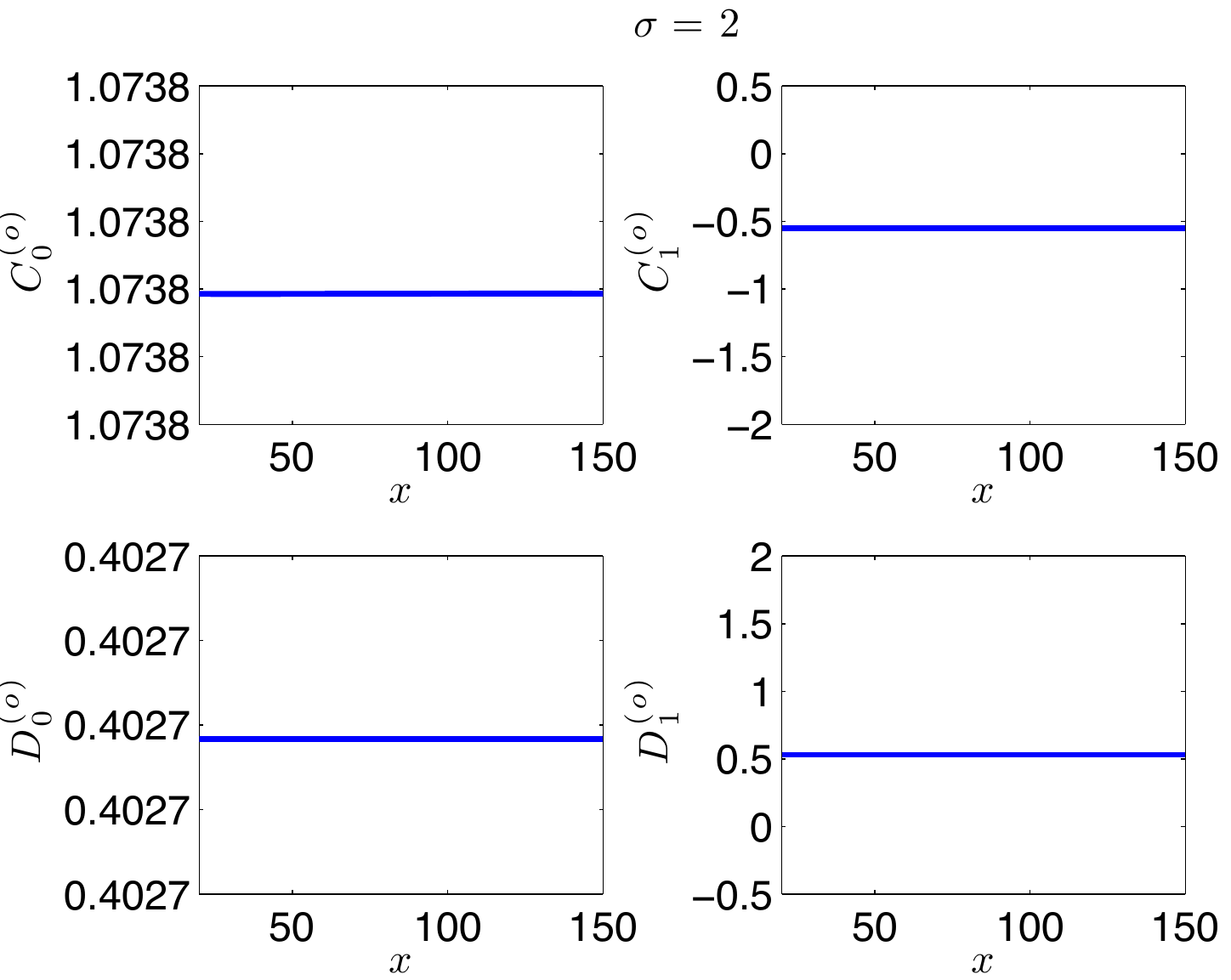}
  }
  \caption{Index computations for critical 1d NLS.  The number of zero
    crossings (other than $x=0$), determines the codimension of the
    subspace on which the operator $\calL_\pm^{(e/o)}$ is positive.}
  \label{fig:1d_sigcrit_idx_odd}
\end{figure}

\begin{figure}
  \centering \subfigure[Even functions]{
    \includegraphics[width=2.5in]{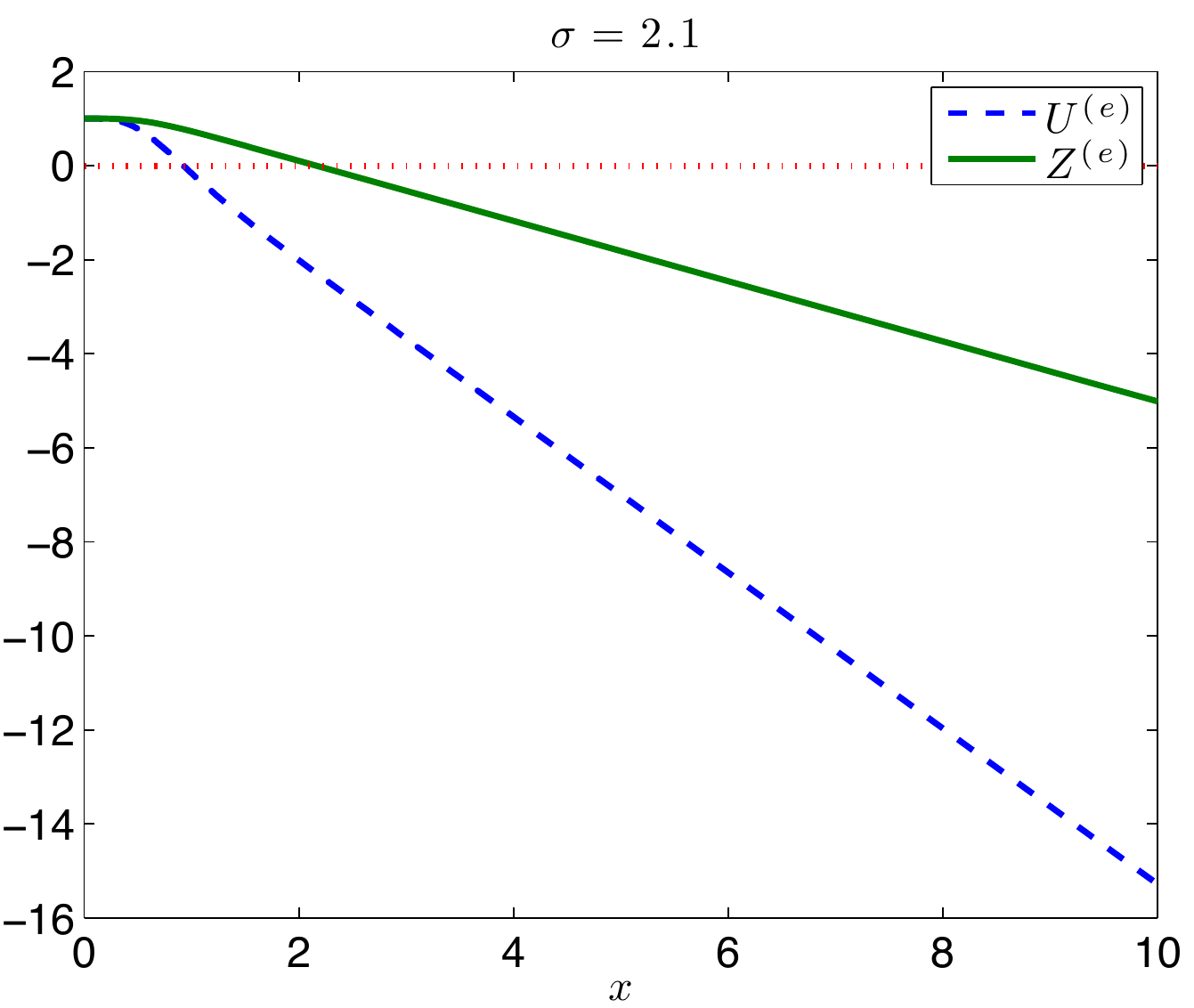}
  } \subfigure[Even function asymptotics]{
    \includegraphics[width=2.5in]{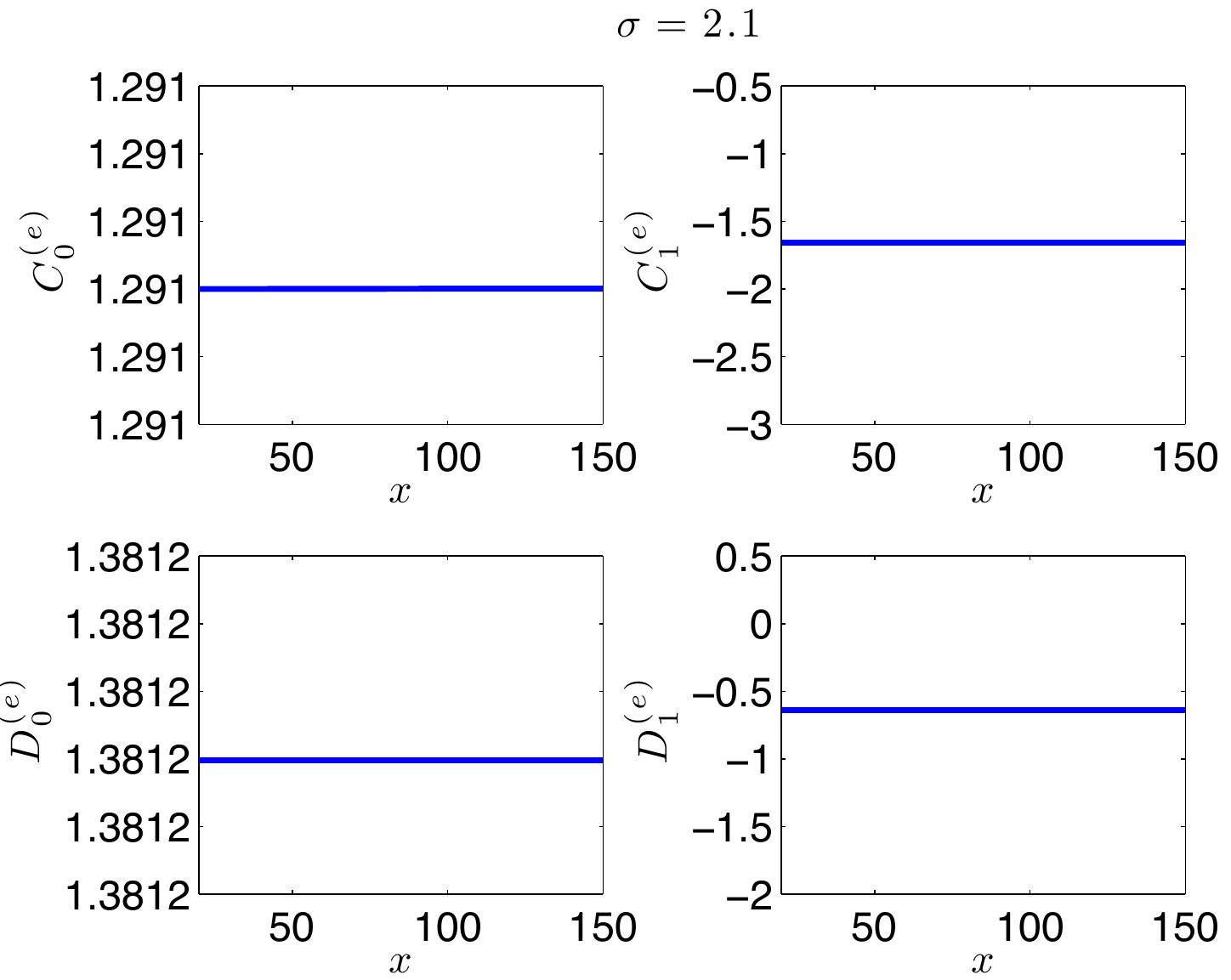}
  }
  \caption{Index computations for 1d NLS with $\sigma=2.1$.  The
    number of zero crossings (other than $x=0$), determines the
    codimension of the subspace on which the operator
    $\calL_\pm^{(e/o)}$ is positive.}
  \label{fig:1d_sig21_idx_even}
\end{figure}

\begin{figure}
  \centering \subfigure[Odd functions]{
    \includegraphics[width=2.5in]{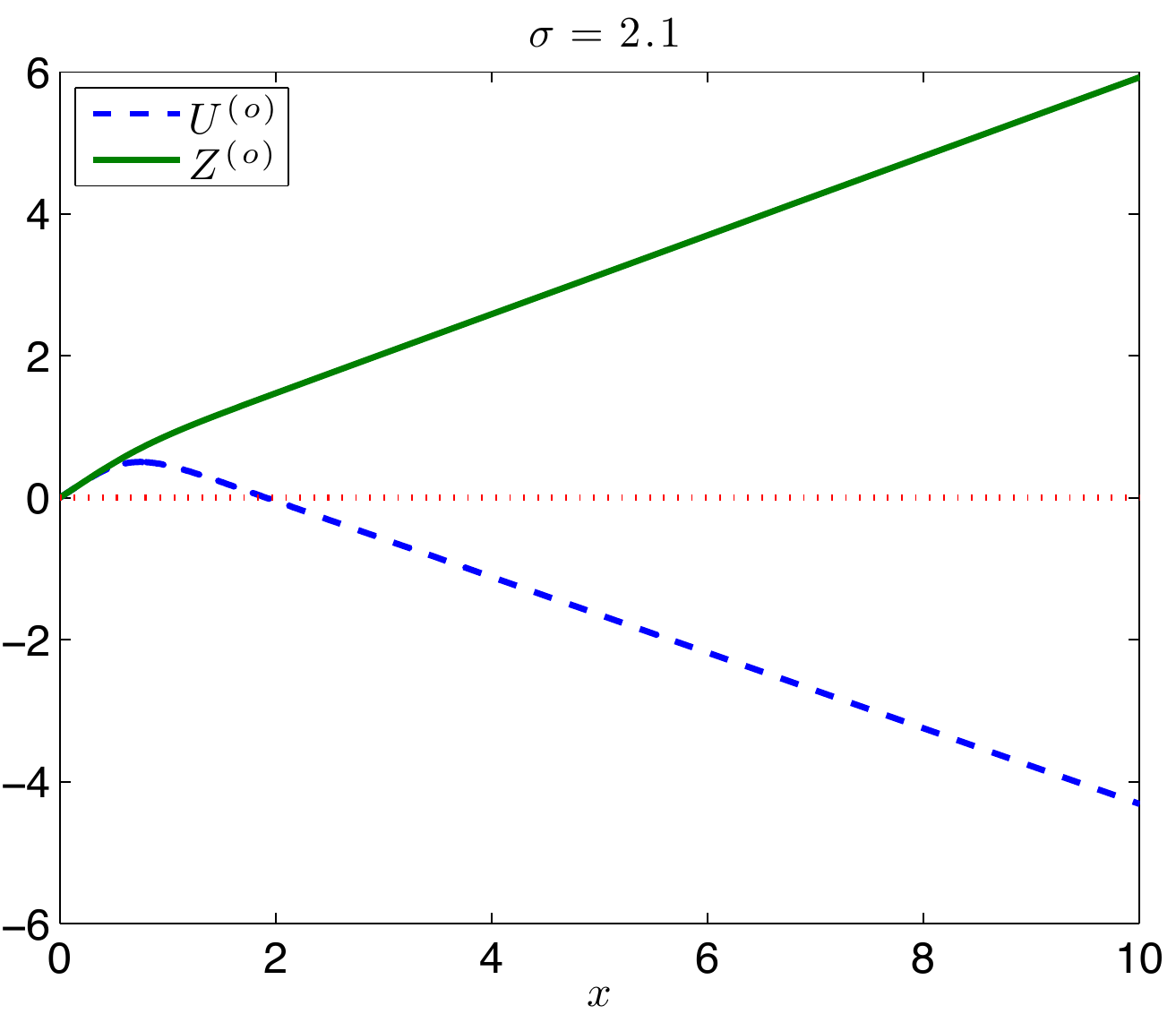}
  } \subfigure[Odd function asymptotics]{
    \includegraphics[width=2.5in]{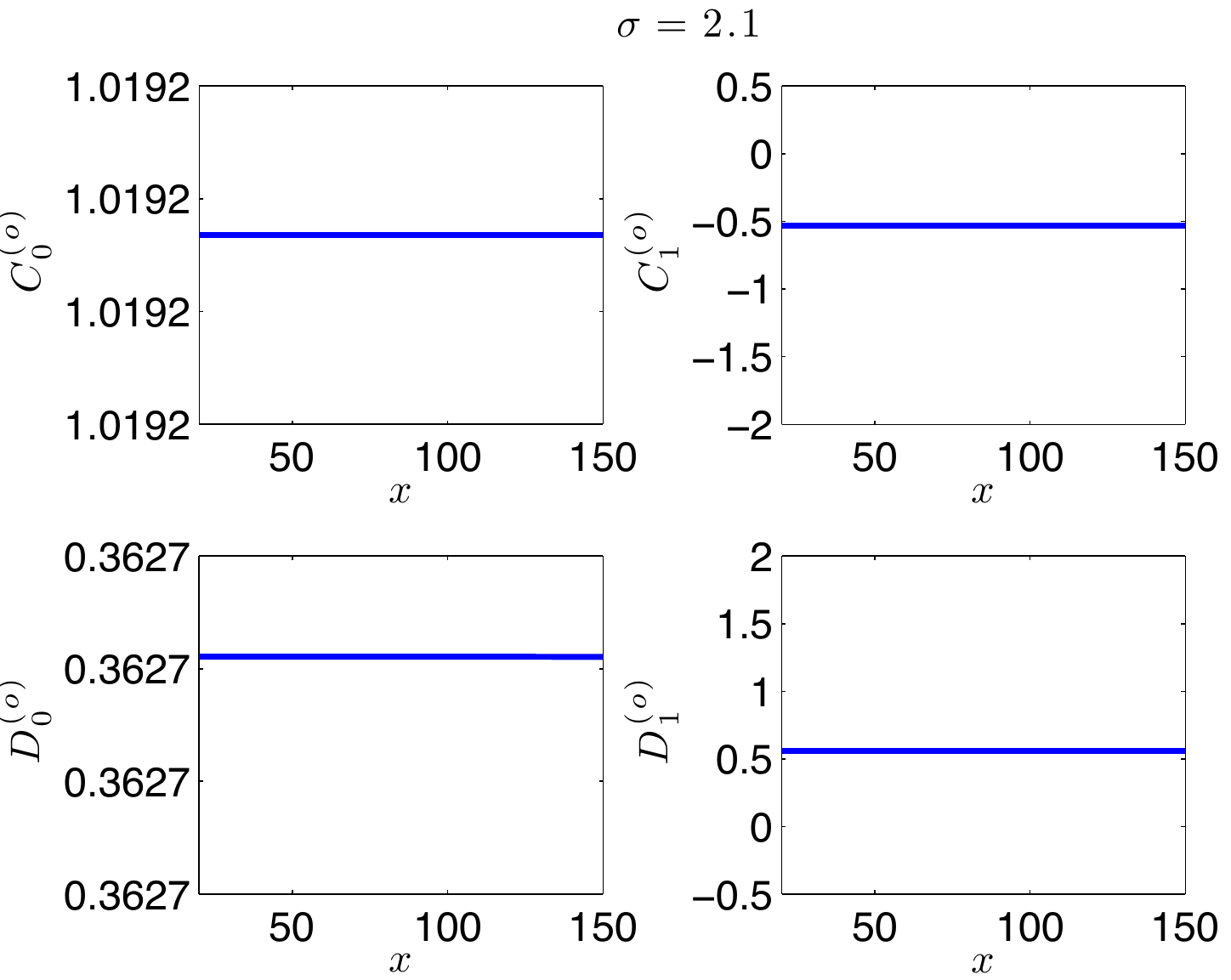}
  }
  \caption{Index computations for 1d NLS with $\sigma=2.1$.  The
    number of zero crossings (other than $x=0$), determines the
    codimension of the subspace on which the operator
    $\calL_\pm^{(e/o)}$ is positive.}
  \label{fig:1d_sig21_idx_odd}
\end{figure}

\begin{figure}
  \centering \subfigure[Even functions]{
    \includegraphics[width=2.5in]{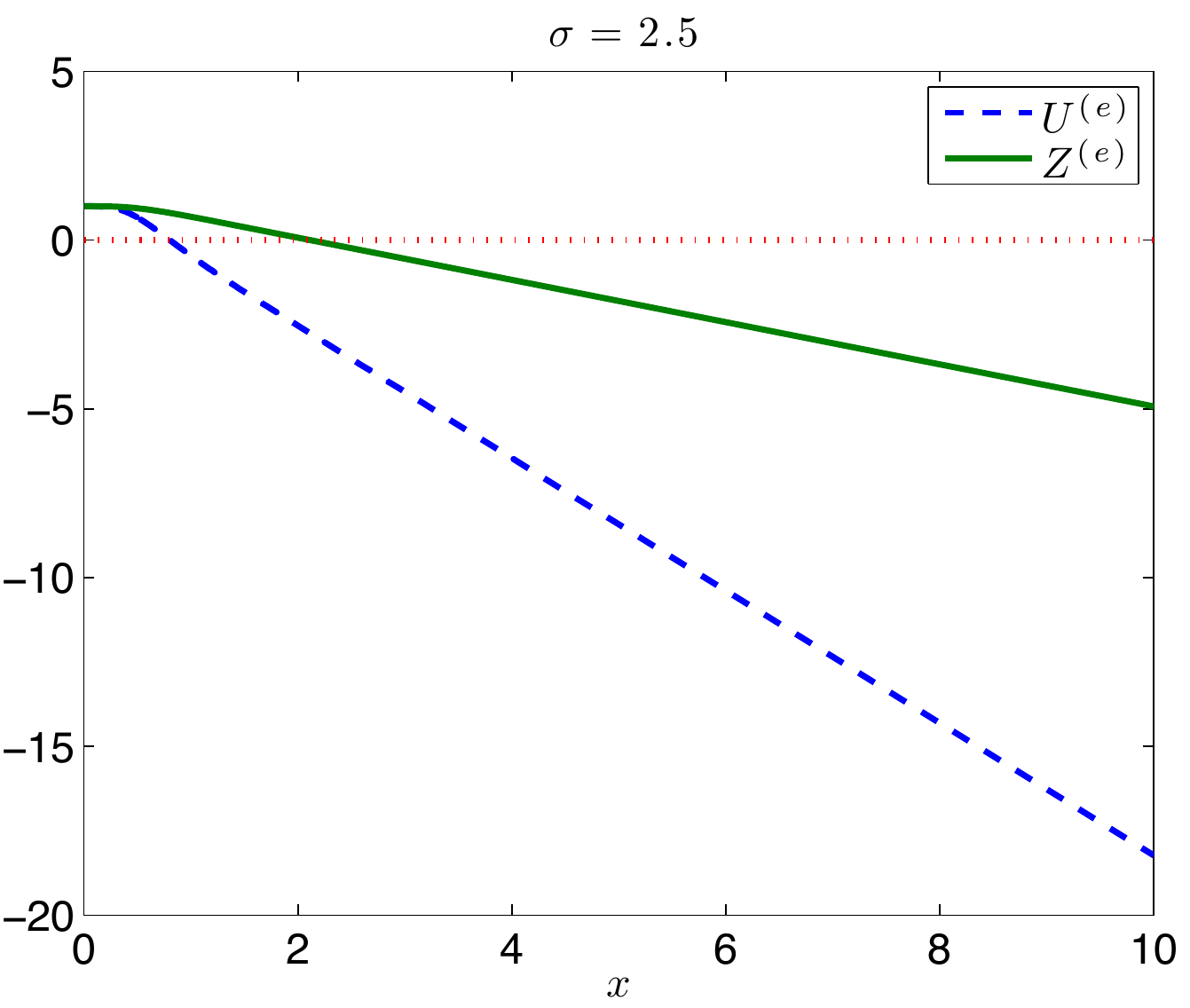}
  } \subfigure[Even function asymptotics]{
    \includegraphics[width=2.5in]{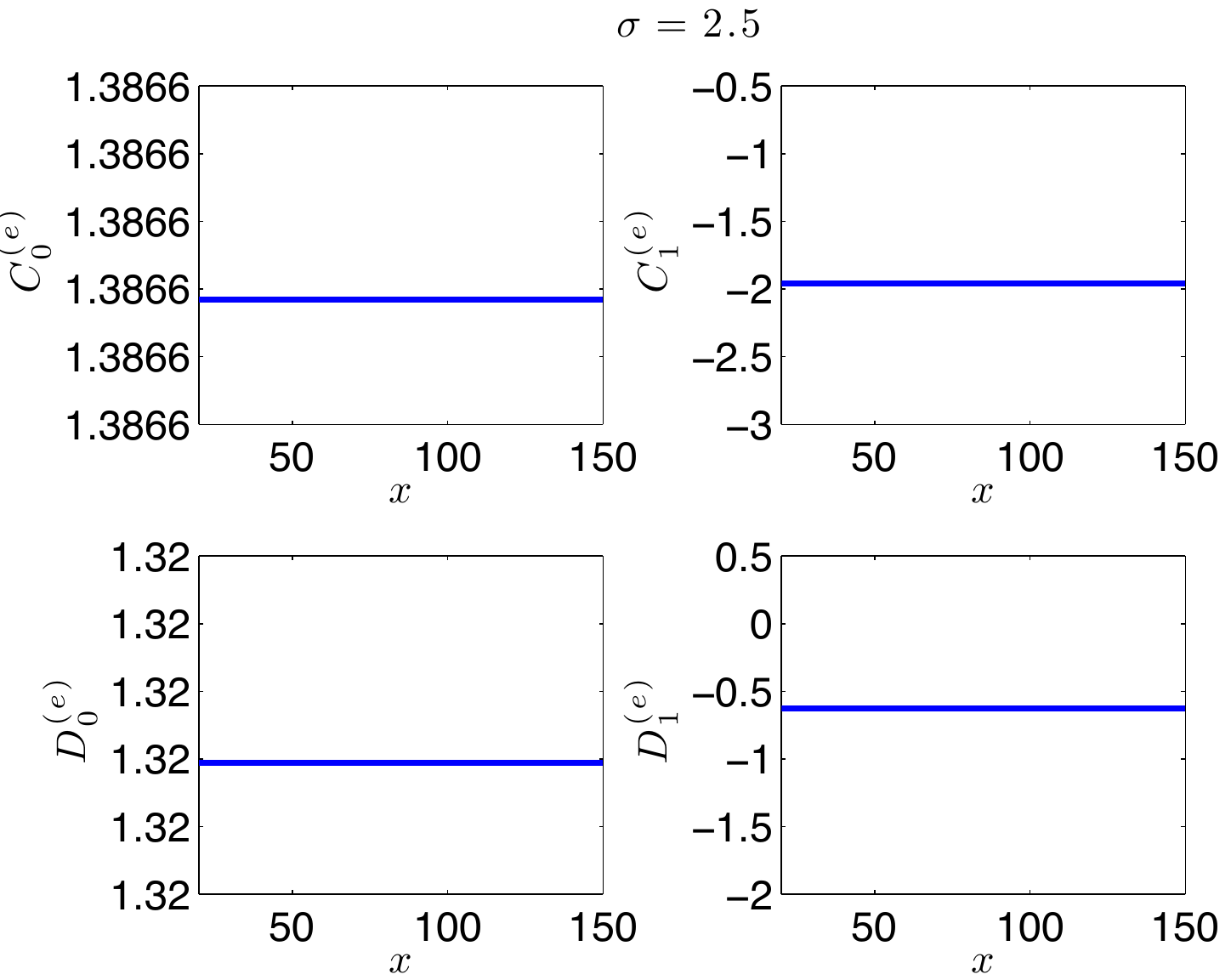}
  }
  \caption{Index computations for 1d NLS with $\sigma=2.5$.  The
    number of zero crossings (other than $x=0$), determines the
    codimension of the subspace on which the operator
    $\calL_\pm^{(e/o)}$ is positive.}
  \label{fig:1d_sig25_idx_even}
\end{figure}

\begin{figure}
  \centering \subfigure[Odd functions]{
    \includegraphics[width=2.5in]{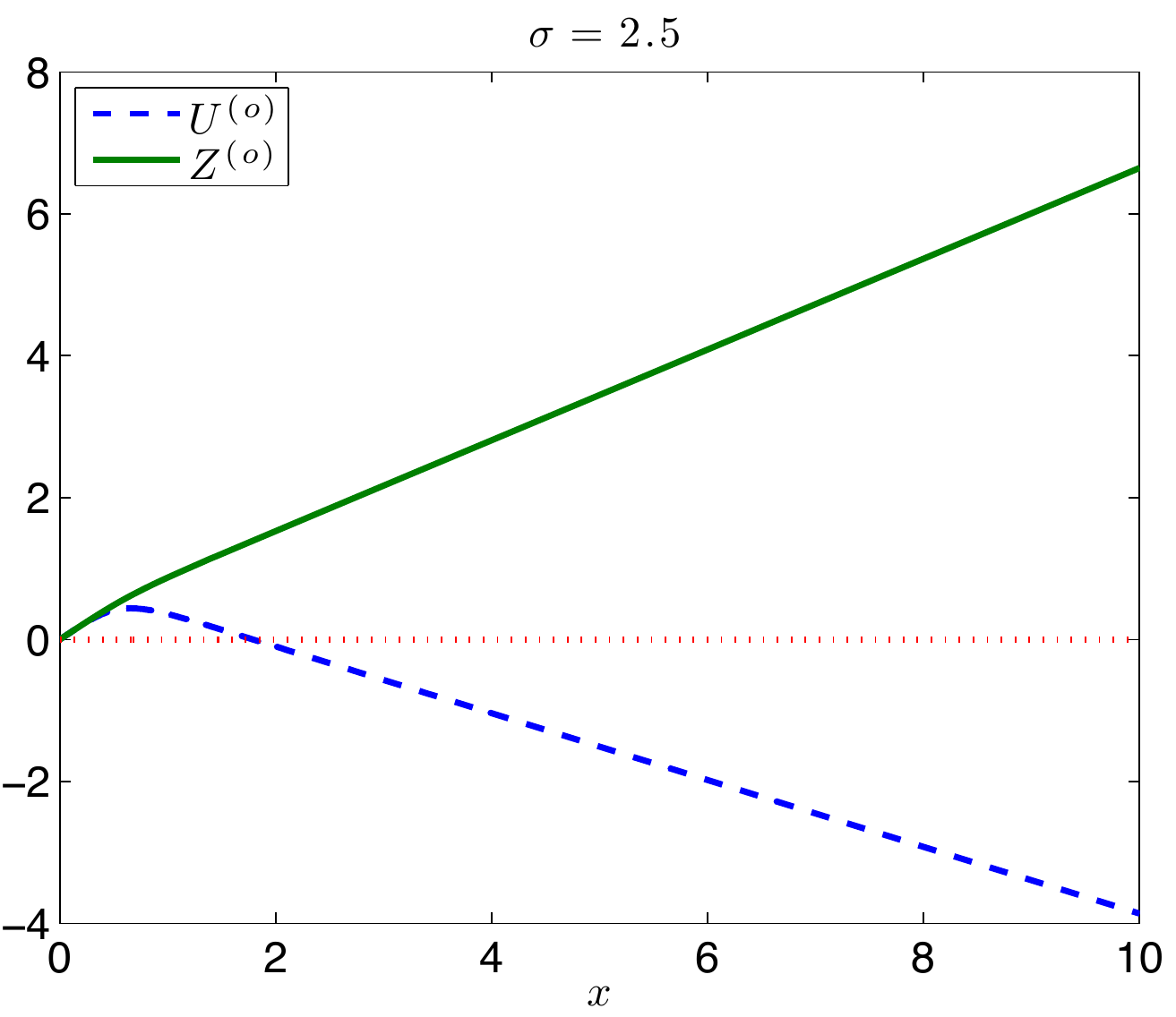}
  } \subfigure[Odd function asymptotics]{
    \includegraphics[width=2.5in]{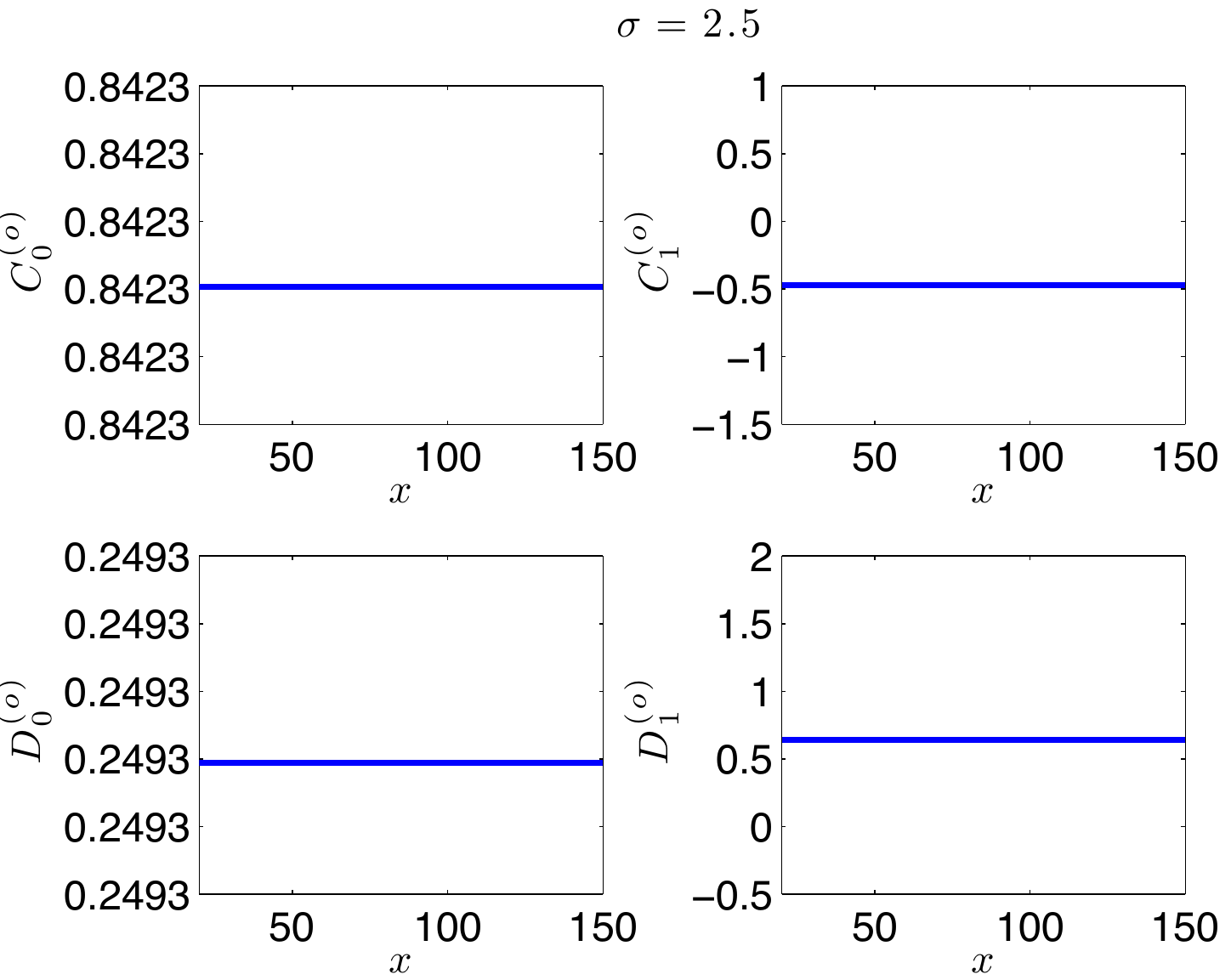}
  }
  \caption{Index computations for 1d NLS with $\sigma=2.5$.  The
    number of zero crossings (other than $x=0$), determines the
    codimension of the subspace on which the operator
    $\calL_\pm^{(e/o)}$ is positive.}
  \label{fig:1d_sig25_idx_odd}
\end{figure}

\begin{figure}
  \centering \subfigure[Even functions]{
    \includegraphics[width=2.5in]{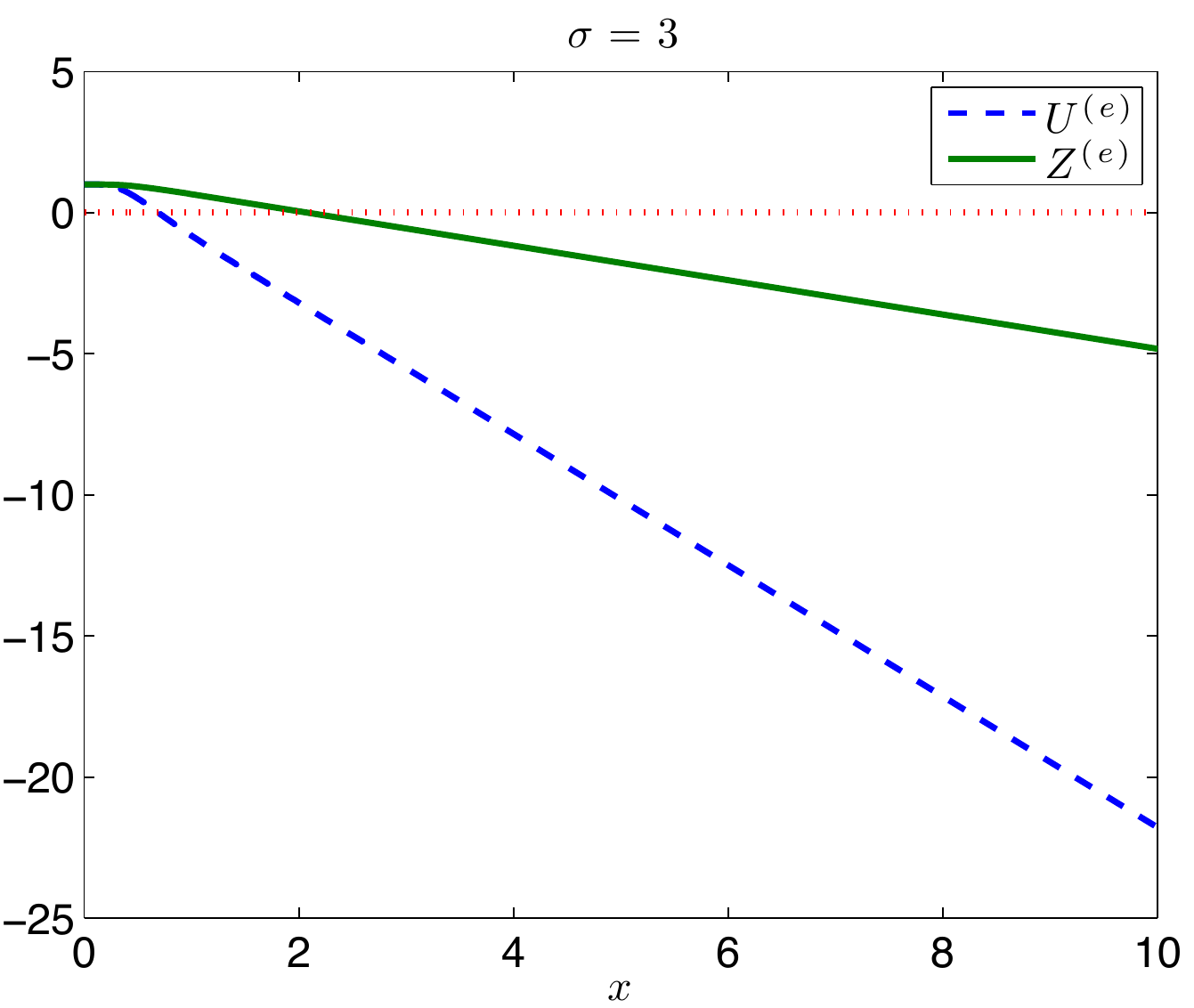}
  } \subfigure[Even function asymptotics]{
    \includegraphics[width=2.5in]{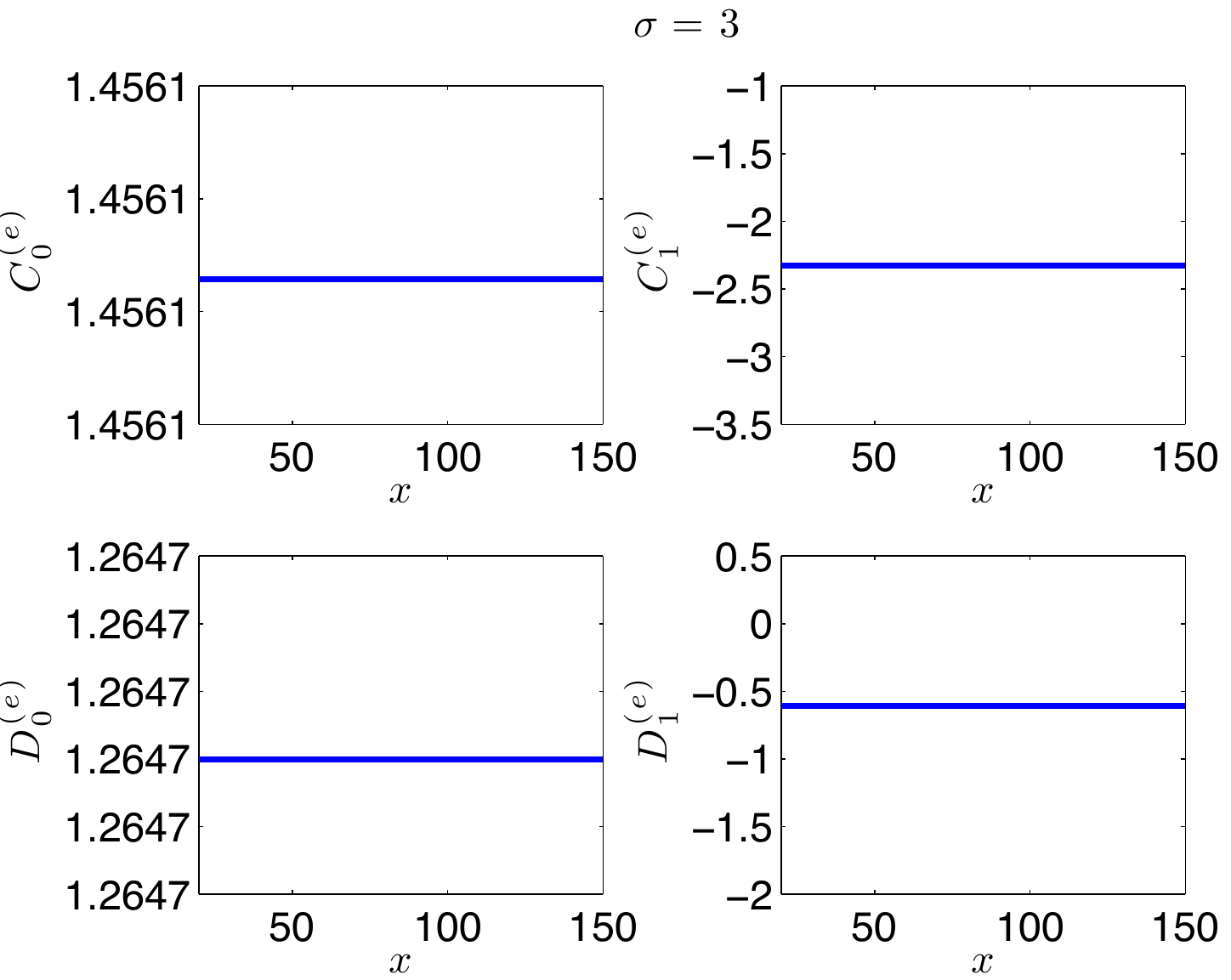}
  }
  \caption{Index computations for 1d NLS with $\sigma=3$.  The number
    of zero crossings (other than $x=0$), determines the codimension
    of the subspace on which the operator $\calL_\pm^{(e/o)}$ is
    positive.}
  \label{fig:1d_sig30_idx_even}
\end{figure}

\begin{figure}
  \centering \subfigure[Odd functions]{
    \includegraphics[width=2.5in]{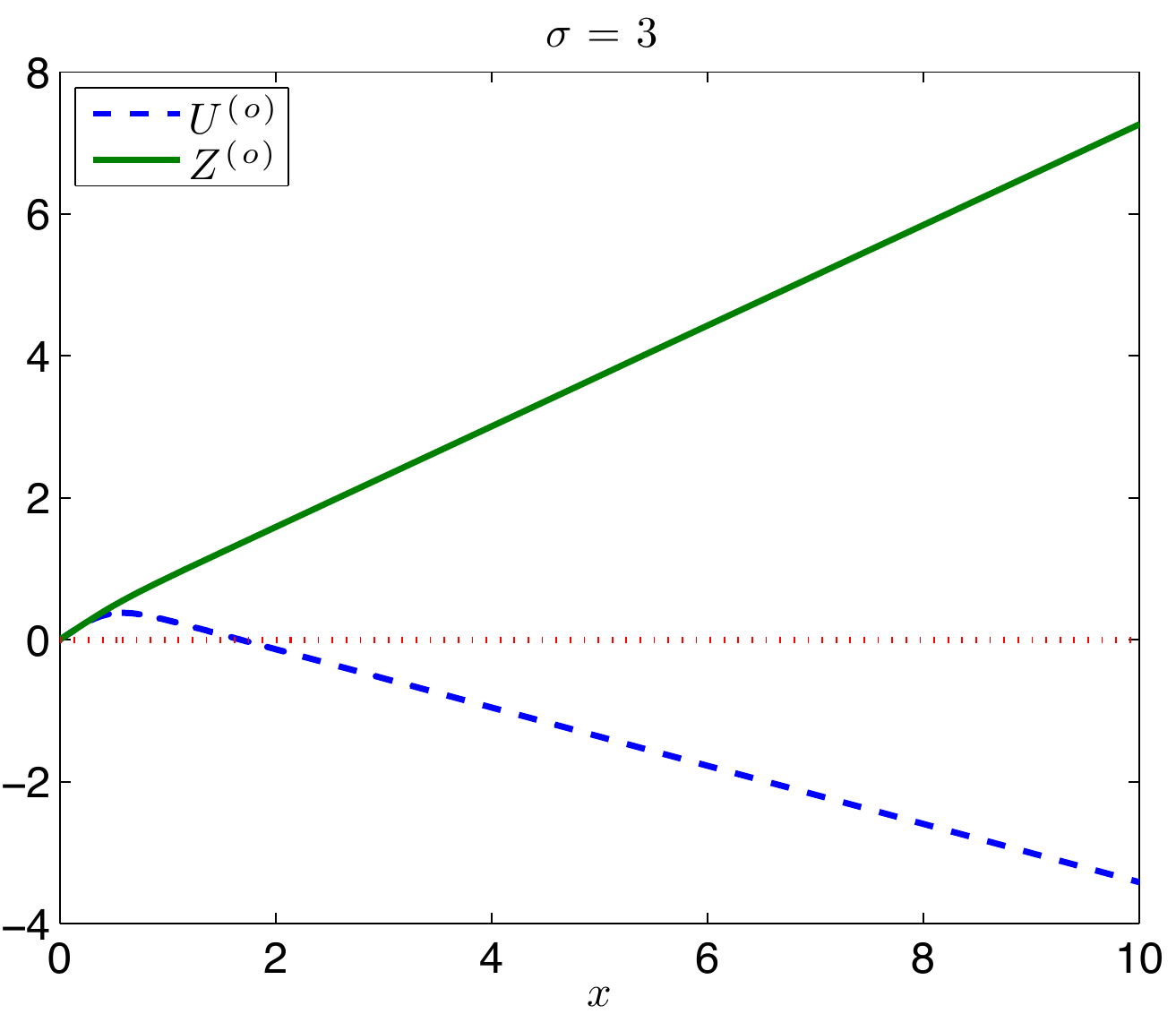}
  } \subfigure[Odd function asymptotics]{
    \includegraphics[width=2.5in]{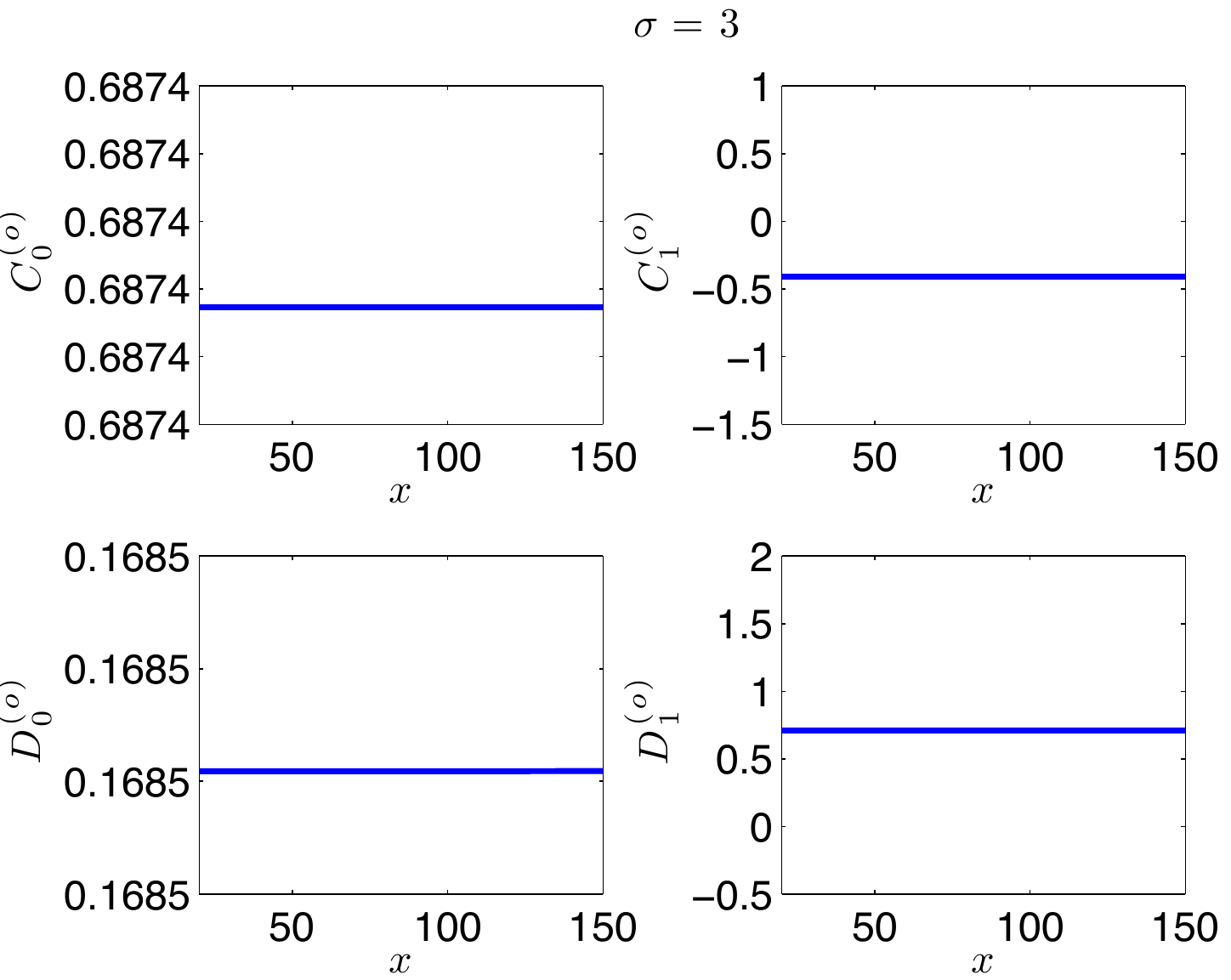}
  }
  \caption{Index computations for 1d NLS with $\sigma=3$.  The number
    of zero crossings (other than $x=0$), determines the codimension
    of the subspace on which the operator $\calL_\pm^{(e/o)}$ is
    positive.}
  \label{fig:1d_sig30_idx_odd}
\end{figure}
\end{proof}
Proposition \ref{prop:idx_perturbation} applies to these 1d problems
too.  As in the 3d case, we ultimately use a perturbed bilinear
form in the proof of the spectral property.

\subsection{Estimates of the Inner Products}
We now compute a series of inner products and show that in some cases
the natural orthogonality conditions are sufficient to yield a
spectral property.  Rigoursly, these results require the following
Proposition on the invertibility of the $\calL_\pm^{(e/o)}$ operators:

\begin{prop}[Numerically Verified for 1d Saturated Problem]
  \label{prop:eu_bvp_1d}

  Let $f$ be a smooth, localized function satisfying the bound
  $\abs{f(x)} \leq C e^{-\kappa \abs{x}}$ for some positive constants
  $C$ and $\kappa$.  If $f$ is even/odd, there exists a unique
  even/odd solution $u \in L^\infty(\R)\cap C^2(\R)$ to
  \begin{equation}
    \calL u = f,
  \end{equation}
  where $\calL = \calL^{(e/o)}_\pm$ for a 1d problem.
\end{prop}
\begin{proof}
  See Appendix \ref{sec:inverse_proof}.
\end{proof}

\begin{prop}[Numerical]
  \label{prop:ip_estimates_1d}

  Let $U_1^{(e)}$, $Z_1^{(e)}$, and $U_1^{(o)}$, all elements of
  $L^\infty(\R)$, solve the following boundary value problems:
  \begin{align}
    \calL_+^{(e)} U_1^{(e)} & =R , \quad \ddx U_1^{(e)}(0) = 0,\\
    \calL_-^{(e)} Z_1^{(e)} & =\frac{1}{\sigma}R + x R' , \quad \ddx Z_1^{(e)}(0) = 0,\\
    \calL_+^{(o)} U_1^{(o)} & =R' , \quad U_1^{(o)}(0) = 0.
  \end{align}
  Let
  \begin{align}
    K_1^{(e)} & \equiv\mathcal{B}_+^{(e)}(U_1^{(e)}) =
    \inner{\calL_+^{(e)} U_1^{(e)}}{U_1^{(e)}} ,\\
    J_1^{(e)}& \equiv\mathcal{B}_-^{(e)}(Z_1^{(e)}) =
    \inner{\calL_-^{(e)} Z_1^{(e)}}{Z_1^{(e)}} ,\\
    K_1^{(o)} & \equiv\mathcal{B}_+^{(o)}(U_1^{(o)}) =
    \inner{\calL_+^{(o)} U_1^{(o)}}{U_1^{(o)}} .
  \end{align}
  Then,
  \begin{center}
    \begin{tabular}{llll}
      $\sigma$ & $K_1^{(e)}$ & $J_1^{(e)}$ & $K_1^{(o)}$\\
      \hline
      $2.0$ & $-0.557768$
      &  $0.292551$
      & $-1.30410$
      \\
      $2.1$ & $-0.496932$
      & $0.216284$
      & $-1.21364$ 
      \\
      $2.5$ & $-0.297841$
      & $-0.0216292$
      & $-0.924662$ 
      \\
      $3.0$ & $-0.122559$ 
      & $-0.218499$
      & $-0.671783$.
    \end{tabular}
  \end{center}
\end{prop}

\subsubsection{Proof of the Spectral Property for Certain Supercritical Cases}
We restrict our attention to the 1d supercritical problems $\sigma =
2.5$ and $\sigma = 3$.  Repeating the procedure of Section
\ref{sec:spec_prop}, for $\mathbf{z} = (f, g)^T$, the orthogonality of $f$ to $R$
and $x R$ gives us $\overline{\calL}_+\geq 0$ and the orthogonality of $g$ to
$\frac{1}{\sigma }R + x R'$ gives us $\overline{\calL}_-\geq 0$.  This proves the
spectral property on the restricted subspace.  Since these
orthogonality conditions are consistent with those formulated in
Section \ref{sec:verification}, we conclude that there are no non zero
purely imaginary eigenvalues.

\subsubsection{An Inconclusive Supercritical Case}
In the case of $\sigma = 2.1$, we have that $J_1^{(e)}>0$, which means
that orthogonality of $g$ with respect to $\frac{1}{\sigma} R + x R'$,
is insufficient to guarantee positivity of $\calL_-$.  It is possible
that if we extend our scope, as in the 3D cubic problem, to include
orthogonality to the eigenstate associated with the unstable eigenvalue
we will be able to prove the spectral property for this problem.
However, we do not pursue that here; rather we wish to highlight the
failure of our algorithm at a seemingly arbitrary supercritical
nonlinearity.

\subsubsection{The Critical Case}
\label{s:1d_crit}
The critical 1D problem, with $\sigma = 2$, is also inconclusive.  As
in the supercritical problems we will look at the inner products
against $R$ and $\frac{1}{\sigma} R + x R'$.  We also employ 
inner products arising from with the rest of the generalized kernel,
$x^2 R$ and $\beta$, where $\beta$ solves
\begin{equation*}
  L_+ \beta = - x^2 R.
\end{equation*}
See \cite{W1} for details.  This motivates the following numerical result:
\begin{prop}[Numerical]
  Let $Z_2^{(e)}$ solve
  \begin{equation}
    \calL_-^{(e)}Z_2^{(e)} = \rho, \quad \ddx Z_1^{(e)}(0) = 0
  \end{equation}
 and let
  \begin{align}
    J_2^{(e)} & \equiv\mathcal{B}_-^{(e)}(Z_2^{(e)},Z_2^{(e)}) =
    \inner{\calL_-^{(e)} Z_2^{(e)}}{Z_2^{(e)}} ,\\
    J_3^{(e)} &\equiv\mathcal{B}_-^{(e)}(Z_1^{(e)},Z_2^{(e)})
    =\inner{\calL_-^{(e)} Z_1^{(e)}}{Z_2^{(e)}}.
  \end{align}
  Then,
  \begin{align}
    J_2^{(e)}  & =3.77915 ,\\ 
    J_3^{(e)} & =0.864273. 
  \end{align}
\end{prop}

Since $K_1^{(e)}<0$, we may conclude that $\calL_+\geq 0$, when the
operator is restricted to even functions that are orthogonal to $R$.
However, orthogonality to neither $ \frac{1}{2} R + x R'$ nor $\rho$
is, individually, sufficient to gain positivity of $\calL_-^{(e)}$.
We are thus motivated to consider orthogonality to the subspace
$\spn\{\frac{1}{2} R + x R', \rho\}$, as in the proof of the 3D cubic
problem.  We examine the quantity
\begin{equation}
  -\frac{1}{J_2^{(e)}}\paren{(J_3^{(e)})^2 -J_1^{(e)}J_2^{(e)} } = 0.0948958.
\end{equation}
However, we need this to be negative.  Thus, we have no set of natural
orthogonality conditions which yield a spectral
property.  In this case, since $K_1^{(o)}<0$, the only
obstacle to the proof is the $\calL_-^{(e)}$ operator.

\subsubsection{The Critical Case with Other Orthogonality Conditions}
If we had instead used the orthogonality condition, $g \perp R$, and
then solved the boundary value problem $\calL_-^{(e)}\check{Z}_{3}^{(e)} = R$,
the inner product,
\[
\check{J}_1^{(e)} \equiv\inner{\calL_-^{(e)}\check{Z}_1^{(e)}}{\check{Z}_1^{(e)}} = -3.770731.
\]
This would give us a spectral property, but it is not a convenient
subspace.

Suppose we use the orthogonality conditions of \cite{FMR}, and let $g
\perp \Lambda R$ and $g \perp \Lambda^2 R$.  Then, we compute as
follows: Let $\hat{Z}_1^{(e)}$, and $\hat{Z}_2^{(e)}$ solve the
following problems:
\begin{align}
  \calL_-^{(e)} \hat{Z}_1^{(e)} & =\Lambda R , \quad \ddx \hat{Z}_1^{(e)}(0) = 0,\\
  \calL_-^{(e)} \hat{Z}_2^{(e)} & =\Lambda R , \quad \ddx
  \hat{Z}_2^{(e)}(0) = 0,
\end{align}
and define the inner products
\begin{align}
  \hat{J}_1^{(e)} &
  \equiv\mathcal{B}_-^{(e)}(\hat{Z}_1^{(e)},\hat{Z}_1^{(e)}) =
  \inner{\calL_-^{(e)} \hat{Z}_1^{(e)}}{\hat{Z}_1^{(e)}} ,\\
  \hat{J}_2^{(e)} &
  \equiv\mathcal{B}_-^{(e)}(\hat{Z}_2^{(e)},\hat{Z}_2^{(e)}) =
  \inner{\calL_-^{(e)} \hat{Z}_2^{(e)}}{\hat{Z}_2^{(e)}} ,\\
  \hat{J}_3^{(e)}
  &\equiv\mathcal{B}_-^{(e)}(\hat{Z}_1^{(e)},\hat{Z}_2^{(e)})
  =\inner{\calL_-^{(e)} \hat{Z}_1^{(e)}}{\hat{Z}_2^{(e)}}.
\end{align}
We will find that
\begin{align}
  \hat{J}_1^{(e)} & = 0.292551, \\
  \hat{J}_2^{(e)} & =2.57656,\\
  \hat{J}_3^{(e)} & = -1.27657.
\end{align}
As one would hope, given that the 1D spectral property was established
in \cite{MR-Annals},
\begin{equation}
  -\frac{1}{\hat{J}_2^{(e)}}\paren{(\hat{J}_3^{(e)})^2
    -\hat{J}_1^{(e)}\hat{J}_2^{(e)} } =-0.339932.
\end{equation}
This sign ensures that projection away from those two directions is
sufficient to point us away from the negative eigenvalue, rendering
$\overline{\calL}_-^{(e)}\geq 0$.

\section{Discussion}
\label{s:discussion}

We have demonstrated a computer assisted algorithm for proving the
positivity of a bilinear form, $\mathcal{B}$, on a subspace $\calU$.
Because of the relationship between $\mathcal{B}$, $\calL$, and the
linearized operator, $JL$, we infer that there are no embedded
eigenvalues.  We succeeded with this program in the case of the 3d
cubic equation, and a two supercritical 1d problems.  C.~Sulem has
suggested t that is likely to also be successful for solitons (with
$\lambda = 1$) of the 3d cubic-quintic equation,
\[
i \psi_t + \Delta \psi + \abs{\psi}^2 \psi - \gamma \abs{\psi}^4 \psi = 0
\]
for $\gamma$ sufficiently close to zero.  We also anticipate success
for other 1d supercritical problems with $\sigma$ sufficiently large.
These cases warrant further study.

For subcritical problems, a similar algorithm should apply, though it
will certainly require a additional orthogonality
conditions.   Many subcritical problems contian eigenvalues inside the
spectral gap.  Since our approach does not distinguish between
embedded eigenvalues and imaginary eigenvalues in the gap, it would be
essential to project away from those states.

It remains to be seen how to extend our technique to other NLS/GP
equations.  Indeed, the failure in the 1D critical problem is
curious.  The success or failure of the approach is likely related to
the choice of our operator $\Lambda = d/2 + \bx
\cdot \nabla$. In \cite{FMR}, the authors proved the spectral property
using this $\Lambda$, as it is generated by the scaling invariance of
the mass critical problem.  This results in the so-called
``pseudoconformal invariant'' for critical NLS and has great
implications for blow-up.  See \cite{MR-Annals}
and \cite{sulem1999nse} for additional details.

Finally, recall that $\Lambda$ determines the operators $\calL_\pm$.  These
each have an index identifying the number of negative' directions.  We
then choose orthogonality conditions that simultaneously must satisfy
the two properties:
\begin{enumerate}
\item They must be orthogonal to any embedded eigenvalues,
\item Orthogonality in $L^2$ with respect to these directions must
  imply orthogonality to the negative directions of $\calL_\pm$, with
  respect to the $\calB_\pm$ quadratic form.
\end{enumerate}
The first requirement is satisfied by the vectors from the adjoint
problem, as discussed in Section \ref{sec:verification}.  We appear to
have little flexibility in altering these.  Changing
$\Lambda$ will change $\calL_\pm$; in turn this changes the negative
directions.  Thus, a different skew adjoint operator may extend the
applicability of the algorithm.

\appendix
\newpage

\section{Commutator Estimates}
\label{sec:mourre}
\subsection{Large Eigenvalues}
\label{spec:eig}

We now establish upper bounds on the magnitude of embedded $L^2$
eigenvalues of $JL$.  In this analysis, we
examine the fourth order equation comes from squaring the operator.
Therefore, we study eigenfunctions $u \in L^2$ of the differential
operator:
\begin{eqnarray}
  \label{eqn:lplm1}
  L_{-} L_{+} u = \mu^2 u,
\end{eqnarray}
where the $L_\pm$ operators are
\begin{equation*}
L_{\pm} = - \Delta + \lambda - V_{\pm}
\end{equation*}
and $\mu \in \sigma_{\cont}(JL) = (\lambda,\infty)$.  As our proof
applies to many NLS equations, we do not further specify the
potentials $V_\pm$; they are defined as in Section \ref{sec:sollin}.

From the properties of the soliton and the nonlinearity, we have that
any solution to equation \ref{eqn:lplm1} is locally smooth via an 
iteration argument \cite{Mlin}.  Following \cite{BeLi},  asymptotic
analysis shows that a solution decays exponentially fast.  As a
result, the possible range of frequencies is limited by
\begin{eqnarray}
  \label{spec:eqfreq}
  \| \nabla u \|_{L^2} \leq \mu \| u \|_{L^2}.
\end{eqnarray}
See \cite{Tat,Ste} for references on microlocal analysis.
We prove the following:
\begin{thm}

  There exists a $\mu_0 > \lambda$ such that for all $\mu \geq \mu_0$,
  the eigenvalue equation \eqref{eqn:lplm1} has only the trivial zero
  solution in $L^2$.

\end{thm}

\begin{proof}

We begin by defining the standard Mourre commutator as:
\begin{equation}
M \equiv \bx \cdot \nabla + \nabla \cdot \bx.
\end{equation}
Using the structure of these operators, we immeadiately have the two
identities
\begin{align}
\inner{[M L_- L_+ + L_+ L_- M]u }{u} & = 0,\\
\inner{[L_-L_+ - \mu^4]u}{u} & = 0,
\end{align}
where $\langle \cdot, \cdot \rangle$ is the standard $L^2$ inner
product.  Combining these two identities with the frequency
bound of \eqref{spec:eqfreq}, we can rule out $L^2$ solutions for
sufficiently large $\mu$.  Indeed, 
\begin{eqnarray*}
  \langle (\Delta^2 & - & 2 \lambda \Delta + \Delta(V_{+}(\bx)) + V_{-}(\bx) \Delta - \lambda(V_{-}+V_{+})  \\
  && + (\lambda^2 - \mu^2) + V_{-} V_{+}) u, u \rangle = 0, \ (*) 
\end{eqnarray*}
\begin{eqnarray*}
  M L_{-} L_{+} & - & L_{+} L_{-} M = [M,L_{-}]L_{+} + L_{-} [M,L_{+}] + [L_{-},L_{+}]M,\\
  \ [M,-\Delta] & = & 4 \Delta, \ [ M,V_{-} ] = 2 \bx \cdot \nabla V_{-}.
\end{eqnarray*}
Hence,
\begin{eqnarray*}
  \langle ( -8 \Delta^2 & + & 8 \lambda \Delta - 4(V_{-}+V_{+})\Delta + 2(\bx \cdot \nabla V_{-} + \bx \cdot \nabla V_{+}) \Delta  \\
  & - & 2 \lambda (\bx \cdot \nabla V_{-} + \bx \cdot \nabla V_{+}) + 2(\bx \cdot \nabla V_{-}) V_{+} + 2 (\bx \cdot \nabla V_{+}) V_{-} \ (**) \\
  & - & [(V_{-}-V_{+})\Delta - \Delta (V_{-}-V_{+}) ] [d + 2 \bx \cdot \nabla ] ) u,u \rangle = 0.
\end{eqnarray*}
Since the last term is a product of skew-adjoint operators, they commute.

Combining $4(*) + (**)$ and with the frequency bound $\| \nabla u
\|_{L^2} \leq \mu^2 \| u \|_{L^2}$, we have for $\mu > \mu_0$, this is
a negative definite system.  Thus, there are no eigenvalues.
Furthermore, this estimate combined and standard Sobolev embeddings
implies $u \in H^k$ for any $k$; hence, $u$ is smooth.

The system we assess is:
\begin{eqnarray*}
  \int [-4 (\Delta u)^2 & - & 4\lambda (V_{-}+V_{+}) u^2 - 4 (\mu^2 - \lambda^2) u^2] \\
  & + & [4V_{-}V_{+} - 2 \lambda (\bx \cdot \nabla V_{-} + \bx \cdot \nabla V_{+}) \\
  & + & 2 (\bx \cdot \nabla V_{-}) V_{+} + 2( \bx \cdot \nabla V_{+})V_{-} + \Delta (\bx \cdot \nabla V_{-} + \bx \cdot \nabla V_{+})  \\
  & - & \bx \cdot \nabla (\Delta V_{-} - \Delta V_{+}) - d (\Delta V_{-} - \Delta V_{+}) ]u^2 \\
  & + & [ 4 (\bx \cdot \nabla u) (\nabla(V_{-}-V_{+}) \cdot \nabla u) \\
  & - & 2(\bx \cdot \nabla V_{-} + \bx \cdot \nabla V_{+}) \nabla u \cdot \nabla u ] \ d\bx = 0.
\end{eqnarray*}
Hence,
\begin{eqnarray*}
  \int (-4 (\Delta u)^2 & - & 4\lambda (V_{-}+V_{+}) u^2 )d\bx - 4 (\mu^2 - \lambda^2) \|u\|_{L^2}^2 +  \| F \|_{L^\infty} \| u \|^2_{L^2} \\
  & + & ( 2 \| \bx \cdot \nabla V_{-} + \bx \cdot \nabla V_{+} \|_{L^\infty} \\
  & + & C_d \max_{j,k} \| \partial_j (V_{-}-V_{+})  x_k \|) \|\nabla u \|_{L^2}^2 \\
  & \leq & \int (-4 (\Delta u)^2 - 4\lambda (V_{-} +V_{+}) u^2 )d\bx \\
  && - 4 (\mu^2 - \lambda^2 - C_1 - C_2 \mu^2) \|u\|_{L^2}^2,
\end{eqnarray*}
where
\begin{eqnarray*}
  F & = & 4V_{-} V_{+} - 2 \lambda ( \bx \cdot \nabla V_{-} + \bx \cdot \nabla V_{+}) + 2  (\bx \cdot \nabla V_{-})V_{+} \\
  & + & 2 (\bx \cdot \nabla V_{+})V_{-} + \Delta (\bx \cdot \nabla V_{-} + \bx \cdot \nabla V_{+}) \\
  & - & \bx \cdot \nabla (\Delta V_{-} - \Delta V_{+}) - d( \Delta V_{-} - \Delta V_{+})
\end{eqnarray*}
and $C_j = C_j (V_{-},V_{+},\lambda,d)$ for $j = 1,2$. Hence, for $\mu$ large,
we have that:
\begin{eqnarray*}
  \int (-4 (\Delta u)^2 - 4\lambda (V_{-}+V_{+}) u^2 - C_3(\mu,\lambda,V_{-},V_{+}) u^2) \ d\bx \leq 0,
\end{eqnarray*}
for $C_3 > 0$ and $u$ a smooth function, hence $u = 0$.
\end{proof}

\subsection{Spherical Harmonics}
\label{spec:sphhar}

As our potentials are radially symmetric, we expand our
functions in spherical harmonics.  Seperating the radial variable,
$r$, from the angular variables, $\theta$, the expansion takes the form
\begin{eqnarray*}
  \sum_k u_k (r) \phi_k (\theta),
\end{eqnarray*}
where
\begin{eqnarray*}
  \Delta_S \phi_k (\theta) = (k^2 + (d-1) d) \phi_k (\theta).
\end{eqnarray*}
See \cite{Tay2} for a description of the eigenspaces of the spherical
Laplacian, $\Delta_S$.  Then, we have the following ODE eigenvalue
problem:
\begin{equation}
  \label{eqn:lplm2}
\begin{split}
  \left[ \left( -\frac{\partial^2}{\partial r} \right. \right. & - \left. \frac{d-1}{r} \frac{\partial}{\partial r} +
    \frac{\alpha^2}{r^2} + \lambda - V_{-}(r) \right) \times \\
  \left( -\frac{\partial^2}{\partial r} \right. & -  \left. \left.\frac{d-1}{r} \frac{\partial}{\partial r} + \frac{\alpha^2}{r^2} + \lambda - V_{+}(r) \right)  - \mu^2 \right] u_k(r) = 0,
\end{split}
\end{equation}
where $\alpha^2 = k^2 + (d-1) d$ for $k = 0,1,2,\dots$.  Note, we use the notation
$\alpha^2$ since all of the eigenvalues are non-negative, which will
be important in the sequel.

We have following theorem:
\begin{thm}
  There exists some $\alpha_0 > 0$ such that for all $\alpha \geq
  \alpha_0$, the eigenvalue equation \eqref{eqn:lplm2} has only the
  trivial  solution in $L^2$.
\end{thm}
\begin{proof}

  Let us denote the radial inner product by:
  \begin{eqnarray*}
    \langle u,v \rangle_r = \int uv r^{d-1} dr,
  \end{eqnarray*}
  and the operators:
  \begin{eqnarray*}
    \Delta_r & = & r^{1-d} \frac{\partial}{\partial r} \left( r^{d-1} \frac{\partial}{\partial r} \right), \\
    P_r & = & d + 2 r \frac{\partial}{\partial r}.
  \end{eqnarray*}

  Using the same commutator approach as in Section \ref{spec:eig},
  \begin{eqnarray*}
    \left\langle [ 4 \Delta_r \right.& - & \frac{4 \alpha^2}{r^2} -2 r V_{-}'(r)][-\Delta_r + \frac{\alpha^2}{r^2} + \lambda - V_{+}(r)] \\
    & + & [-\Delta_r + \frac{\alpha^2}{r^2} + \lambda -V_{-}(r)][ 4 \Delta_r - \frac{4 \alpha^2}{r^2} -2 r V_{+}'(r)] \\
    & + & [(V_{-}-V_{+})\Delta_r - \Delta_r(V_{-}-V_{+})][d+2r \frac{\partial}{\partial r} ] u_k, u_k \rangle_r = 0.
  \end{eqnarray*}
  Thus, 
  \begin{eqnarray*}
    \langle -8 (\Delta_r)^2 u  & + & 16 \Delta_r (\frac{\alpha^2}{r^2} u_k) + 8 \lambda \Delta_r u_k - 4(V_{-}+V_{+})\Delta_r u_k - \frac{8 \alpha^4}{r^4} u_k \\
    & - & \frac{8 \lambda \alpha^2}{r^2} u_k + \frac{4 \alpha^2}{r^2} (V_{-}+V_{+}) u_k + 2 (r V_{-}'(r) + r V_{+}'(r)) \Delta_r u_k \\
    & - & 2 \frac{\alpha^2}{r} (V_{-}'(r) + V_{+}'(r)) u_k - 2 \lambda r (V_{-}'(r)+V_{+}'(r)) u_k \\
    & + & 2 r V_{-}'(r) V_{+}(r) u_k + 2 r V_{-}(r) V_{+}'(r) u_k \\
    & + & (d+2r \frac{\partial}{\partial r})((V_{-}-V_{+}) \Delta_r - \Delta_r (V_{-}-V_{+})) u_k, u_k  \rangle_r \\
    & = & 0.
  \end{eqnarray*}
  This implies:
  \begin{eqnarray*}
    \int \left[ -8 (\Delta_r u_k)^2 \right. & - & (16 \frac{\alpha^2}{r^2} + 8 \lambda )((u_k)_r)^2 - (\frac{8 \alpha^4 }{r^4} + \frac{8 \lambda \alpha^2}{r^2}) u_k^2] r^{d-1} dr \\
    & + & \int [4(V_{-}+V_{+})  + 2r (V_{-})_r -6 r (V_{+})_r] ((u_k)_r)^2 r^{d-1} dr \\
    & + & \int [ \frac{4 \alpha^2}{r^2} (V_{-}+V_{+}) - (d-4) \frac{\alpha^2}{r^4} - \frac{d(d-1)+2\alpha^2}{r} (V_{-})_r \\
    & + & \frac{(d-1)(d-2)-2 \alpha^2}{r} ((V_{+})_r) - (d-2) (V_{-})_{rr} + 3d (V_{+})_{rr}  \\
    & + & 2r ((V_{+})_r V_{-} + V_{-} (V_{+})_r) + 2r (V_{+})_{rrr} - 2 \lambda r ((V_{-})_r+(V_{+})_r)] u_k^2 r^{d-1} dr \\ 
    & \leq & 0.
  \end{eqnarray*}

  In the preceding calculations, we integrated by parts several times above.  To
  justify this, $r=0$ must be a root of $u_k$ of sufficiently high multiplicity to
  compensate for the singular terms.  Fortunately, spherical harmonics
  result from eigenvalues of the spherical Laplacian.  These take the
  values

$$ \nu_k = k^2 + (d-2)k,$$ 
and for each $k$, the eigenfunctions (and hence the spherical
harmonics) are traces of harmonic polynomials of degree $k$.  As a
result, in order to give a smooth solution as guaranteed above, $u_k
(0)$ must be a zero of multiplicity $k$, or $u^{(m)}_k (0) = 0$ for all $m =
0,1,\dots,k$.  Hence, for $k \geq \max\{0,5-d\}$, the behavior of $u$
is sufficient to make the calculations rigorous.  See
\cite{Tay2}, Chapter 8 for a detailed description of eigenfunctions
for the Laplacian on the sphere.

In the commutator expression, the parameter that must dominate is
$\alpha^4$.  Since $V_{-}$, $V_{+}$ are smooth, exponentially decaying
functions by assumption, all terms involving $V_{-}$, $V_{+}$ and
derivatives thereof are nicely bounded at $0$ and exponentially
decaying.  Hence, for $0 \leq r \leq 1$, all of the functions above
are easily controlled by $\frac{\alpha^4}{r^4}$ for $\alpha$
sufficiently large.

Similarly, for $r > r_\star$, $r_\star$ sufficiently large, the exponential decay of $V_{-}$, $V_{+}$
and their derivatives imply that any function above is dominated 
$\frac{\alpha^4}{r^4}$ once $\alpha$ is sufficiently.  In the
intermediate region, using the smoothness of the potential
functions we can find $\alpha^4$ large enough to bound the lower
order terms.  In order to determine $\alpha$ exactly, a careful
analysis must be done involving all of the extrema of $V_{-}$, $V_{+}$
and their derivatives.  As these functions are uniformly bounded,
there exists $\alpha_0$ such that for all $\alpha \geq
\alpha_0$, the operator when conjugated by the radial Mourre operator
gives a negative definite system.  Hence, the result holds.
\end{proof}

\begin{rem}
  Any embedded eigenvalue can be
  expressed purely as a finite sum of spherical harmonics with radial
  coefficients.  Combining this with the result in Section
  \ref{spec:eig} gives a limited range for the calculations one must
  do in order to determine whether or not an operator has no  embedded
  eigenvalues.
\end{rem}

\section{Proof of the Invertibility of the Operators}
\label{sec:inverse_proof}

In this section we give a full proof of Proposition
\ref{prop:eu_bvp_1d}.  This relies on our numerically computed
indexes, from Proposition \ref{prop:index_computations}.  This proof
generalizes to the 3d problem, establishing Proposition
\ref{prop:eu_bvp_3d}.

Following \cite{FMR,MR-Annals}, we first prove uniqueness and then
existence.  Before beginning the uniqueness part, we recall the
following extension of the Levinson Theorem,
\cite{levinson1947asn,coddington1984tod}, from \cite{eastham1989asl}:
\begin{thm}[Eastham] For the equation
  \[
  y^{(n)}+a_1(x)y^{(n-1)}+ \ldots + a_n(x) y(x) = 0,
  \]
  assume
  \[
  \int_a^\infty x^{j-1}\abs{a_j(x)}dx < \infty
  \]
  for some $a>0$ and $j=1,\ldots n$.  Then there exist $n$ solutions,
  $y_k(x)$, such that as $x\to \infty$,
  \begin{align*}
    y_k^{(i-1)}(x) &\sim \frac{x^{k-i}}{(k-i)!}, \quad 1\leq i\leq k,\\
    y_k^{(i-1)}(x)&=o(x^{(k-i)}), \quad k+1 \leq i \leq k.
  \end{align*}
\end{thm}

To prove uniqueness, let $u \in L^\infty(\R)$ solve $\calL u = 0$.  We
prove $u = 0$.   The equation, $-u'' + V(x) u = 0$, satisfies the
hypotheses of the above theorem,  so
there exist two solutions, $\rho_1(x)$ and $\rho_2(x)$, such that as
$x\to +\infty$,
\begin{align}
  \rho_1(x)\sim 1,&\quad \rho_1'(x)= o(x^{-1}),\\
  \rho_2(x)\sim x, & \quad \rho_2'(x)=o(1).
\end{align}
Due to the behaviour as $x\to +\infty$, $\rho_1$ and $\rho_2$ are
linearly independent.  These can then be extended to all of $\R$ by
the classical theory of linear systems with smooth coefficients.

By the same argument, there exist $\tilde{\rho}_1$ and
$\tilde{\rho}_2$ satisfying
\begin{align}
  \tilde\rho_1(x)\sim 1, &\quad \tilde\rho_1'(x)= o(x^{-1}),\\
  \tilde\rho_2(x)\sim x,& \quad \tilde\rho_2'(x)=o(1)
\end{align}
as $x\to - \infty$.  These two are also  linearly independent.  We will make
use of these two sets of functions in what follows.  An important
relation between them is that, by uniqueness, $\tilde
\rho_j(x) = \rho_j(-x)$ for $j = 1, 2$.

Given the values $u(0) = u_0$ and $u'(0)=u_0'$, we know from the
existence and uniqueness of solutions to linear systems with smooth
coefficients, that there exist two unique pairs of constants,
$\set{c_1, c_2}$ and $\set{\tilde{c}_1, \tilde{c}_2}$ such that
\begin{align*}
  u(x) &= c_1 \rho_1(x) + c_2 \rho_2(x)\\
  & = \tilde{c}_1 \tilde{\rho}_1(x) + \tilde{c}_2 \tilde{\rho}_2(x).
\end{align*}

As $x\to +\infty$,
\[
u(x) \sim c_1 + c_2 x.
\]
Since $u(x) \in L^\infty$, we conclude $c_2=0$.  Since it is proportional
to $\rho_1(x)$, $u'(x)$ vanishes as $x\to +\infty$.  An analogous argument
ensures that $u'(x)$ also vanishes as $x\to-\infty$.  Thus we have
\begin{equation}
  \label{eq:vanishing_homog_soln}
  \lim_{\abs{x}\to\infty} u'(x) = 0.
\end{equation} 

Since $u \in C^2(\R)$ and its derivative vanishes, we conclude $u' \in
L^\infty(\R)$.  Furthermore, we claim $u' \in L^2(\R)$.  Multiplying the
equation by $u$ and integrating by parts,
\[
\int_{-L}^L \abs{u'}^2 dy - u u'|^{L}_L + \int_L^L V u^2 dy=0.
\]
Sending $L\to \infty$ proves the claim.  Additionally, this shows $\inner{\calL
  u}{u}=0$.

Let $\chi_A(x)$ be the cutoff function
\begin{equation}
  \chi_A(x) =\begin{cases} 1, & \abs{x}\leq A,\\
    2\paren{\frac{\log A}{\log\abs{x}}-\frac{1}{2}}, & A < \abs{x} \leq A^2,\\
    0, & A^2 < \abs{x}.
  \end{cases}
\end{equation}
and assume $A>1$. Let $u_A(x)=\chi_A(x)u(x)$.  We prove
\begin{equation}
\lim_{A\to \infty} \inner{\calL u_A}{u_A} = \inner{\calL u}{u}.
\end{equation}
which will be essential to the uniqueness proof.  Trivially, the
potential component, $\int V u_A^2$, will converge to $\int V u^2$,
since $V$ is highly localized.  We now justify the convergence of the
kinetic component,
\begin{equation}
  \begin{split}
    \int \abs{u_A'}dy &= \int \chi_A^2 \abs{u'}^2dy + \int
    \abs{\chi_A'}^2
    \abs{u}^2 dy  + 2 \int \chi_A u' \chi_A' u dy\\
    &= I_1 + I_2 + I_3.
  \end{split}
\end{equation}
The integral $I_1$ converges to $\int \abs{u'}^2$, the desired
quantity.  We must show the other two vanish.  First, we split up
$I_2$ into
\[
I_2 = \int \abs{\chi_A'}^2 u^2 dy = \int_{-A^2}^{-A} \abs{\chi_A'}^2
u^2 dy + \int_{A}^{A^2} \abs{\chi_A'}^2 u^2 dy.
\]
Using our explicit characterization of the cutoff function
\[
\chi_A'(x) = -\frac{2\log A}{x(\log\abs{x})^2},
\]
\[
\int_{A}^{A^2} \abs{\chi_A'}^2 \abs{u}^2 dy \leq \norm{u}_{L^\infty}^2
\int_{A}^{A^2}\frac{4(\log A)^2}{y^2 (\log y)^4} dy.
\]
As $A\to +\infty$, the integral is $\sim 1/((\log(A^{-1}) )^6A)$, which
vanishes as $A \to +\infty$.  The integral over $(-A^2, -A)$ is treated similarly.

The other integral is
\[
I_3 = 2 \int_{-A^2}^A \chi_A u' \chi_A' u dy + 2 \int_{A}^{A^2} \chi_A
u' \chi_A' u dy.
\]
Again, using the explicit characterization of the cutoff function,
\begin{equation*}
  \begin{split}
    \int_{A}^{A^2} \chi_A u' \chi_A' u dy &\leq
    \norm{u}_{L^\infty}\norm{u'}_{L^\infty(A,A^2)}\int_A^{A^2}\frac{2\log
      A}{y(\log\abs{y})^2}dy\\
    & \leq \norm{u}_{L^\infty}\norm{u'}_{L^\infty(A,A^2)}.
  \end{split}
\end{equation*}
Since the derivative vanishes as $x\to +\infty$,
\eqref{eq:vanishing_homog_soln}, this also vanishes.
The other part of $I_3$ is treated analogously.  This proves convergence
of the bilinear form.

We now specialize to either even or odd functions.  By our index computations in Proposition
\ref{prop:index_computations}, $\calL^{(e/o)}_\pm$ each have index
$1$, except for $\calL^{(o)}_-$ which has index $0$. Without loss of
generality, we assume $\calL$ is an 
index 1 operator, and proceed.  Let $\psi$ be the negative eigenvector with
$\norm{\psi}_{L^2}=1$ for the relevant symmetry.  Let $V_A = \spn
\set{\psi, u_A}$.  We now show that $\overline{\calL}$ restricted to
this subspace is negative definite.  By the index computations this will
imply $u_A$ and $\psi$ are collinear, allowing us to conclude that
$u=0$ since $u_A$ has the same symmetry properties to $u$.

Let $q$ be any element of $V_A$,
\[
q = c_1 u_A + c_2 \psi.
\]
Then,
\[
\inner{\overline{\calL} q}{q} = c_1^2 \inner{\overline{\calL}
  u_A}{u_A} + 2 c_1 c_2 \inner{\overline{\calL} u_A}{q} + c_2^2
\inner{\overline{\calL} \psi}{\psi}.
\]
We claim this is negative, which shows $\overline{\calL}$, restricted
to $V_A$, is negative definite.  By the index computation, there is
only one negative eigenvalue.  Thus, $\dim V_A=1$, and we must have
$u_A = c(A) \psi$.  But then
\[
-\lambda c(A) = \inner{u_A}{\lambda \psi}= \inner{u_A}{\calL \psi} =
\inner{\calL u_A}{\psi}.
\]
Since the right hand side vanishes as $A\to \infty$, we conclude that $c(A)
= 0$; hence $u = 0$.

To prove the claim that the form is negative, it is equivalent to show
\begin{equation}
  \label{eq:uniqueness_negdef}
  \inner{\overline{\calL} u_A}{\psi}^2 < \inner{\overline{\calL} u_A}{u_A} \inner{\overline{\calL}\psi}{\psi}.
\end{equation}
As $A\to \infty$,
\begin{align*}
  \inner{\overline{\calL} u_A}{u_A}&\to - \delta_0 \int e^{-\abs{y}}\abs{u}^2,\\
  \inner{\overline{\calL}u_A}{\psi}&\to -\delta_0\int e^{-\abs{y}} u
  \psi
\end{align*}
and
\[
\inner{\overline{\calL}\psi}{\psi} \leq \lambda < 0.
\]
Thus, \eqref{eq:uniqueness_negdef} holds for $A$ sufficiently large
and $\delta_0$ sufficiently small. Indeed, given $u\neq 0$, and
$\psi$, let
\[
\delta_0 \leq \frac{1}{2}\frac{\abs{ \lambda} \int
  e^{-\abs{y}}\abs{u}^2dy}{\int e^{-\abs{y}} \abs{u} \abs{\psi}dy}.
\]
Fixing this value of $\delta_0$, we can then find a value of $A$ sufficiently large such
that the inequality holds.

For the operator $\calL^{(o)}_-$, the proof is simpler, as we need
only observe that
\[
\inner{\overline{\calL}^{(o)}_- u_A}{u_A} < 0,
\]
contradicting the positivity of the operator.  This conludes our proof
of uniqueness of the solutions.

We now prove existence.  Again, this follows \cite{FMR,MR-Annals}.  We
have the two fundamental sets of solutions $\set{\rho_1, \rho_2}$ and
$\set{\tilde{\rho}_1, \tilde{\rho}_2}$.  $\rho_1$ and $\tilde{\rho}_1$
are asymptotically constant at $\infty$ and $-\infty$, respectively.
Note that these two must be linearly independent, for if they were
collinear, we would have a solution in $L^\infty$, solving $\calL u =
0$.  Hence,
\begin{align*}
  \abs{\rho_1(x)}&\leq K \abs{x}\quad\textrm{as $x\to-\infty$},\\
  \abs{\tilde{\rho}_1(x)}&\leq K \abs{x}\quad\textrm{as $x\to+\infty$}
\end{align*}
for some constant $K$.  We construct a Green's function from these two
to get the solution
\[
u(x) = \tilde{\rho}_1(x) \int_{x}^\infty \frac{\rho_1(s)f(s)}{W(s)}ds+
\rho_1(x)\int_{-\infty}^x \frac{\tilde{\rho}_1(s)f(s)}{W(s)}ds,
\]
where $W = \tilde{\rho}_1 \rho_1' - \tilde{\rho}_1' \rho_1$ is the
Wronskian.  The integrals converge and have the appropriate decay due
to the properties of $\rho_1$ and $\tilde{\rho}_1$, and our assumption
that $f$ is highly localized.

Finally, if $f$ is even, consider
\[
\begin{split}
  u(-x) &= \tilde{\rho}_1(-x) \int_{-x}^\infty
  \frac{\rho_1(s)f(s)}{W(s)}ds+ \rho_1(-x)\int_{-\infty}^{-x}
  \frac{\tilde{\rho}_1(s)f(s)}{W(s)}ds\\
  &= \tilde{\rho}_1(-x) \int_{-\infty}^x
  \frac{\rho_1(-s)f(s)}{W(-s)}ds+ \rho_1(-x)\int_x^\infty
  \frac{\tilde{\rho}_1(-s)f(s)}{W(-s)}ds\\
  & = {\rho}_1(x) \int_{-\infty}^x
  \frac{\tilde{\rho}_1(s)f(s)}{W(s)}ds+ \tilde{\rho}_1(x)\int_x^\infty
  \frac{{\rho}_1(s)f(s)}{W(s)}ds.
\end{split}
\]
Thus $u(x) = u(-x)$.  An analogous proof holds for $f$ odd.

\section{Numerical Methods}
\label{sec:numerics}
The software tools we use in our computations are the \matlab and
Fortran 90/95 implementations of an adaptive nonlinear collocation
algorithm discussed in \cite{shampine2006user,shampine2003singular,
  shampine2003solving}.  Though they are quite similar, we found the
Fortran implementation to be faster and more robust for solving the 3d
problems which require us to compute the ground state.  For
the 1d problems, where we have an explicit formula, the \matlab algorithm, \verb bvp4c  sufficed.  We use these tools to solve for the soliton,
compute the index functions, and solve the relevant boundary value
problems and associated inner products.

The codes used to perform these computations are available at \url{http://www.math.toronto.edu/simpson/files/spec_prop_code.tgz}.

\subsection{Singularities at the Origin}
\label{s:numerical_sing}
A useful feature of this algorithm is that it can handle boundary
value problems of the form
\[
\frac{d}{dr}\mathbf{y} = \frac{1}{r}S\mathbf{y} + \mathbf{f}(r,
\mathbf{y}),
\]
where $S$ is some constant coefficient matrix.  The $r^{-1}$
singularity naturally appears in the 3D problems due to the Laplacian.
The higher harmonics introduce a $r^{-2}$ singularity which can be
addressed by a change of variables.  Let $\calL^{(k)}$ denote one of
the operators applied to the $k$-th harmonic,
\[
\calL^{(k)} = -\frac{d^2}{dr^2} - \frac{d-1}{r}\frac{d}{dr} + \calV +
\frac{k(k+d-2)}{r^2}.
\]
Then if $W$ solves $\calL^{(k)} W = f$, where $f$ could be zero and
$W$ is non singular at the origin, then
\[
\lim_{r\to 0} r^{-k}W(r)
\]
is a nonzero constant.  This motivates the change of variable $W(r) =
r^k \widetilde W(r)$.  In terms of $\widetilde W(r)$, the equation
becomes
\[
r^k \widetilde\calL^{(k)} \widetilde W = f,
\]
where
\[
\widetilde\calL^{(k)} = -\frac{d^2}{dr^2} - \frac{d-1 + 2k
}{r}\frac{d}{dr} + \calV
\]
and $\widetilde W$ satisfies the condition $ \widetilde W' (0)=0$.  We
compute with $\widetilde\calL^{(k)} $ to get $\widetilde W$ and then
multiply by $r^k$.  When we compute indexes for these higher harmonic
operators, the other initial condition becomes $\widetilde W(0) = 1$.

\subsection{Artificial Boundary Conditions}
Another subtlety of the computations is the far field boundary
conditions.  The soliton, $R$, vanishes as $r \to +\infty$, but we
only compute out to some finite value, $r_{\max}$.  For simplicity, we
use the notation $r$ and $\rmax$ for both 1D and 3D.  To accommodate this, we
introduce an artificial boundary condition at $r_{\max}$, and then
check \emph{a postiori}, that it is consistent.    Thus, we must do
some asymptotic analysis.

We seek an asymptotic expansion for $R$, using \eqref{eqn:sol}
\begin{eqnarray*}
  (-\Delta + \lambda -f(R)) R = 0.
\end{eqnarray*}
As $r \to \infty$, we look for an expansion of the form
\begin{eqnarray}
  \label{eqn:asymp}
  e^{-\sqrt{\lambda} r} r^\gamma \sum_{n=0}^{\infty} c_n r^{-n}.
\end{eqnarray}
To extract the leading order behavior, we wish to find $\gamma$.
To this end, we have
\begin{align*}
  \frac{\partial}{\partial_r} (e^{-\sqrt{\lambda} r} r^\gamma) & =
  -\sqrt{\lambda} e^{-\sqrt{\lambda} r} r^\gamma + \gamma
  e^{-\sqrt{\lambda} r} r^{\gamma -1}, \\ 
  \frac{\partial^2}{\partial^2 r} (e^{-\sqrt{\lambda} r} r^\gamma) & =
  \lambda e^{-\sqrt{\lambda} r} r^\gamma - 2 \lambda \gamma
  e^{-\sqrt{\lambda} r} r^{\gamma - 1} + \gamma (\gamma -1)
  e^{-\sqrt{\lambda} r} r^{\gamma -2}. 
\end{align*}
Plugging \eqref{eqn:asymp} into \eqref{eqn:sol}, we see
\begin{eqnarray*} [ (-\lambda + \lambda) + \frac{(2 \sqrt{\lambda}
    \gamma + \sqrt{\lambda} (d-1))}{r} + O(r^{-2})] = 0.
\end{eqnarray*}
Hence, $\gamma = -(d-1)/2$ and the leading order behavior is
\begin{eqnarray}
  \label{spec:eqasym}
  r^{-\frac{d-1}{2}} e^{-\sqrt{\lambda} r}.
\end{eqnarray}
Therefore, as $r\to \infty$, 
\begin{equation}
  R(r) \approx \begin{cases}
    R_\star\exp{{-\sqrt{\lambda} r}} & d=1,\\
    R_\star \frac{1}{r}\exp{{-\sqrt{\lambda} r}} & d=3.
  \end{cases}
\end{equation}
From this, we develop the Robin boundary condition,
\[
\lim_{r\to\infty}\frac{R(r)}{R'(r)}\to \begin{cases}
  -\frac{1}{\sqrt\lambda} & d=1,\\
  -\frac{r}{1+ \sqrt\lambda r} & d=3,
\end{cases}
\]
which we formulate as
\begin{subequations}
  \begin{equation}
    \label{eq:bc_R_1d}
    R(r_{\max}) + \frac{1}{\sqrt\lambda}R'(r_{\max})= 0\quad \textrm{for $d=1$}
  \end{equation}
  and
  \begin{equation}
    \label{eq:bc_R_3d}
    R(r_{\max}) + \frac{r_{\max}}{1+ \sqrt{\lambda}r_{\max}}R'(r_{\max})= 0\quad \textrm{for $d=3$}
  \end{equation}
\end{subequations}
assuming we have taken $\rmax$ sufficiently large.  For our
computations, aside from noting that the solver algorithm ends without
errors, we have two {\it a postiori} checks available.  The first is to
verify that we have, in fact, computed the ground state.  Plotting the
computed $R$ on both a linear and a log scale in Figure
\ref{fig:groundstate1_3d_cubic} we verify that $R$ is a hump shaped
monotonically decaying function.  It also has the anticipated $r^{-1} e^{-r}$ decay rate.

\begin{figure}
  \centering

  \subfigure[3d Cubic NLS]{
    \includegraphics[width=2.5in]{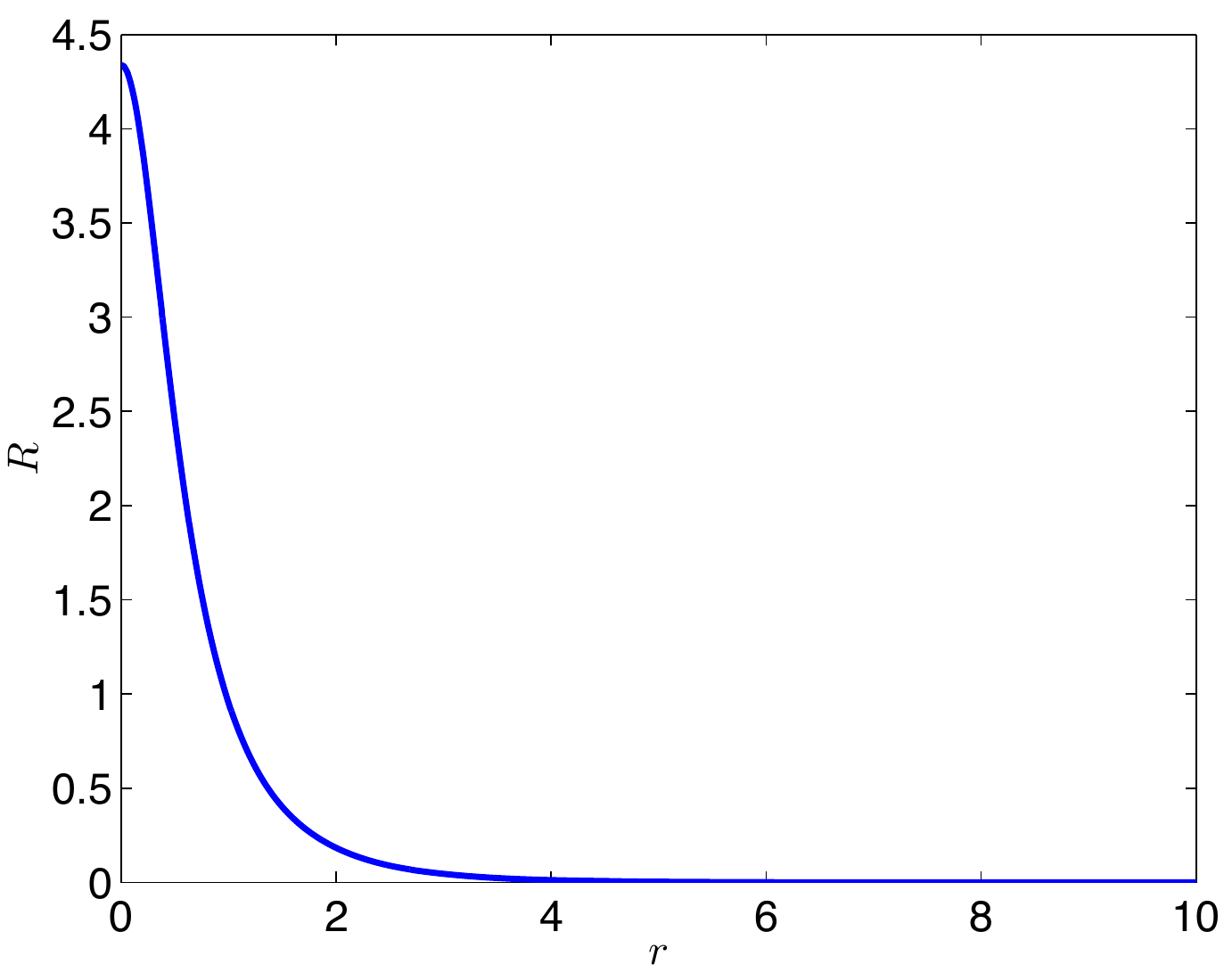}
  }\subfigure[3d Cubic NLS Decay Rate]{
    \includegraphics[width=2.5in]{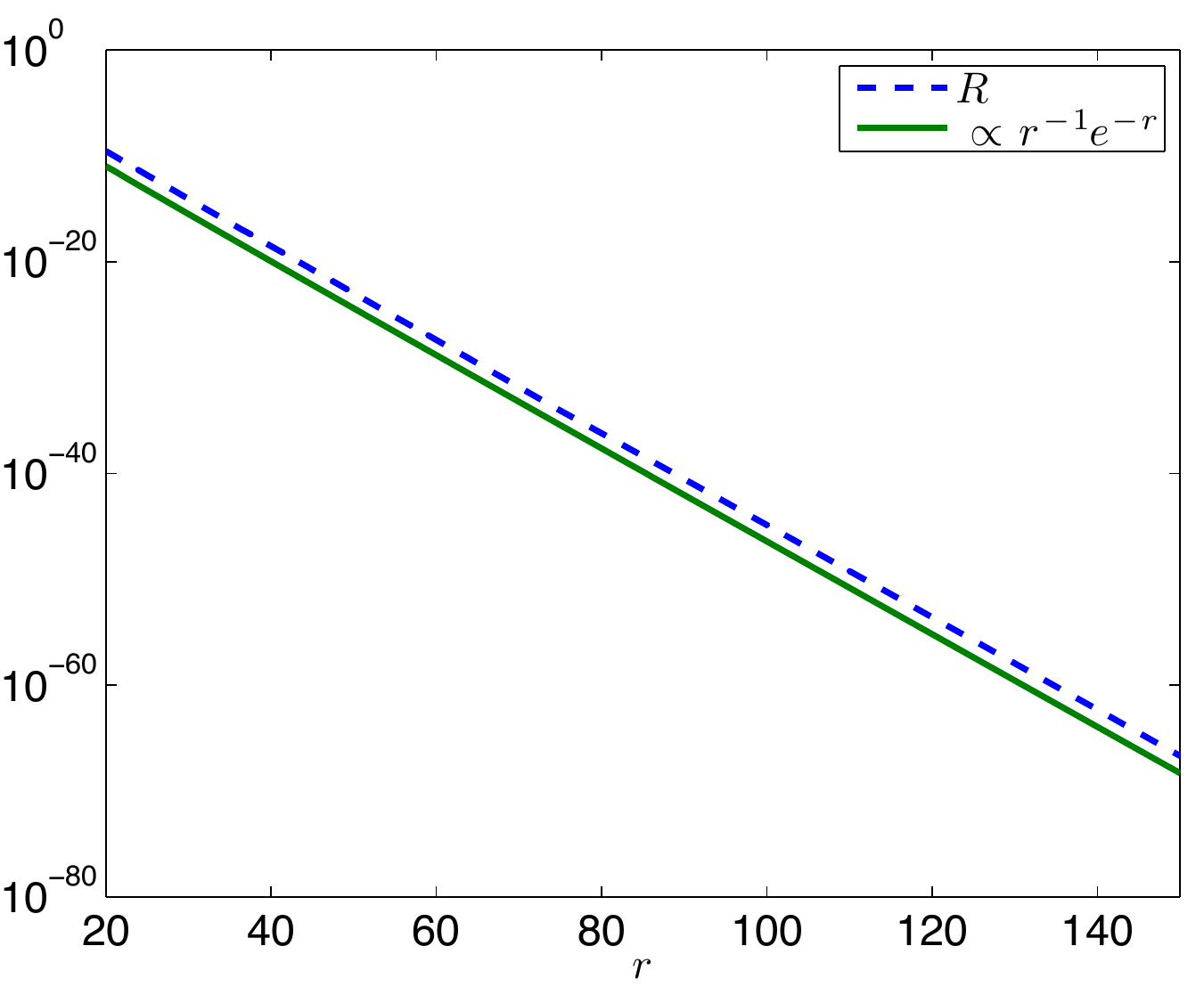}
  }

  \caption{The numerically computed ground state for the different
    problems.  Computed on the indicated domain, with a tolerance of
    $10^{-13}$ while assessing the indexes of the operators restricted
    to even functions for $d=1$ and the zeroth harmonic for $d=3$.
    The computed solitons are monotonic and decay at the anticipated
    rate.}
  \label{fig:groundstate1_3d_cubic}
\end{figure}

The second thing that can be checked is that the numerically computed
$R$ satisfies, asymptotically, the artificial boundary condition
\eqref{eq:bc_R_3d}.  To do this, we plot $R(r) +
\frac{r}{1+\sqrt{\lambda} r} R'(r)$ and observe that it vanishes as $r
\to \infty$.  As can be seen in Figure \ref{fig:groundstate_bc_chk},
the mismatch in the artificial boundary condition is small and
monotonically decaying in $r$.

\begin{figure}
  \begin{center}
    \centering

     \includegraphics[width=2.5in]{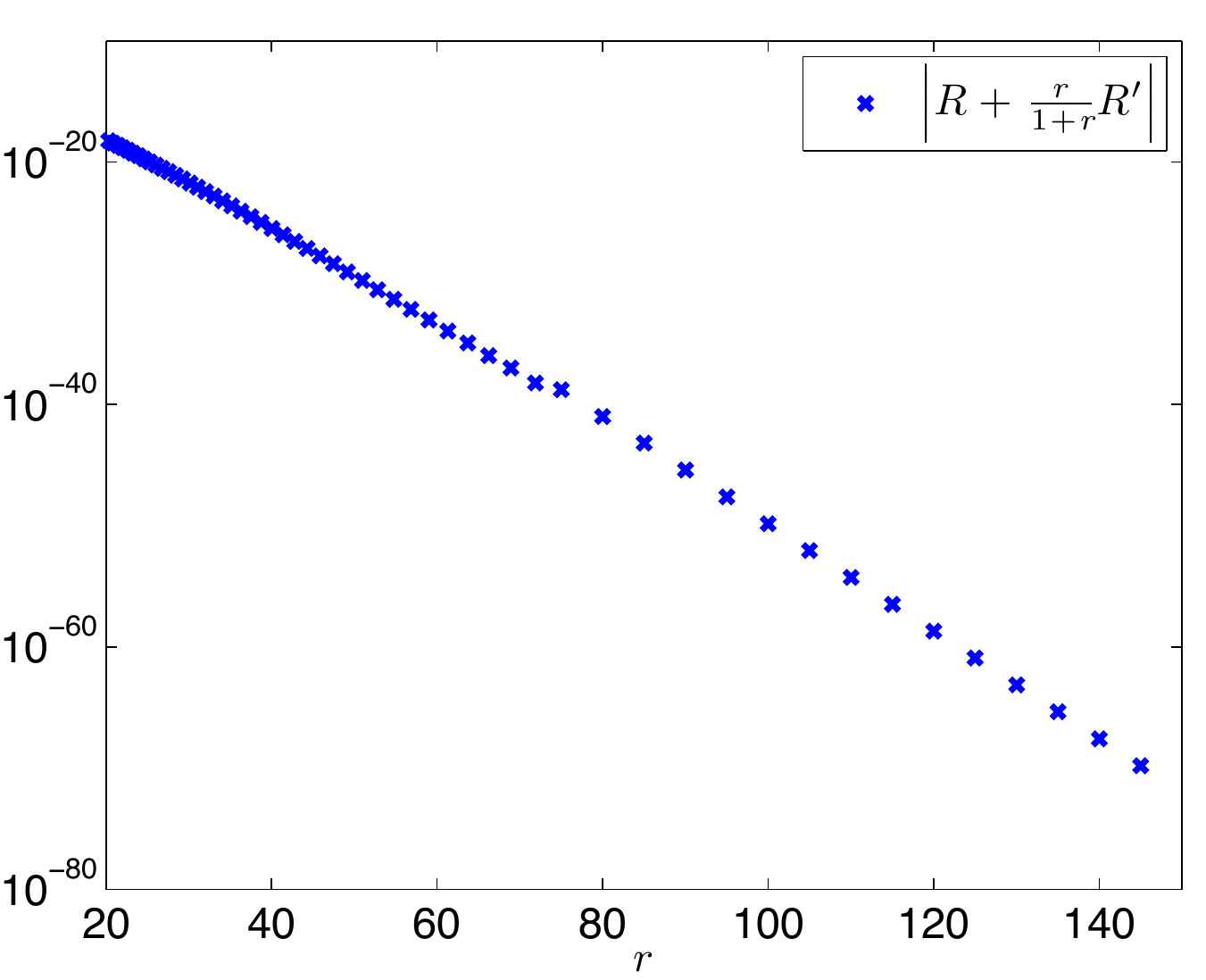}

    \caption{As $r\to \infty$, $R$ asymptotically satisfies $-\Delta
      R + \lambda R \sim 0$. The figures indicate
      that the relevant domain, which is different for different
      problems, is sufficiently large that artificial boundary
      conditions are good approximations.  From the same computation as
      in Figure \ref{fig:groundstate1_3d_cubic}. }
    \label{fig:groundstate_bc_chk}
  \end{center}
\end{figure}

To compute the off axis eigenstates $\vec{\phi}$ for the 3D cubic
problem with eigenvalue $\sigma>0$, we rely on the relationship
\[
L_- L_+ \phi_1 = -\sigma^2 \phi_1.
\]
Asymptotically, this is the free equation
\[
(-\Delta + 1)^2 \phi_1 = - \sigma^2 \phi_1.
\]
Seeking a radially symmetric solution and using similar expansion
techniques as in the case of the soliton, we will find that
\begin{equation}
\phi_1(r) \approx c_1 r^{-(d-1)/2}e^{-r\sqrt{1+ i \sigma}} + c_2 r^{-(d-1)/2}e^{-r\sqrt{1- i \sigma}} 
\end{equation}
as $r\to \infty$.  

Let
\begin{align}
  \theta &\equiv \frac{1}{2}\arctan(\sigma),\\
  \rho &\equiv (1+\sigma^2)^{1/4}.
\end{align}
In 1D, we can construct an artificial boundary condition that, as $r \to \rmax$,
\begin{gather}
  \phi_1' + \rho \cos(\theta) \phi_1 + \rho \sin(\theta) \phi_2=0,\\
  \phi_2' + \rho \cos(\theta) \phi_2 - \rho \sin(\theta) \phi_1=0.
\end{gather}
Analogously, in 3D, as $r\to \rmax$,
\begin{gather}
  \phi_1' + \rho \cos(\theta) \phi_1 +\frac{1}{r} \phi_1 + \rho \sin(\theta) \phi_2=0,\\
  \phi_2' + \rho \cos(\theta) \phi_2 +\frac{1}{r} \phi_2 - \rho
  \sin(\theta) \phi_1=0.
\end{gather}
As in the case of the soliton, we can verify that the solutions have
the appropriate shape, decay as expected, and satisfy the artificial
boundary conditions.  The shape and decay are plotted in Figure
\ref{f:unstable_3d_cubic}.  The functions rapidly reach machine
precision.  If we zoom in on the unstable modes, as in Figure \ref{f:unstable_3dcubic_zoom},
we can see the periodic structure.  However, as is suggested by these
figures, once $\phi_j$'s are sufficiently small, $\lesssim
O(10^{-12})$, this fine structure degrades.  Fortunately, this
numerical error is sufficiently small as to not impact our
computations.  The artificial boundary condition plot appears in
Figure \ref{fig:ip_bcs_3d_cubic}.

\begin{figure}
  \centering

  \subfigure[Unstable Mode of 3d Cubic NLS]{
    \includegraphics[width=2.5in]{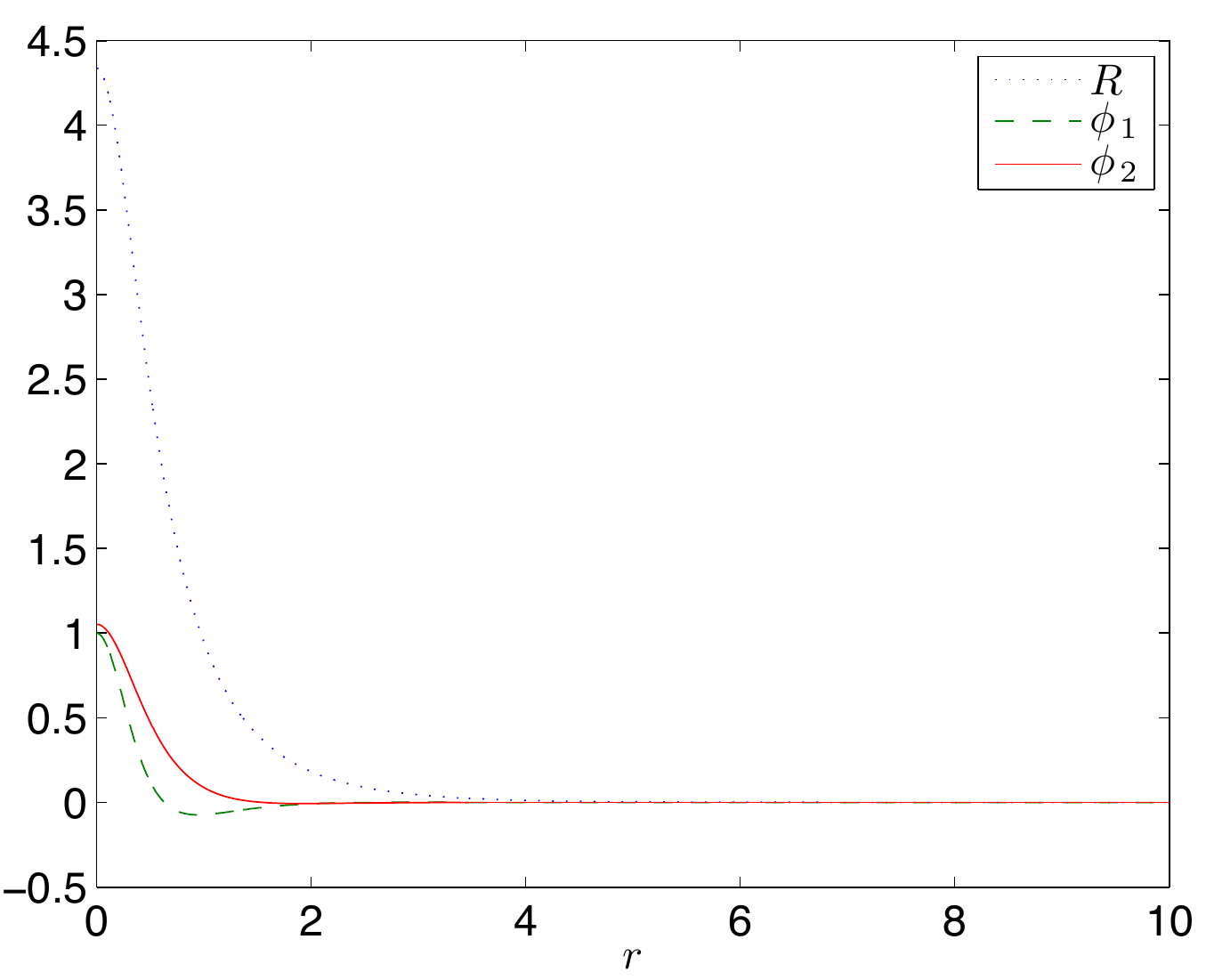}
  }\subfigure[Unstable Mode 3d Cubic NLS Decay Rate]{
    \includegraphics[width=2.5in]{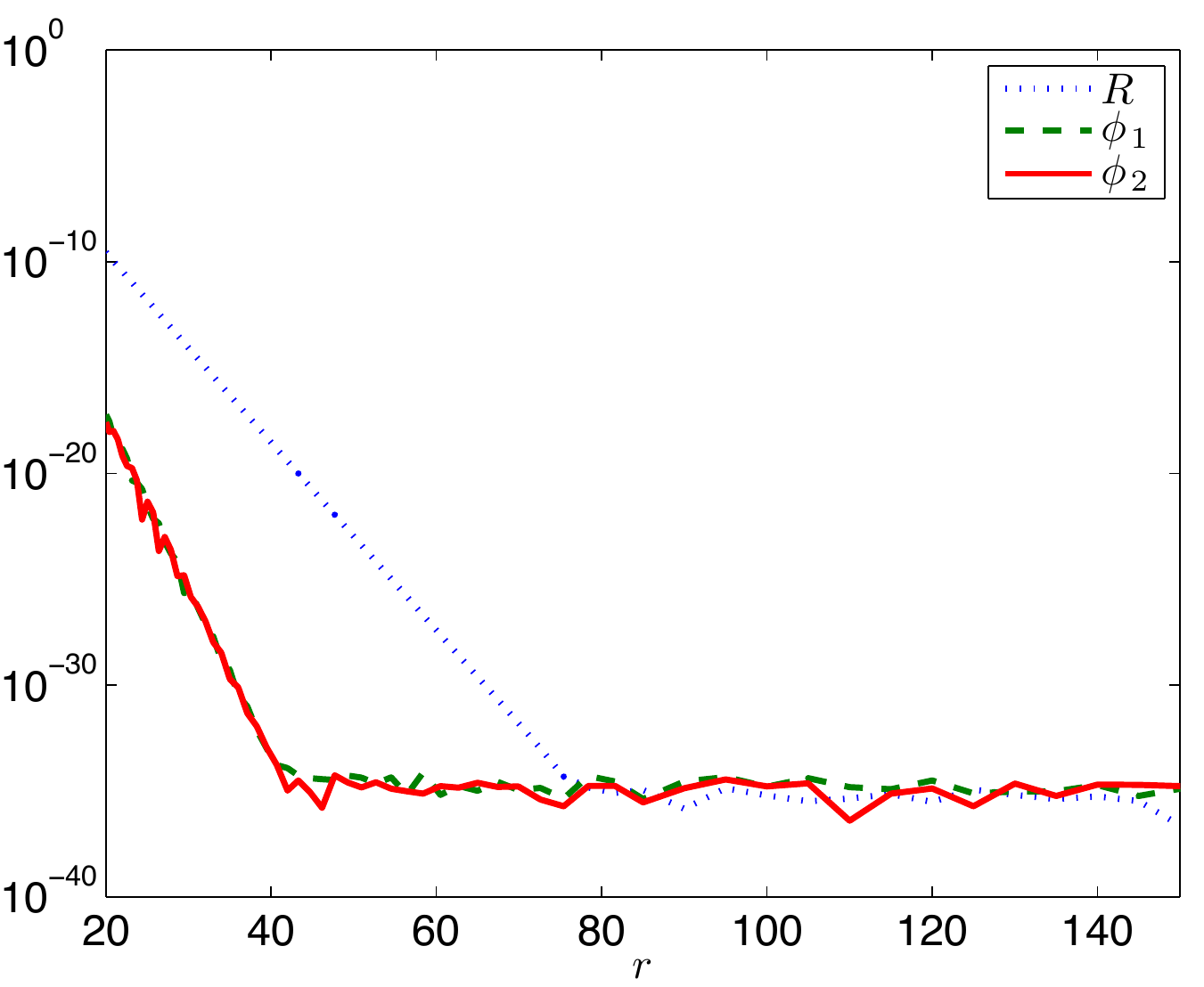}
  }

  \caption{The numerically computed ground state and the off axis
    unstable mode of 3d cubic NLS.  Computed on the indicated domain,
    with a tolerance of $10^{-12}$.}
  \label{f:unstable_3d_cubic}
\end{figure}

\begin{figure}
  \centering
  \includegraphics[width=2.5in]{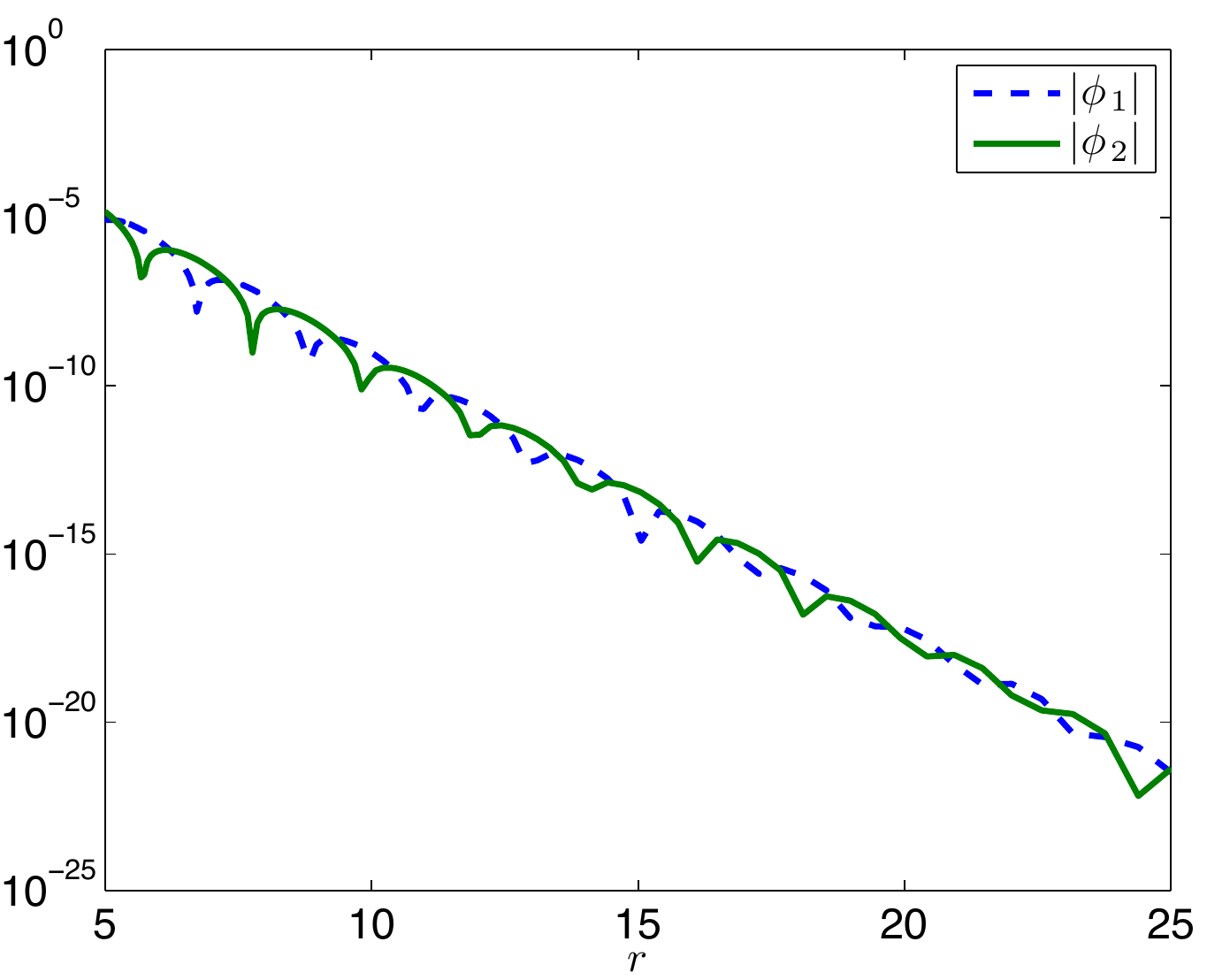}
  \caption{Zooming in on the components of the numerically computed
    unstable mode, $\vec{\phi}$, we see the periodic structure of the
    solution.}
  \label{f:unstable_3dcubic_zoom}
\end{figure}

Another place where we use artificial boundary conditions is in
solving the various boundary value problems for 
$U_j^{(k)}$ or $Z_j^{(k)}$.  Tthe rapid decay of the soliton
leads to the functions satisfying the free equation, for $d>1$, 
\[
-\Delta q +k(k+d-2)/r^{d-2+k} q = 0,
\]
where $q$ is one of these functions.  In this
region, the function must asymptotically be like
\eqref{eq:Zk_asympt}. Of course, we work with the variables
$\widetilde U_j^{(k)}$ and $\widetilde Z_j^{(k)}$.  Since these vanish
as $r\to \infty$, they asymptotically behave as
\[
q \propto r^{2-d-2k}.
\]
Thus, we have the artificial boundary conditions, valid
for all harmonics and $d>2$,
\begin{subequations}
  \begin{gather}
    \widetilde{U}_j^{(k)}(\rmax) + \frac{r}{d-2+2k}
    \frac{d}{dr}\widetilde{U}_j^{(k)}(\rmax) = 0,\\
    \widetilde{Z}_j^{(k)}(\rmax) + \frac{r}{d-2+2k}
    \frac{d}{dr}\widetilde{Z}_j^{(k)}(\rmax) =0.
  \end{gather}
\end{subequations}
Figure \ref{fig:ip_bcs_3d_cubic} shows that these artificial boundary
conditions are asymptotically satisfied.

\begin{figure}
  \begin{center}
    \subfigure[$\calL_+^{(0)}$]{
      \includegraphics[width=2.5in]{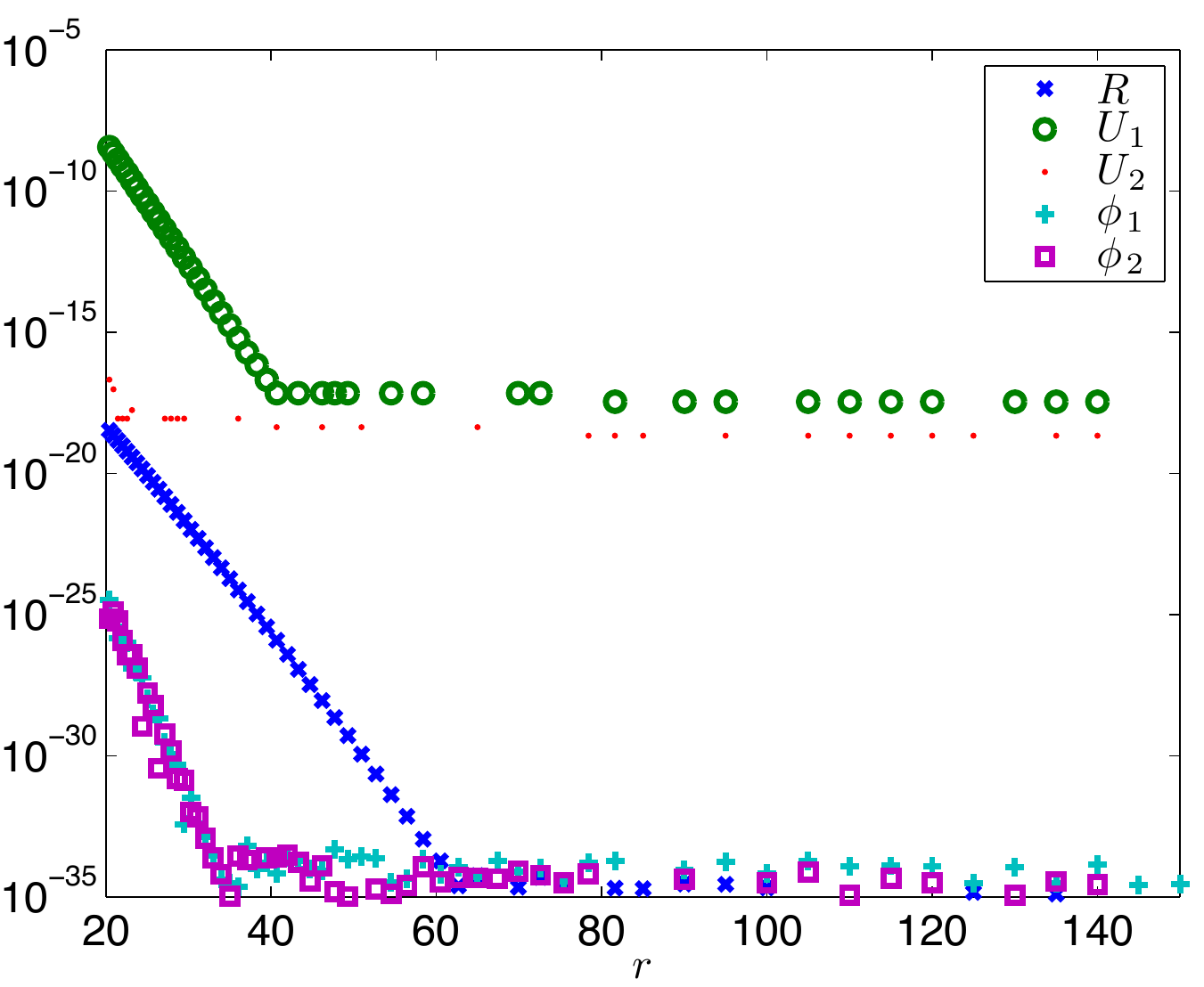}
    } \subfigure[ $\calL_-^{(0)}$]{
      \includegraphics[width=2.5in]{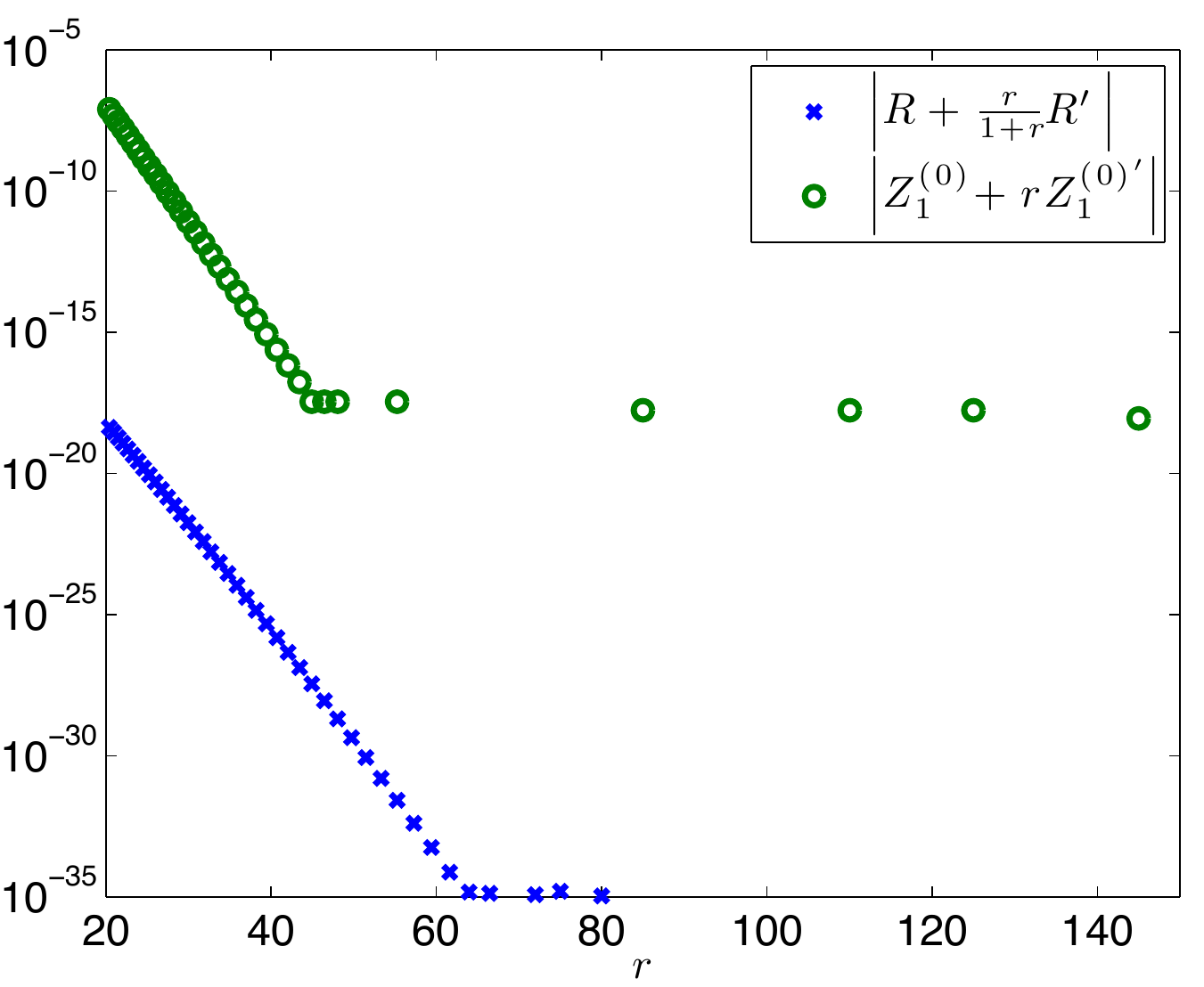}
    } \subfigure[ $\calL_+^{(1)}$]{
      \includegraphics[width=2.5in]{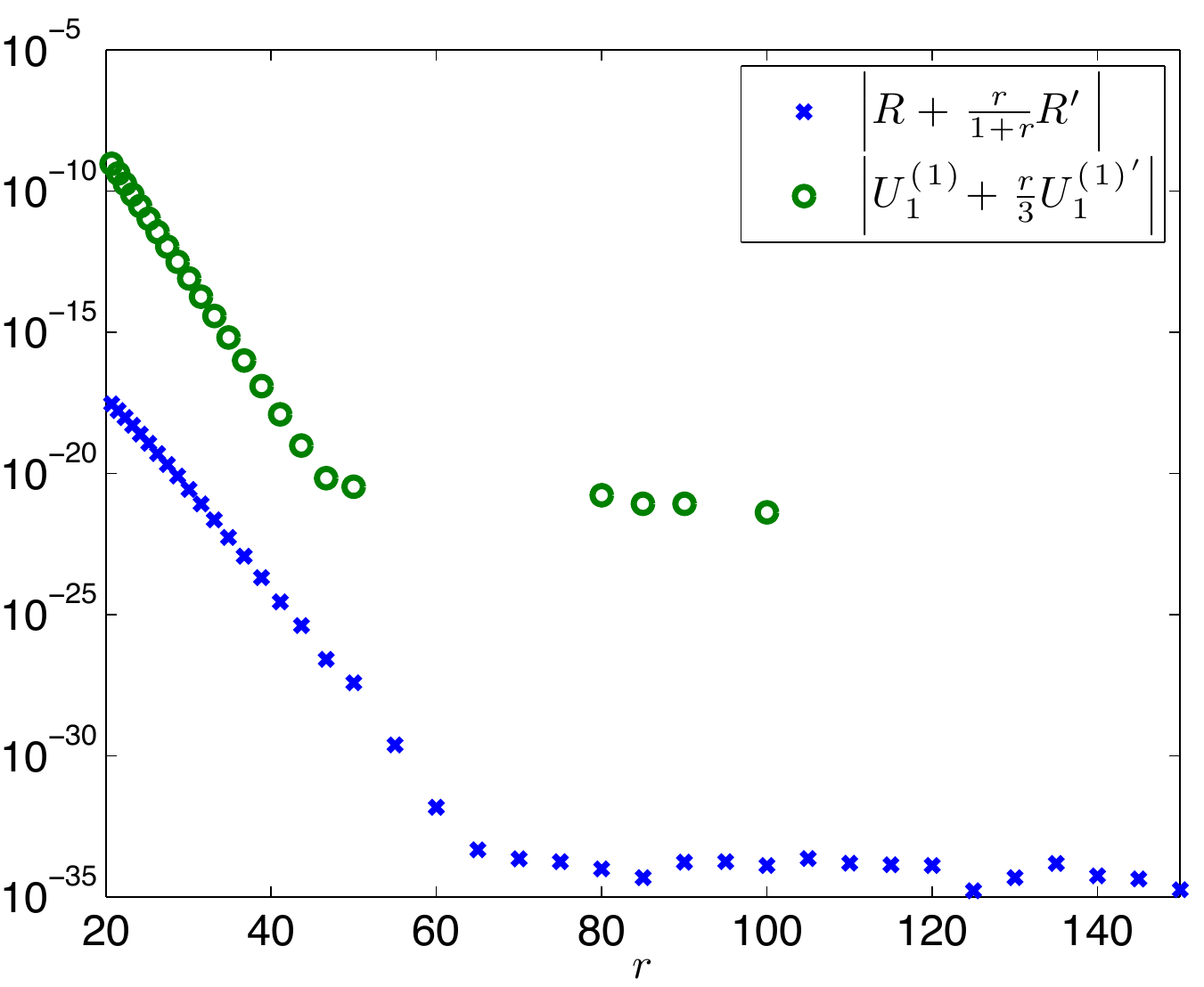}
    }
    \caption{As $r\to \infty$, our functions satisfy the
      asymptotically satisfy the free, linear equations for the
      different parts of the 3d cubic problem. The figures indicate
      that the relevant domain, which is different for different
      problems, is sufficiently large that artificial boundary
      conditions are good approximations.}
    \label{fig:ip_bcs_3d_cubic}
  \end{center}
\end{figure}

For our $d=1$ computations, we have the artificial boundary conditions 
\begin{equation}
\ddx U_j^{(e/o}(\xmax) =0, \quad \ddx Z_j^{(e/o}(\xmax) =0.
\end{equation}
We can similarly check that the inner products are asymptotically
constant and that our computed functions asymptotically satisfy the
artificial boundary conditions.  See Figures
\ref{f:ip_1d_supercrit_bc} and \ref{f:ip_1d_crit_bc}.

\begin{figure}
  \begin{center}
    \subfigure[$\calL_\pm^{(e)}$ with
    $\sigma=2.1$.]{\includegraphics[width=2.1in]{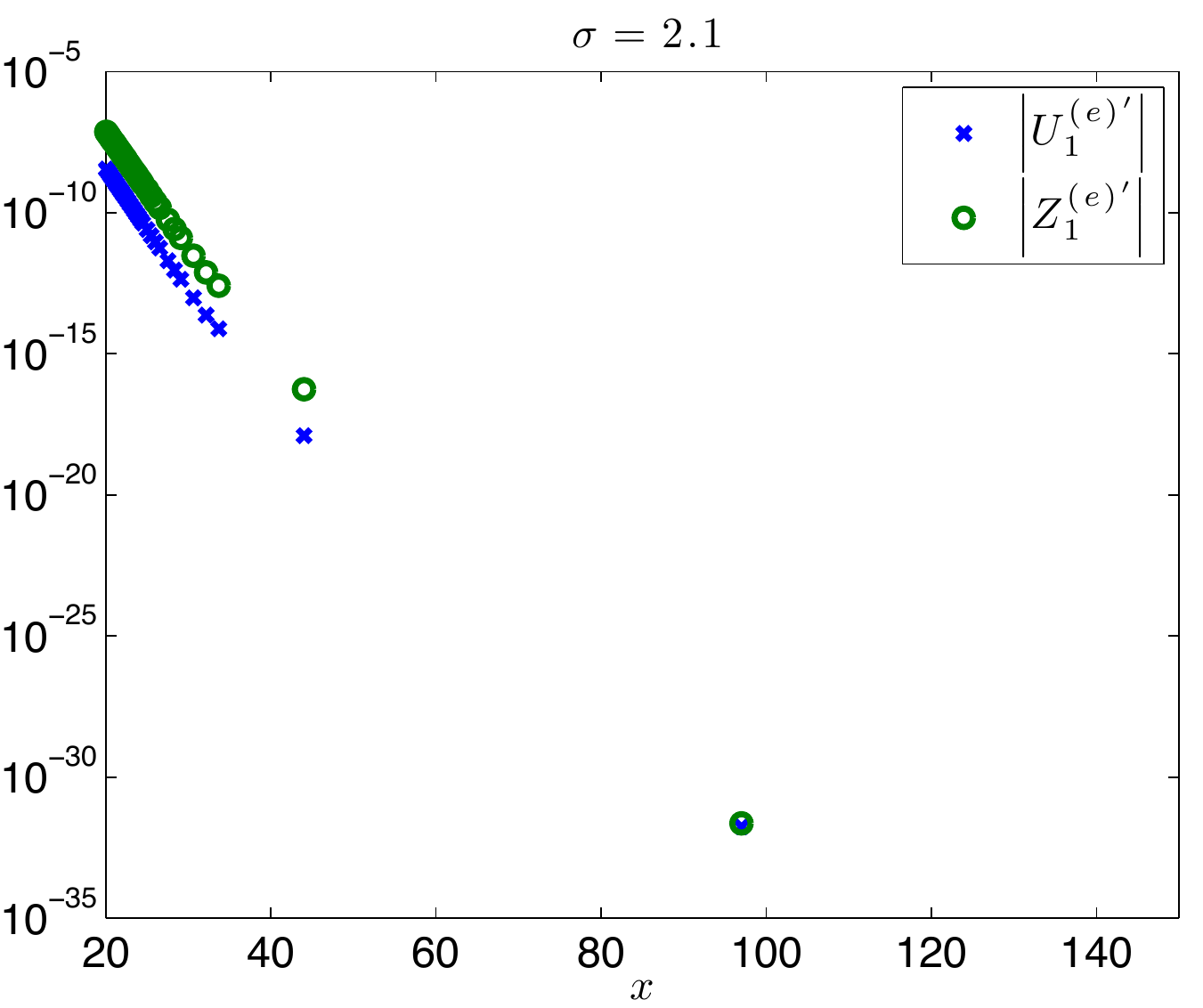}}
    \subfigure[$\calL_+^{(o)}$with
    $\sigma=2.1$.]{\includegraphics[width=2.1in]{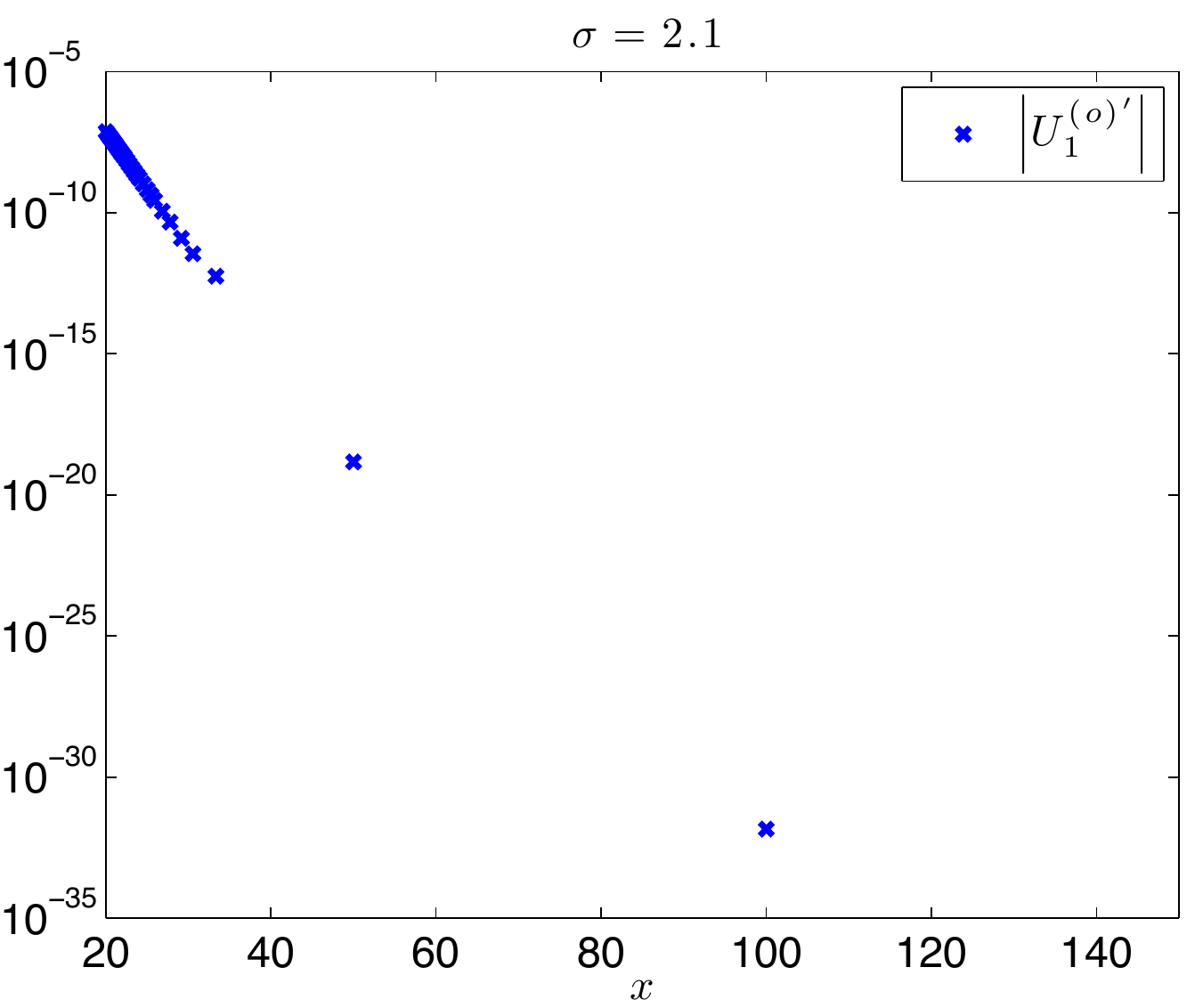}}

    \subfigure[$\calL_\pm^{(e)}$ with
    $\sigma=2.5$.]{\includegraphics[width=2.1in]{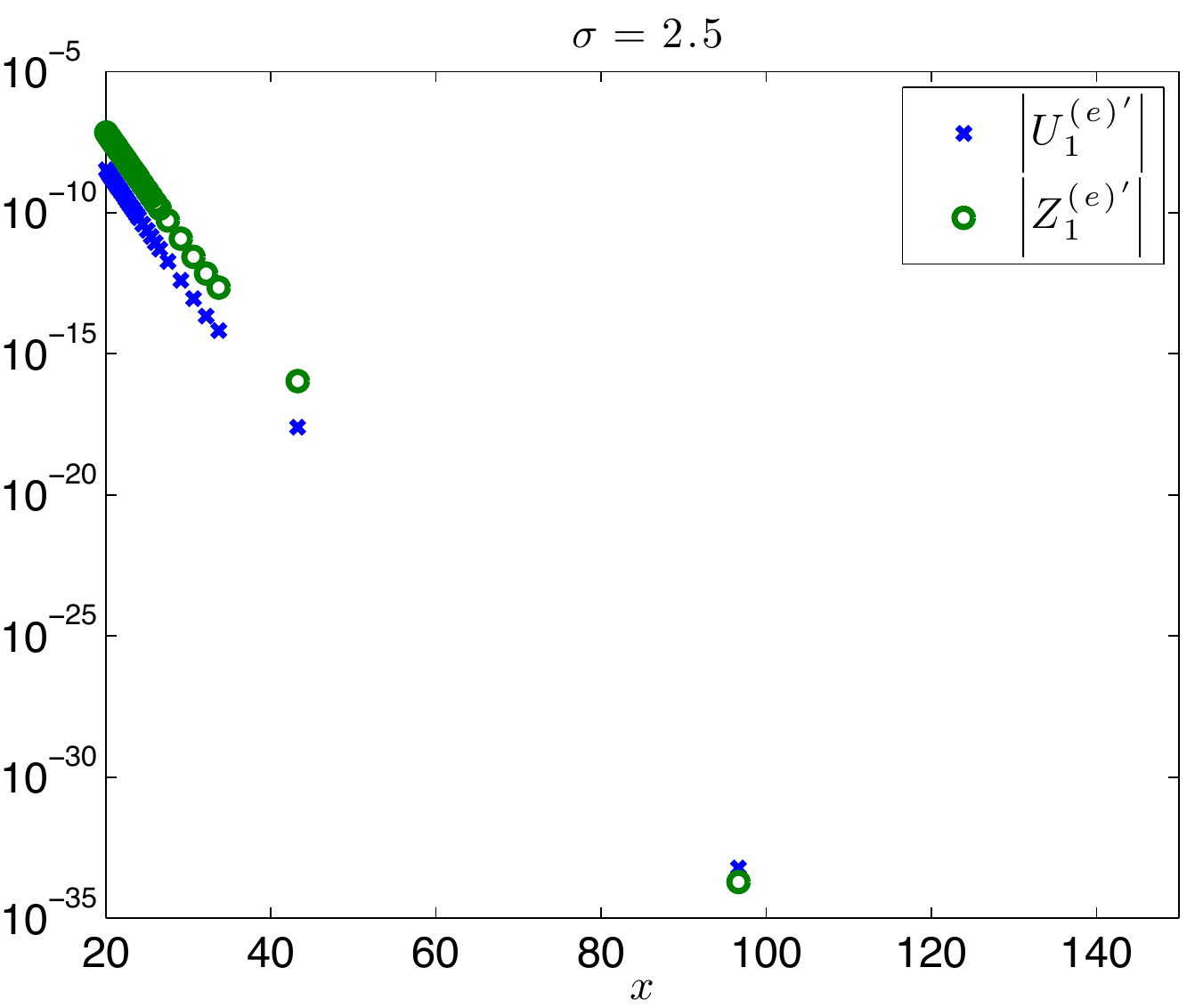}}
    \subfigure[$\calL_+^{(o)}$with
    $\sigma=2.5$.]{\includegraphics[width=2.1in]{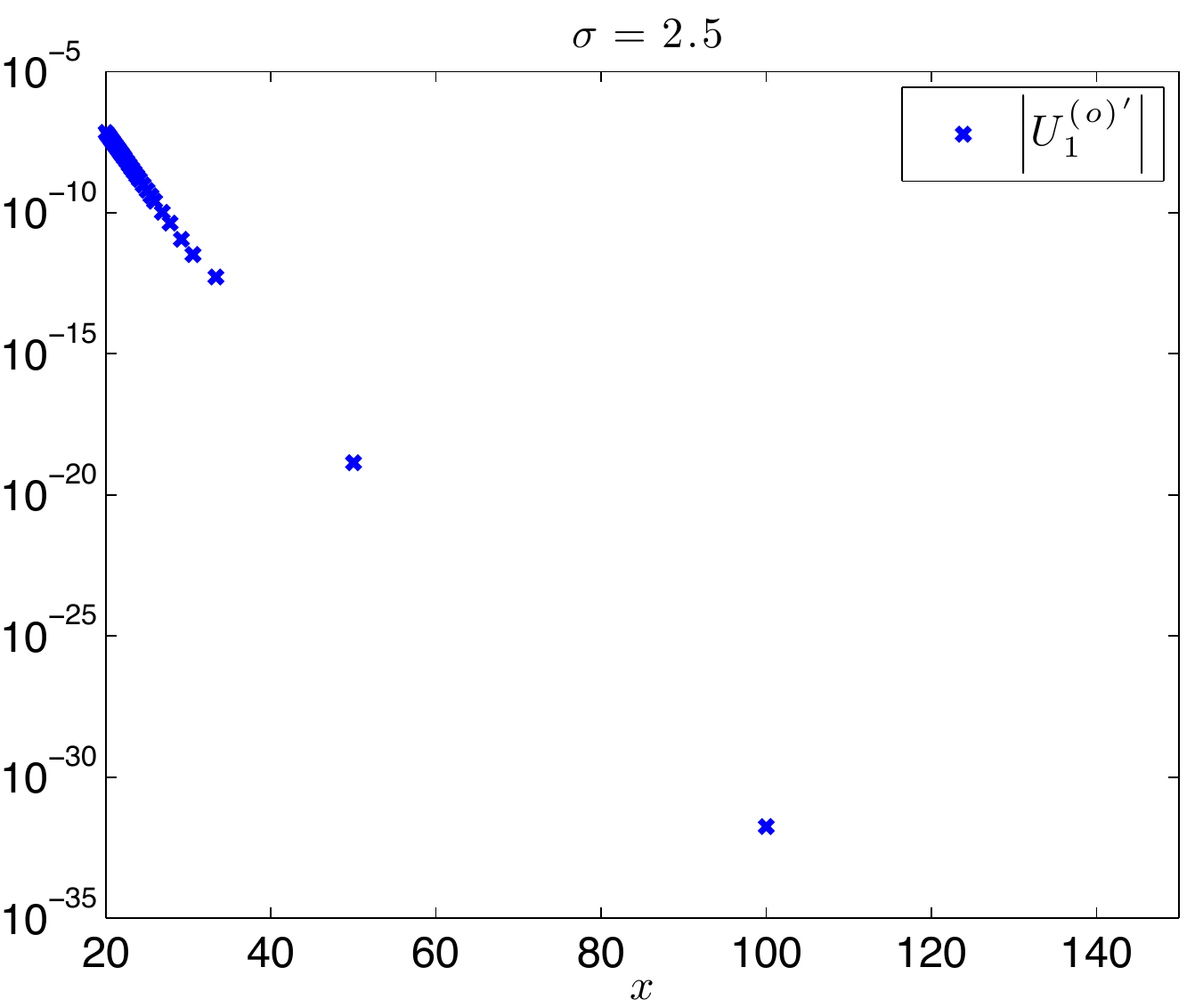}}

    \subfigure[$\calL_\pm^{(e)}$ with
    $\sigma=3.0$.]{\includegraphics[width=2.1in]{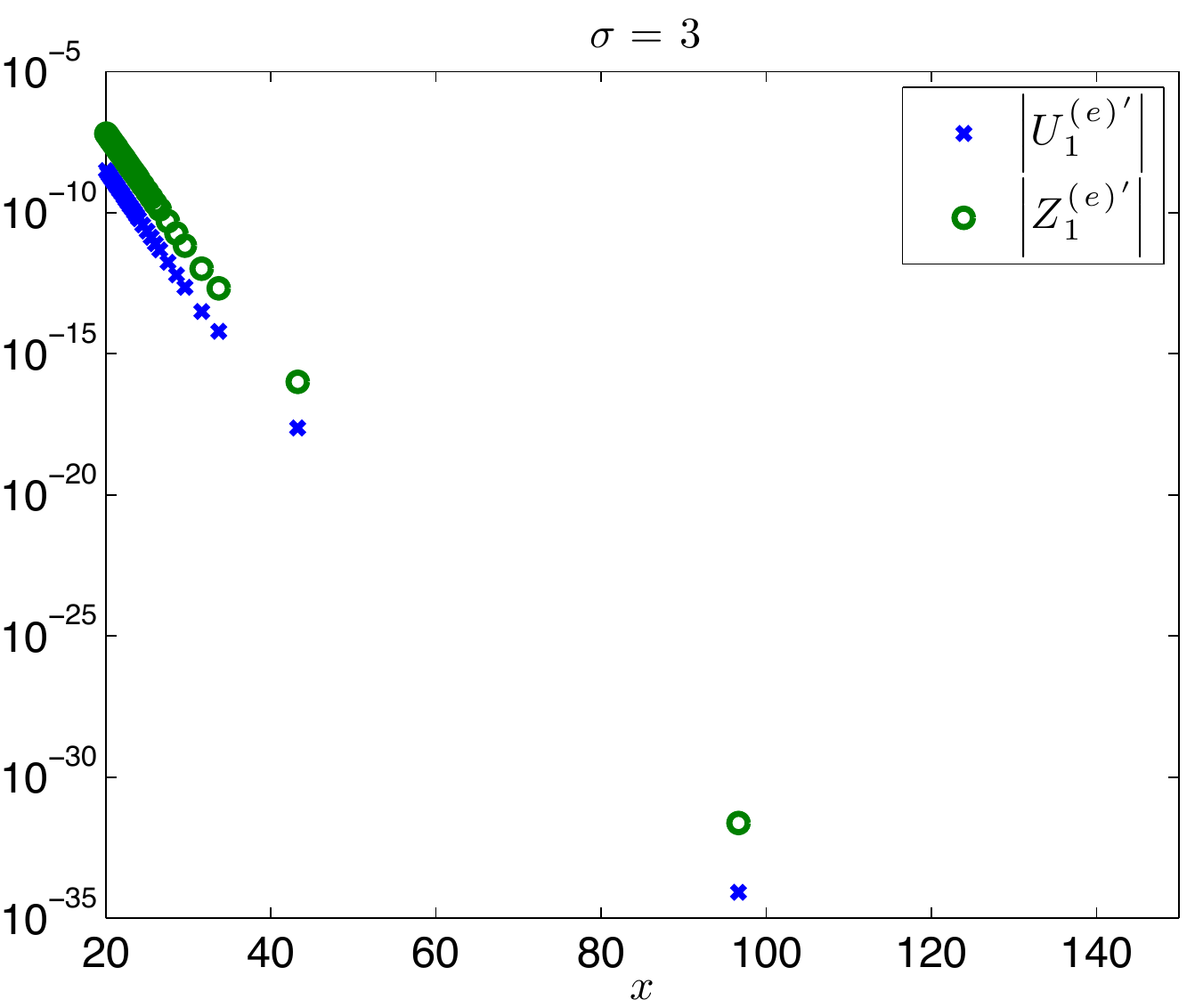}}
    \subfigure[$\calL_+^{(o)}$with
    $\sigma=3.0$.]{\includegraphics[width=2.1in]{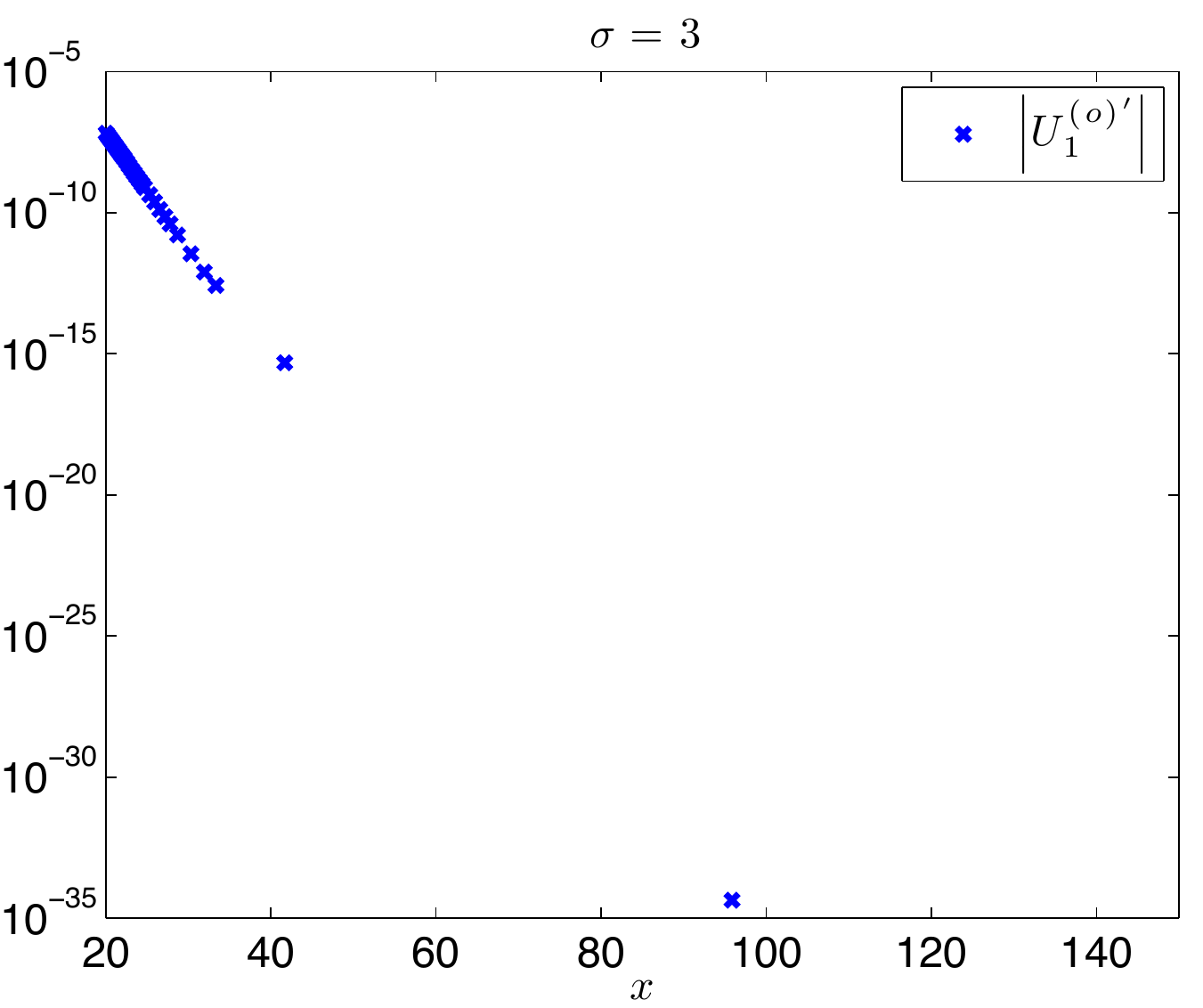}}
    \caption{As $r\to \infty$, our functions asymptotically satisfy
      the free equations for the various 1d supercritical problems.}
    \label{f:ip_1d_supercrit_bc}
  \end{center}
\end{figure}

\clearpage

\begin{figure}
  \begin{center}
    \subfigure[$\calL_\pm^{(e)}$ in the critical case with the {\it
      natural} orthogonality
    conditions.]{\includegraphics[width=2.1in]{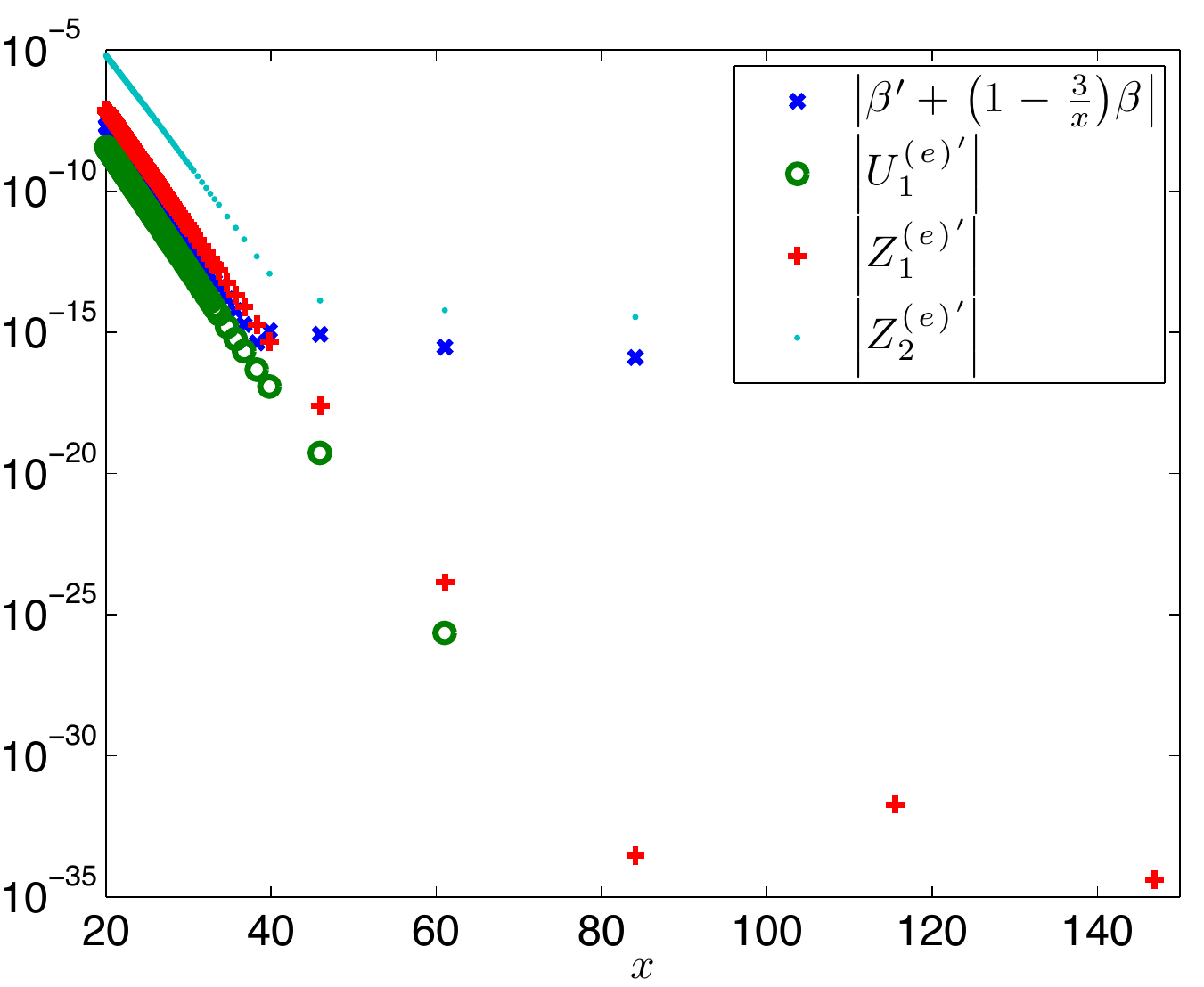}}
    \subfigure[ $\calL_+^{(o)}$ in the critical
    case]{\includegraphics[width=2.1in]{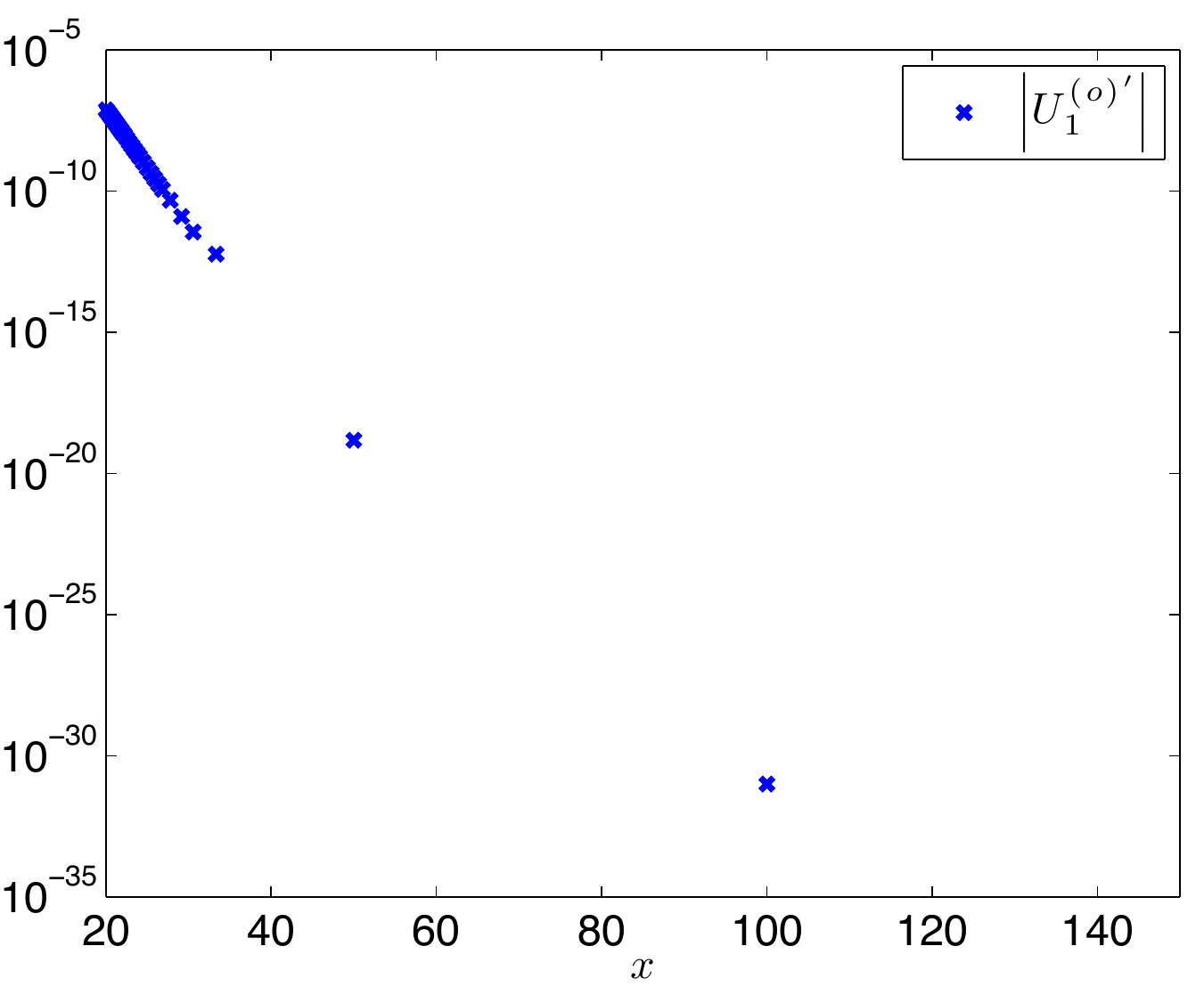}}

    \subfigure[$\calL_-^{(e)}$ in the critical case with the
    alternative orthogonality
    condition.]{\includegraphics[width=2.1in]{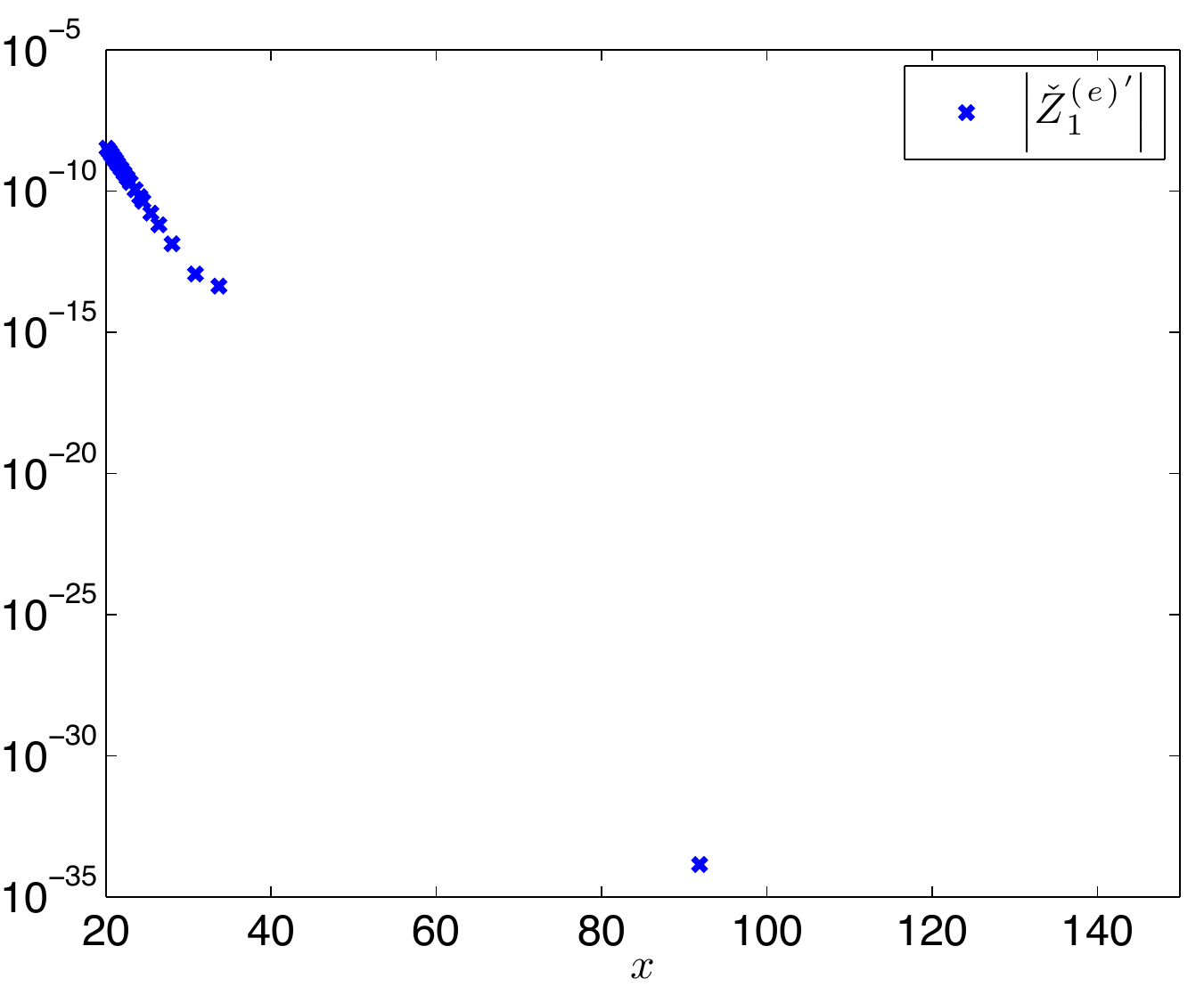}}
    \subfigure[$\calL_-^{(e)}$ in the critical case with the FMR
    orthogonality
    conditions.]{\includegraphics[width=2.1in]{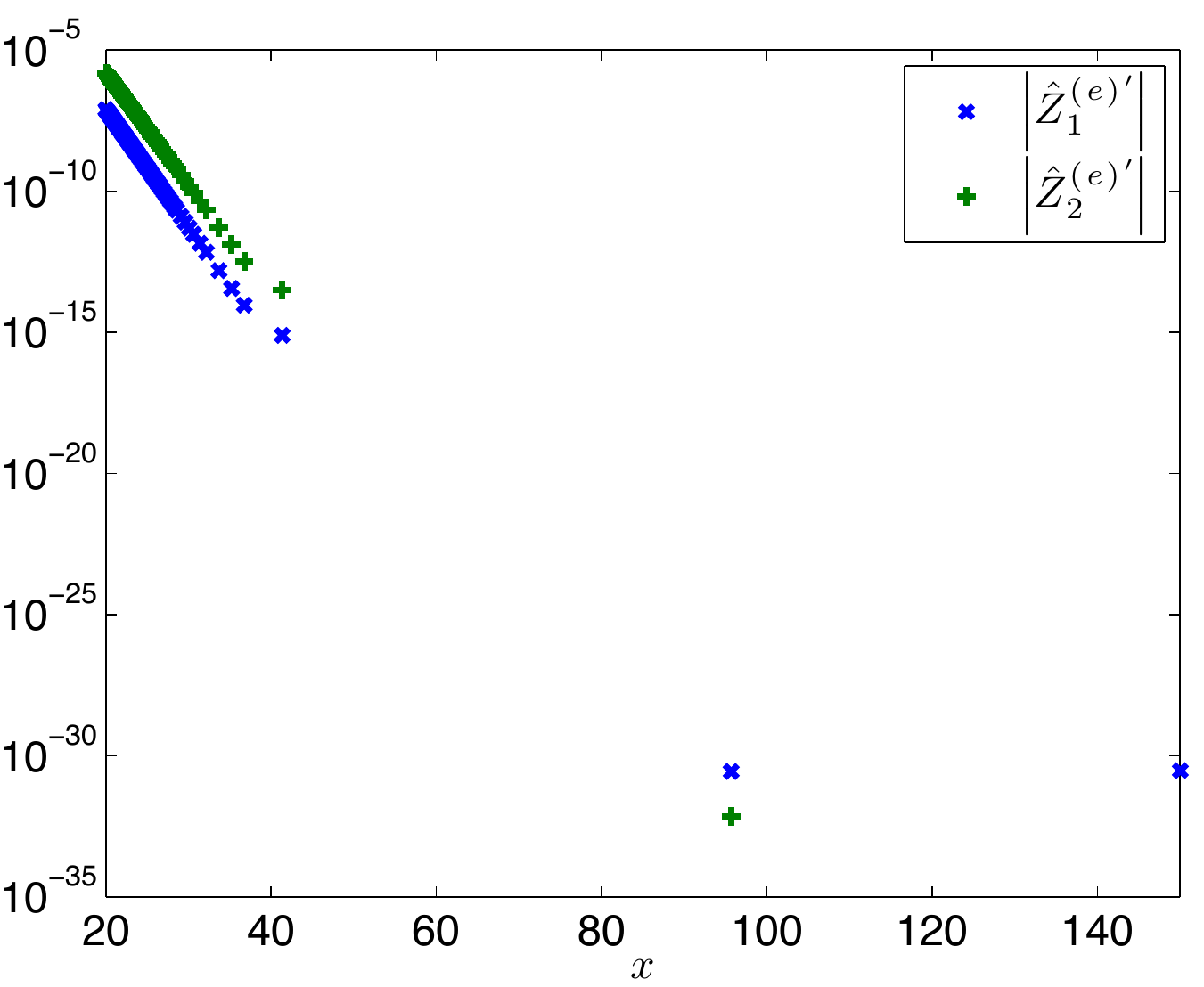}}
    \caption{As $r\to \infty$, our functions asymptotically satisfy
      the free equations for the 1d critical problem with various
      orthogonality conditions.}
  \end{center}
\label{f:ip_1d_crit_bc}
\end{figure}

The last place we make use of artificial boundary conditions is in
solving $L_+\beta = - x^2 R$ for the critical problem in Section
\ref{s:1d_crit}.  Using the same procedure as above, one will find
that $\beta \propto x^3 e^{-x}$ form which the artificial boundary
condition
\begin{equation}
\label{e:rho_abc}
\beta'(r) + \paren{1- \frac{3}{r}} \beta (r) = 0
\end{equation}
can be constructed.

\subsection{Computation of the Indexes}
In computing the indexes of operators $\widetilde\calL_\pm^{(k)}$
(from which we recover the indexes of $\calL_\pm^{(k)}$), we
simultaneously solve the mixed boundary value problem/initial value
problems
\begin{gather}
  -\Delta R + \lambda R - g(\abs{R}^2)R =0, \quad R'(0) = 0,\quad\textrm{\eqref{eq:bc_R_1d} or \eqref{eq:bc_R_3d}},\\
  \widetilde\calL_+^{(k)} \widetilde U^{(k)} = 0, \quad \widetilde
  U^{(k)}(0)=1, \quad \frac{d}{dr}\widetilde U^{(k)}(0)=0,\\
  \widetilde\calL_-^{(k)} \widetilde Z^{(k)} = 0, \quad \widetilde
  Z^{(k)}(0)=1, \quad \frac{d}{dr}\widetilde Z^{(k)}(0)=0.
\end{gather}
For the 3d cubic equation, this is the complete set of equations;
$\lambda = 1$ and $f(s) = s$.  The analogous computations are made in
$d=1$.  In addition to verifying that the soliton was adequately
computed, we can check, {\it a postiori}, that the index functions,
$U^{(k)}$ and $Z^{(k)}$ asymptotically satisfy the free equation.
This was the shown in the index figures of Sections \ref{sec:index_computations} and \ref{s:1d_idx},
where we checked the constants.

\subsection{Computation of the Inner Products}
As in the computation of the indexes, we similarly solve mixed
boundary value/initial value problem for $R$, the $U_j^{(\alpha)}$ and
$Z_j^{(\alpha)}$, and the inner products

In computing the inner products, we introduce the dependent variables
$\kappa_j^{(\alpha)}(r)$ and $\gamma_j^{(\alpha)}(r)$, where
\begin{align}
  \frac{d}{dr}\kappa_j^{(\alpha)}(r) &= \calL_+^{(\alpha)} U_{\ell_1}^{(\alpha)} U_{\ell_2}^{(\alpha)} r^{d-1}, \quad \kappa_j^{(\alpha)}(0)=0,\\
  \frac{d}{dr}\gamma_j^{(\alpha)}(r) &= \calL_-^{(\alpha)}
  Z_{\ell_1}^{(\alpha)} Z_{\ell_2}^{(\alpha)} r^{d-1}, \quad \gamma_j^{(\alpha)}(0)
  =0
\end{align}
for $\ell_1$ and $\ell_2$ the appropriate indexes. Clearly,
\begin{align*}
\lim_{r\to \infty} \kappa_j^{(\alpha)}(r) &= K_j^{(\alpha)},\\
\lim_{r\to \infty} \gamma_j ^{(\alpha)}(r) &= J_j^{(\alpha)}.
\end{align*}
We approximate the inner products by computing to $\rmax$.  As
demonstrated in Figures \ref{fig:ips_3d_cubic},
\ref{f:ip_1d_supercrit}, \ref{f:ip_1d_crit} these converge 
rapidly and are essentially constant in the region of the free
equation.  This is entirely consistent with the exponential decay of
the soliton and functions related to it, such as its derivative.

\begin{figure}
  \begin{center}
    \subfigure[Inner Products for $\calL_+^{(0)}$ functions for 3D
    Cubic]{
      \includegraphics[width=2.5in]{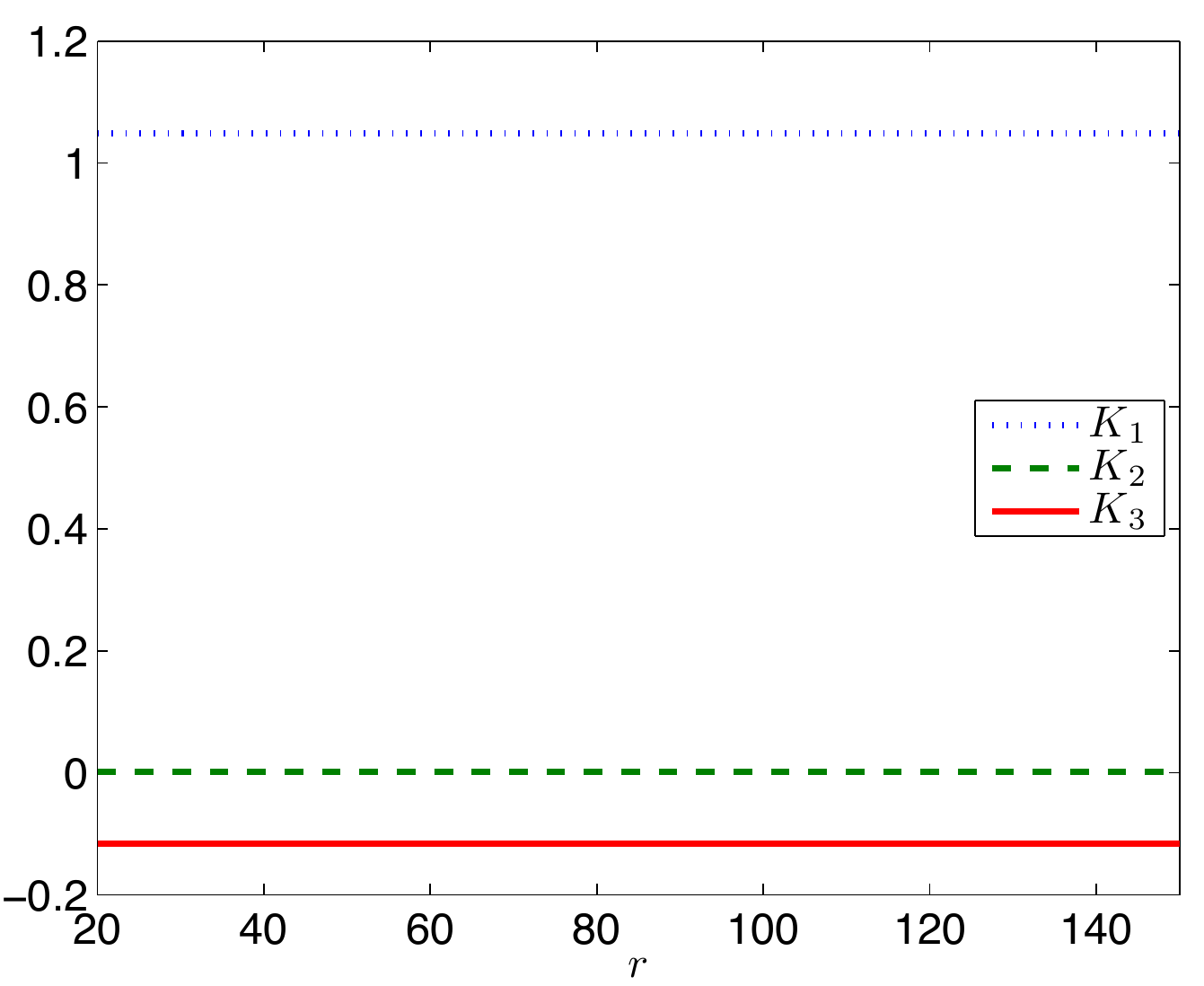}
    } \subfigure[Inner Products for $\calL_-^{(0)}$ functions for 3D
    Cubic]{
      \includegraphics[width=2.5in]{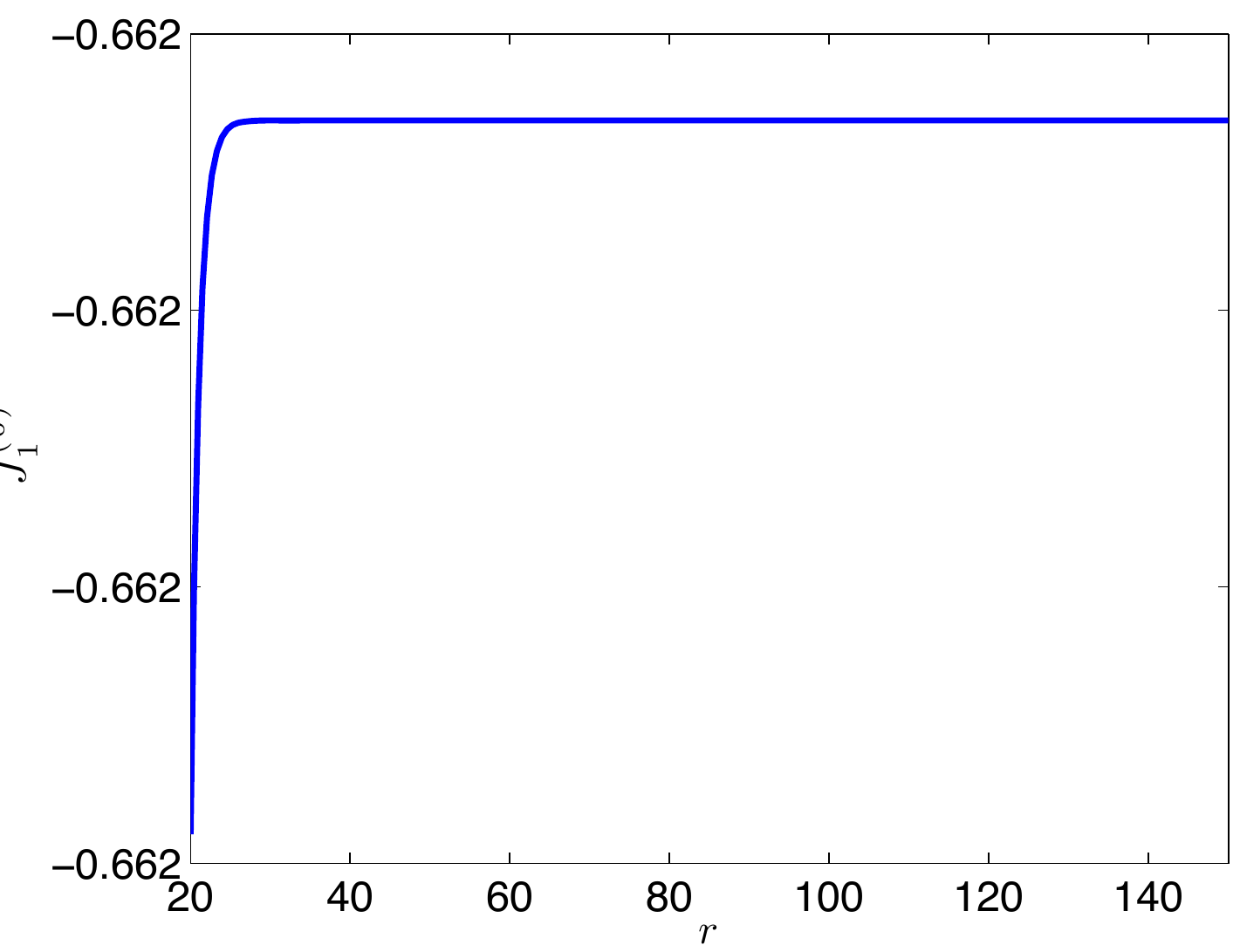}
    } \subfigure[Inner Products for $\calL_+^{(1)}$ functions for 3D
    Cubic]{
      \includegraphics[width=2.5in]{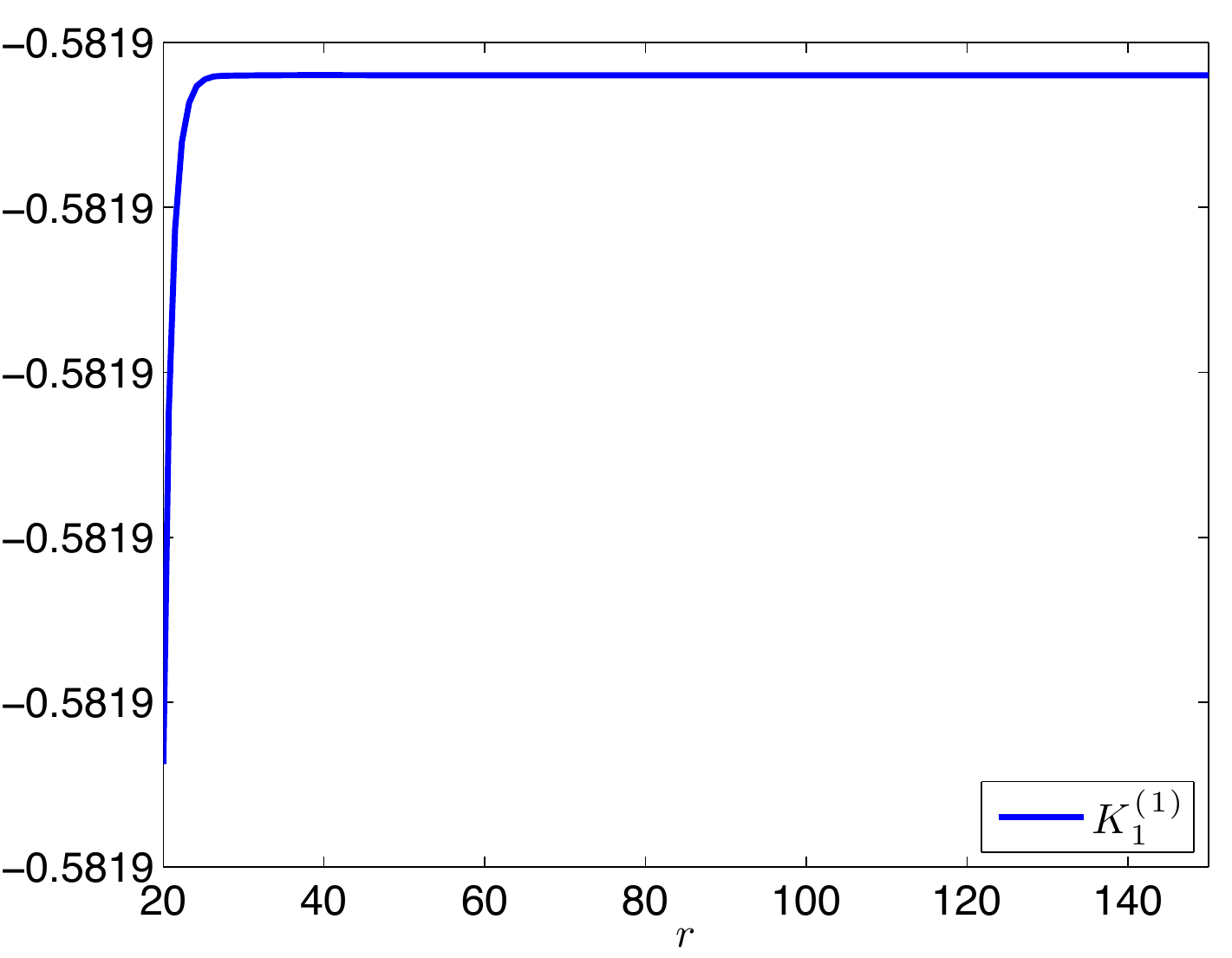}
    }
    \caption{Inner products for the 3d cubic equation.}
    \label{fig:ips_3d_cubic}
  \end{center}
\end{figure}

\begin{figure}
  \begin{center}
    \subfigure[$\calL_\pm^{(e)}$ with
    $\sigma=2.1$.]{\includegraphics[width=2.1in]{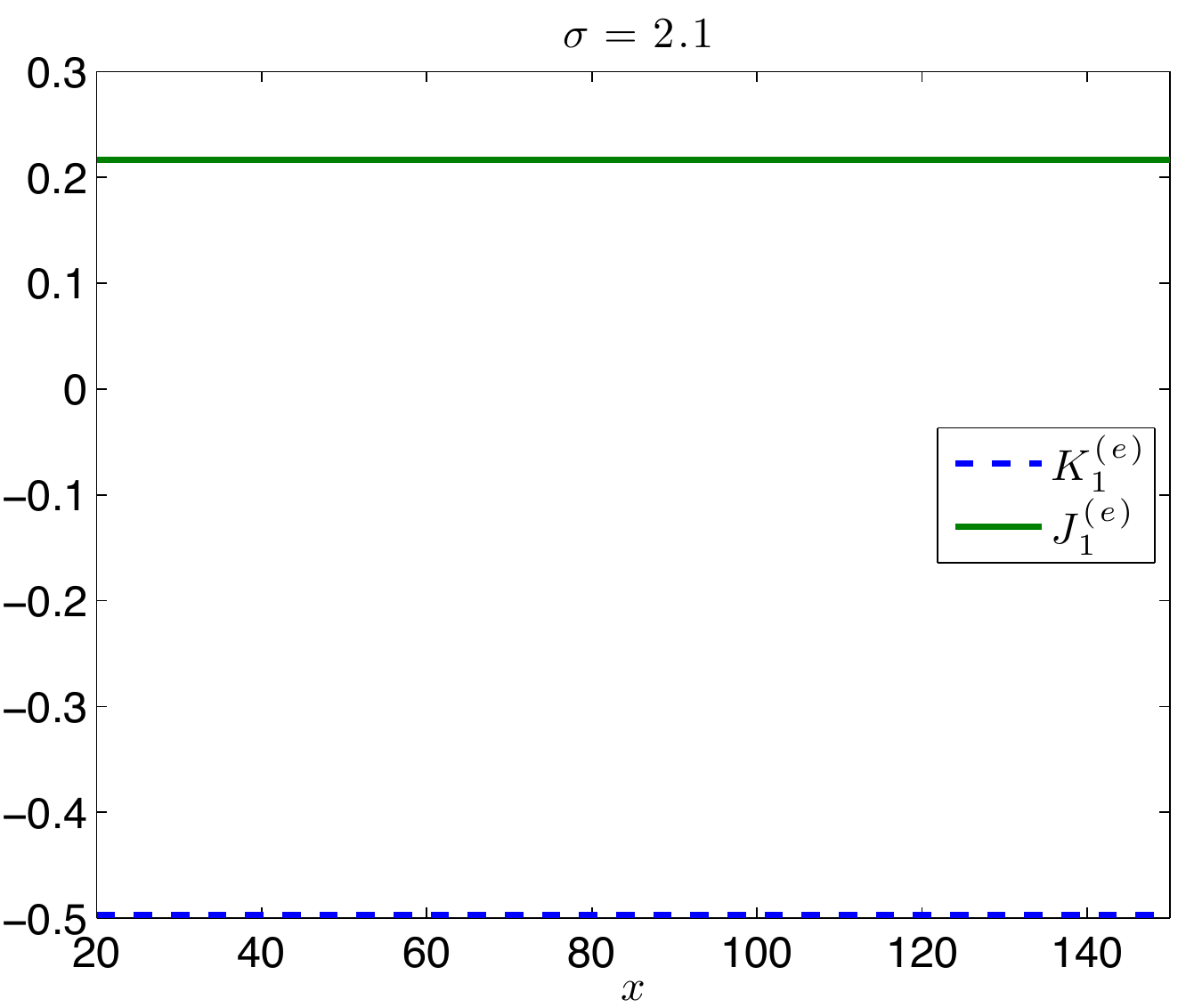}}
    \subfigure[$\calL_+^{(o)}$with
    $\sigma=2.1$.]{\includegraphics[width=2.1in]{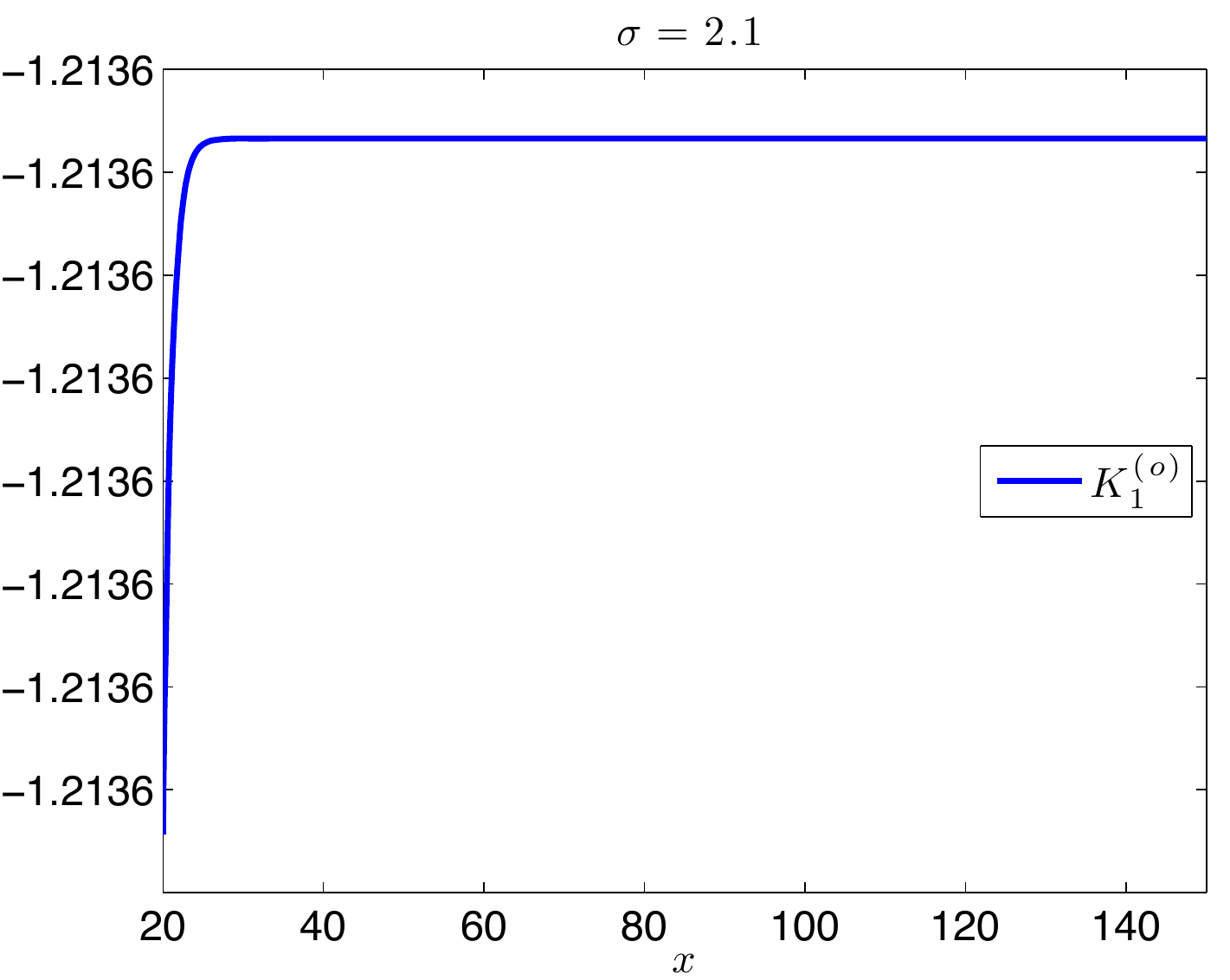}}

    \subfigure[$\calL_\pm^{(e)}$ with
    $\sigma=2.5$.]{\includegraphics[width=2.1in]{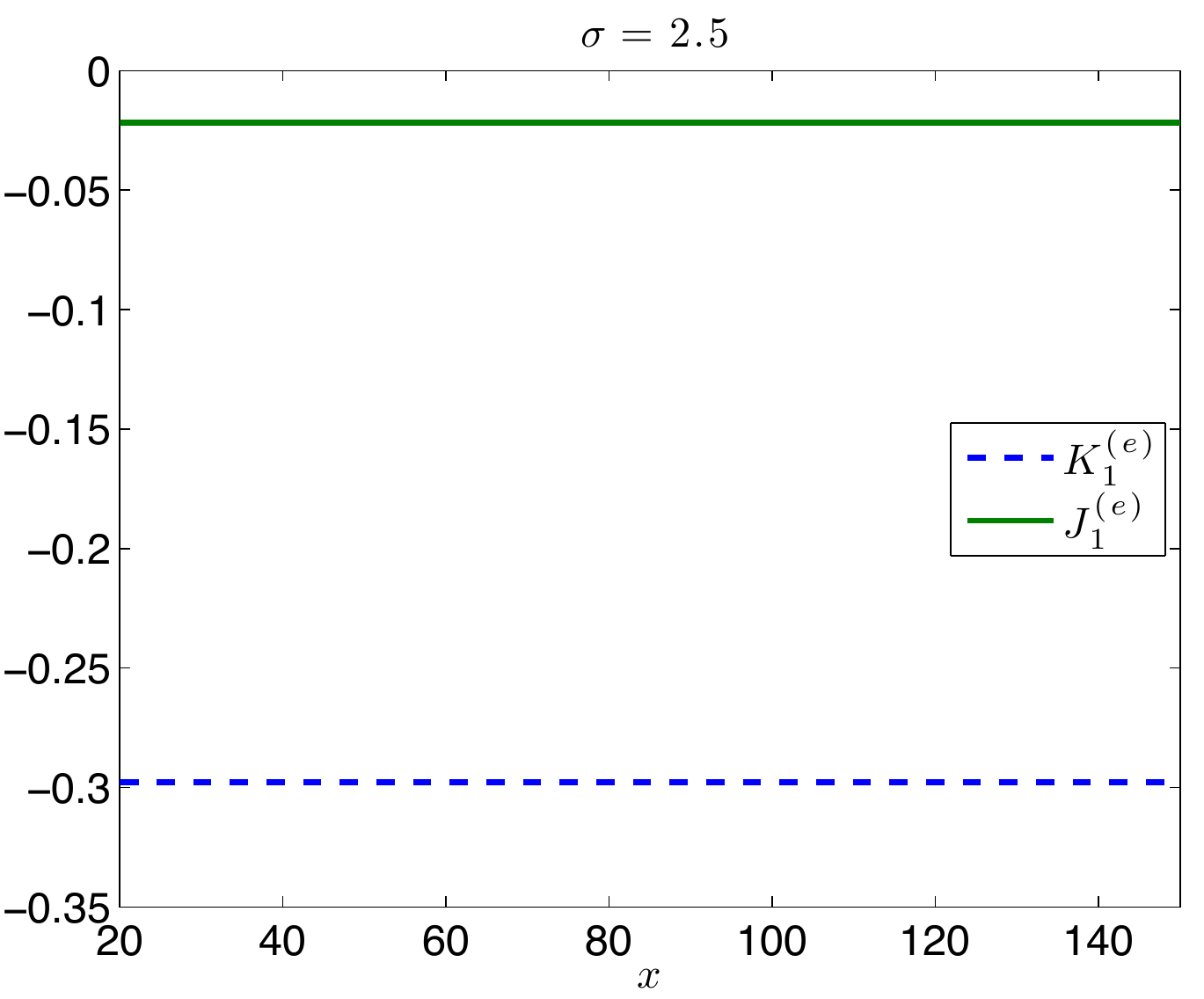}}
    \subfigure[$\calL_+^{(o)}$with
    $\sigma=2.5$.]{\includegraphics[width=2.1in]{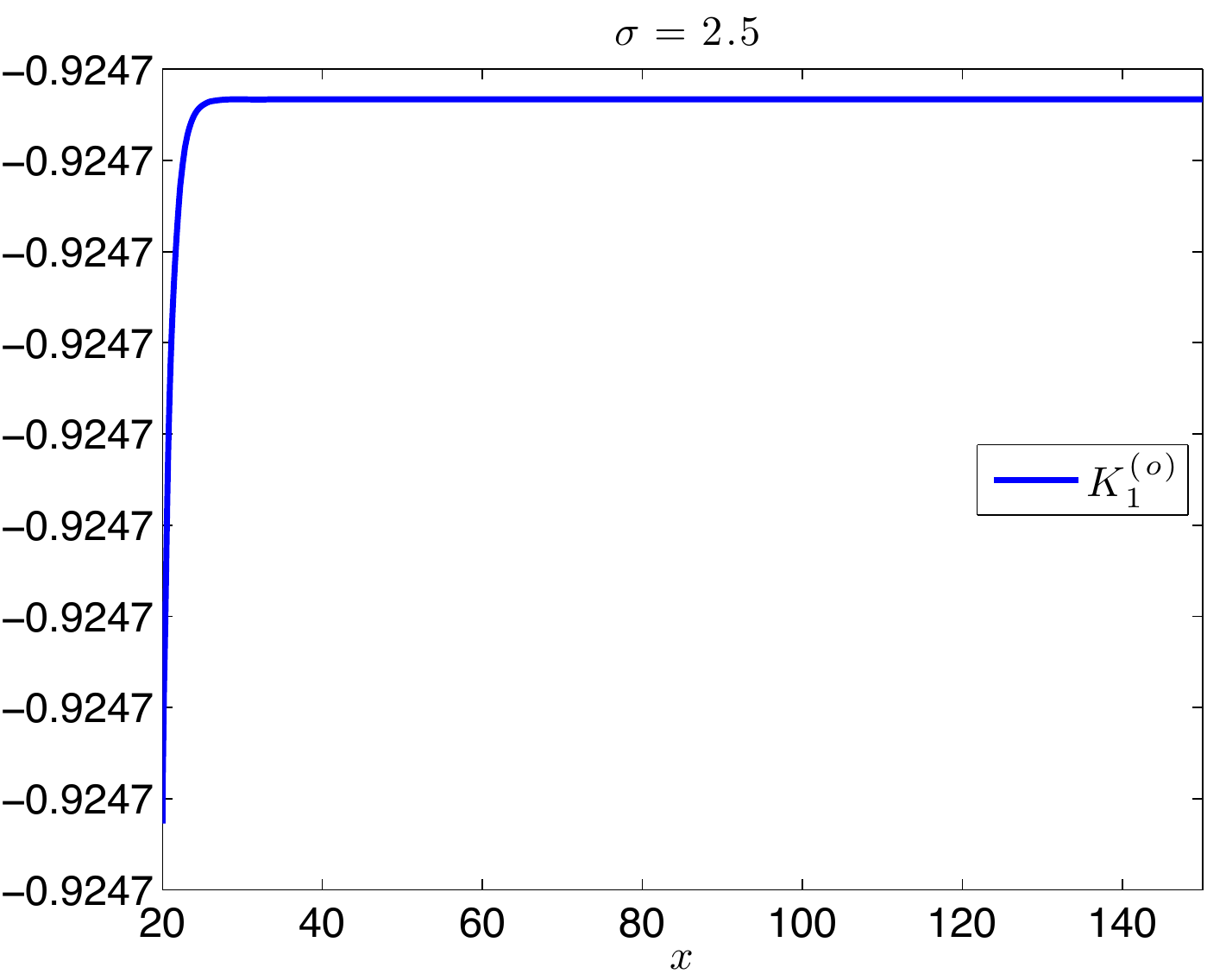}}

    \subfigure[$\calL_\pm^{(e)}$ with
    $\sigma=3.0$.]{\includegraphics[width=2.1in]{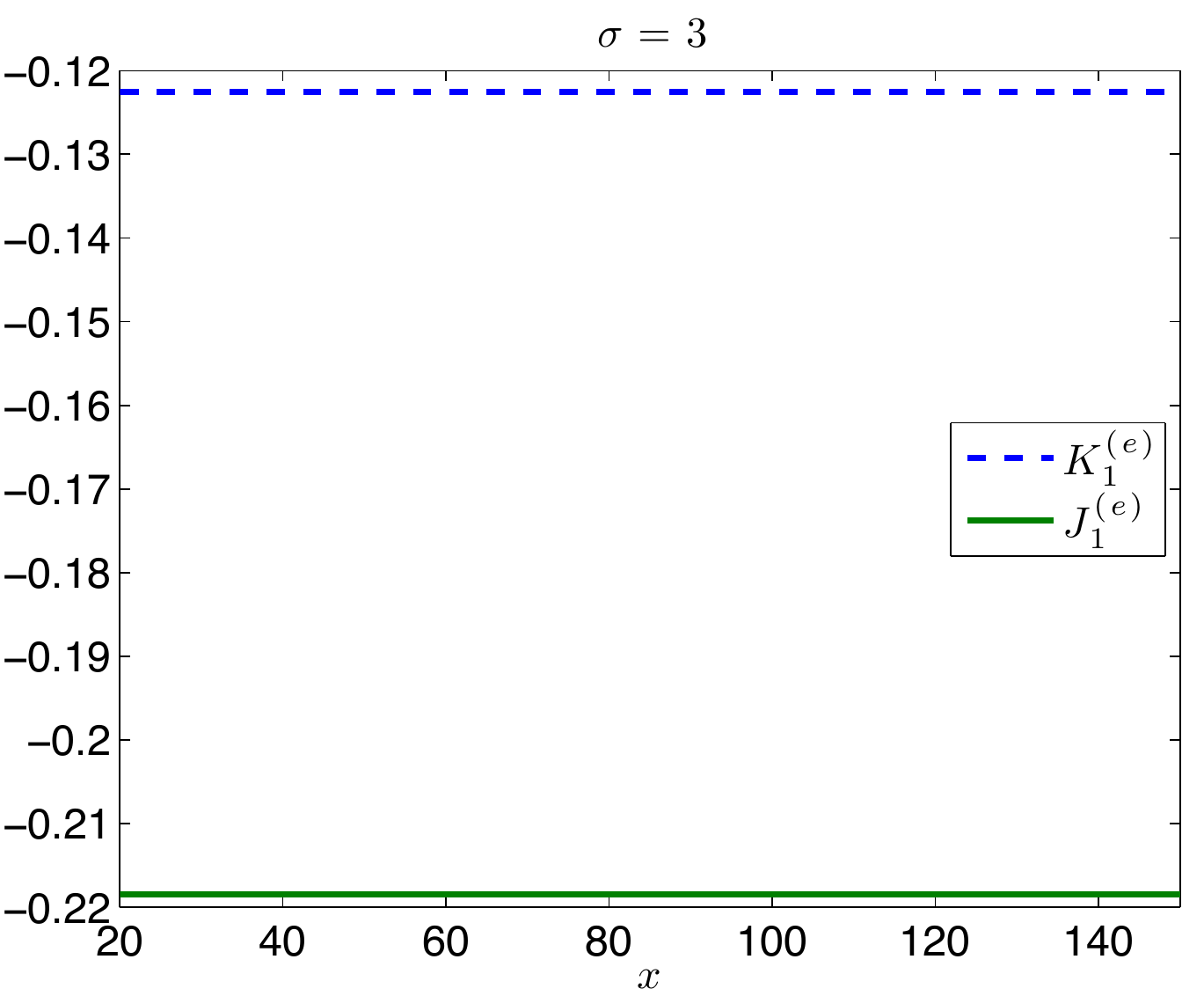}}
    \subfigure[$\calL_+^{(o)}$with
    $\sigma=3.0$.]{\includegraphics[width=2.1in]{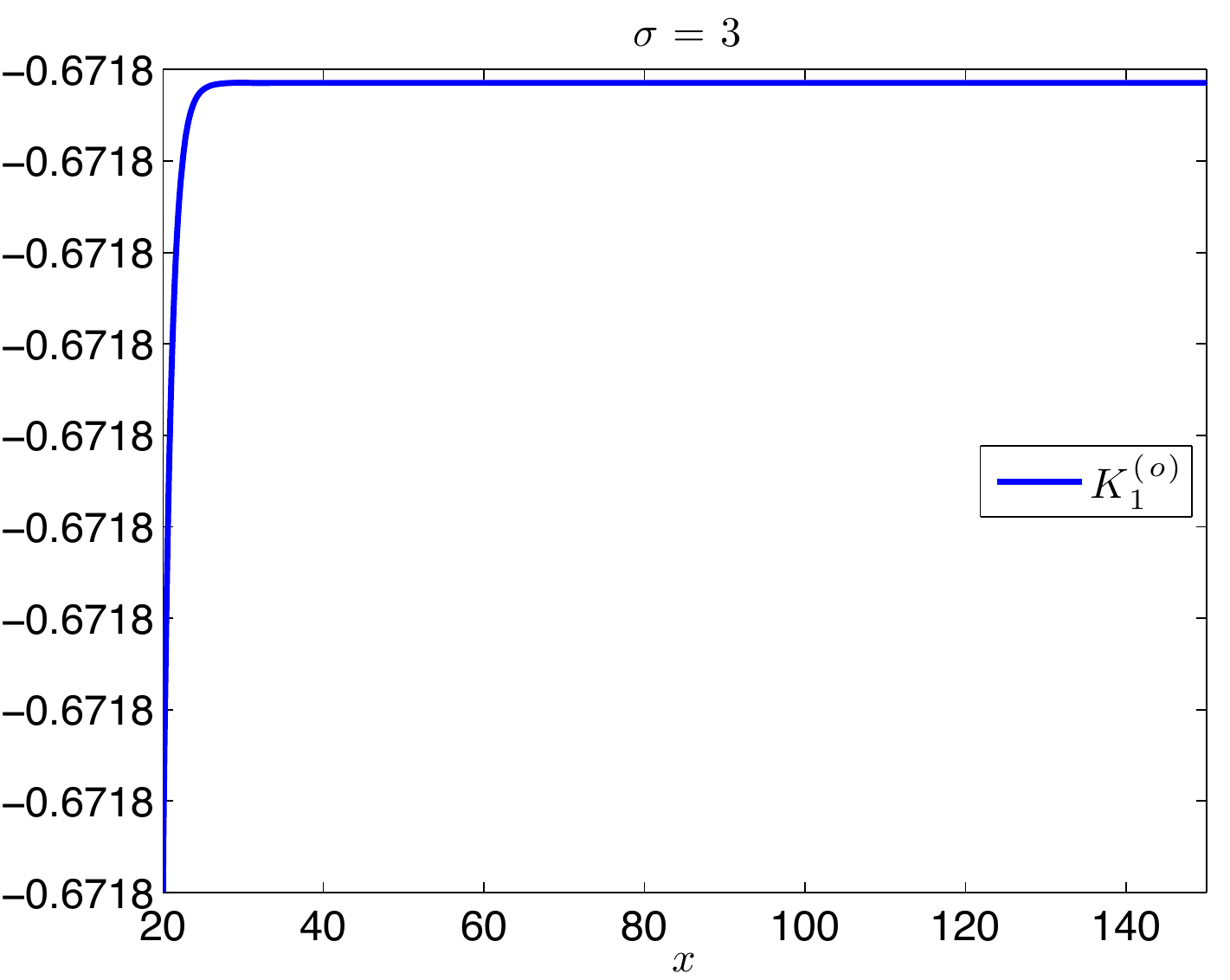}}
    \caption{Inner products for the 1d supercritical
      problems}
    \label{f:ip_1d_supercrit}
  \end{center}
\end{figure}

\clearpage

\begin{figure}
  \begin{center}
    \subfigure[$\calL_\pm^{(e)}$ in the critical case with the {\it
      natural} orthogonality
    conditions.]{\includegraphics[width=2.1in]{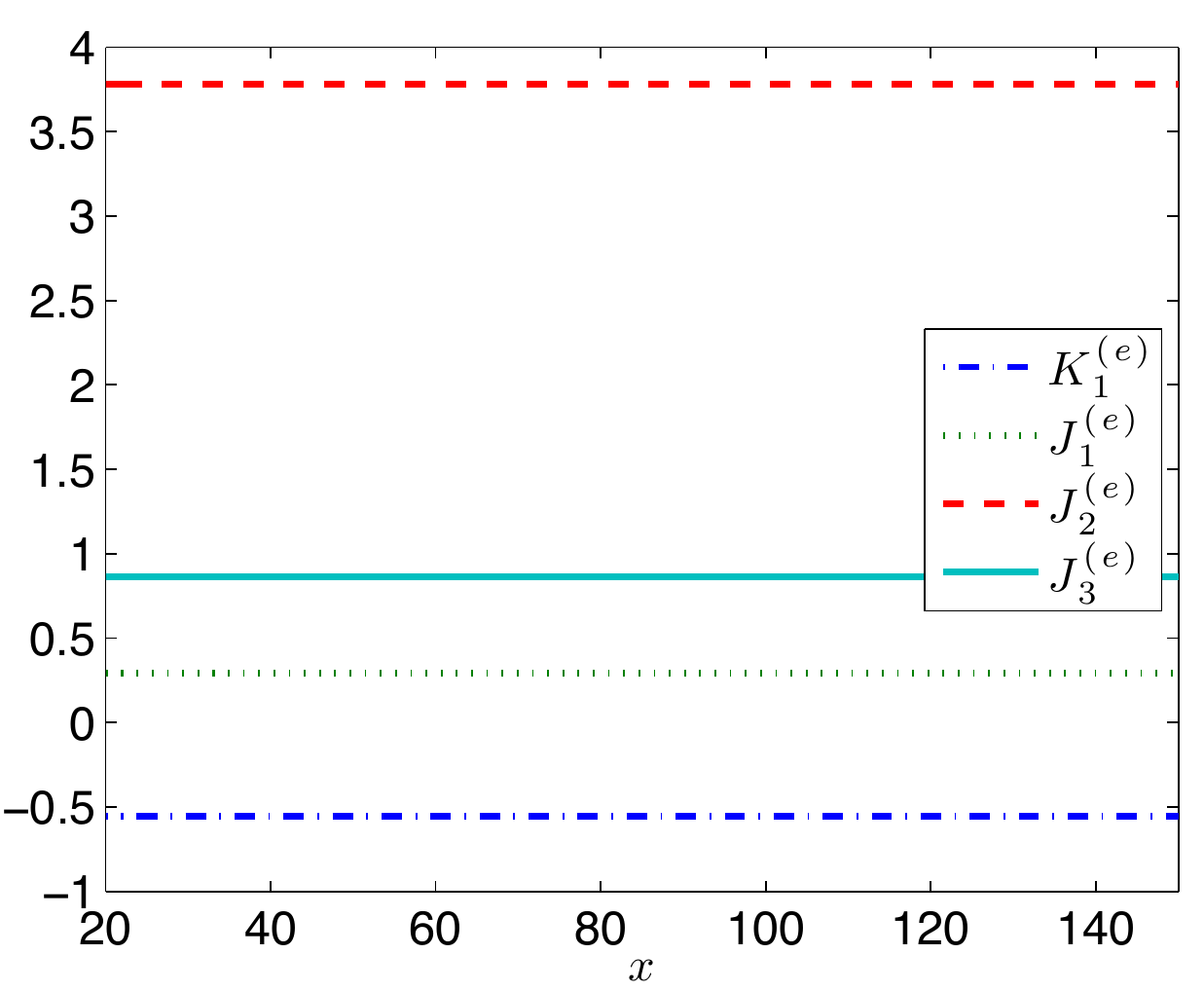}}
    \subfigure[$\calL_+^{(o)}$ in the critical
    case]{\includegraphics[width=2.1in]{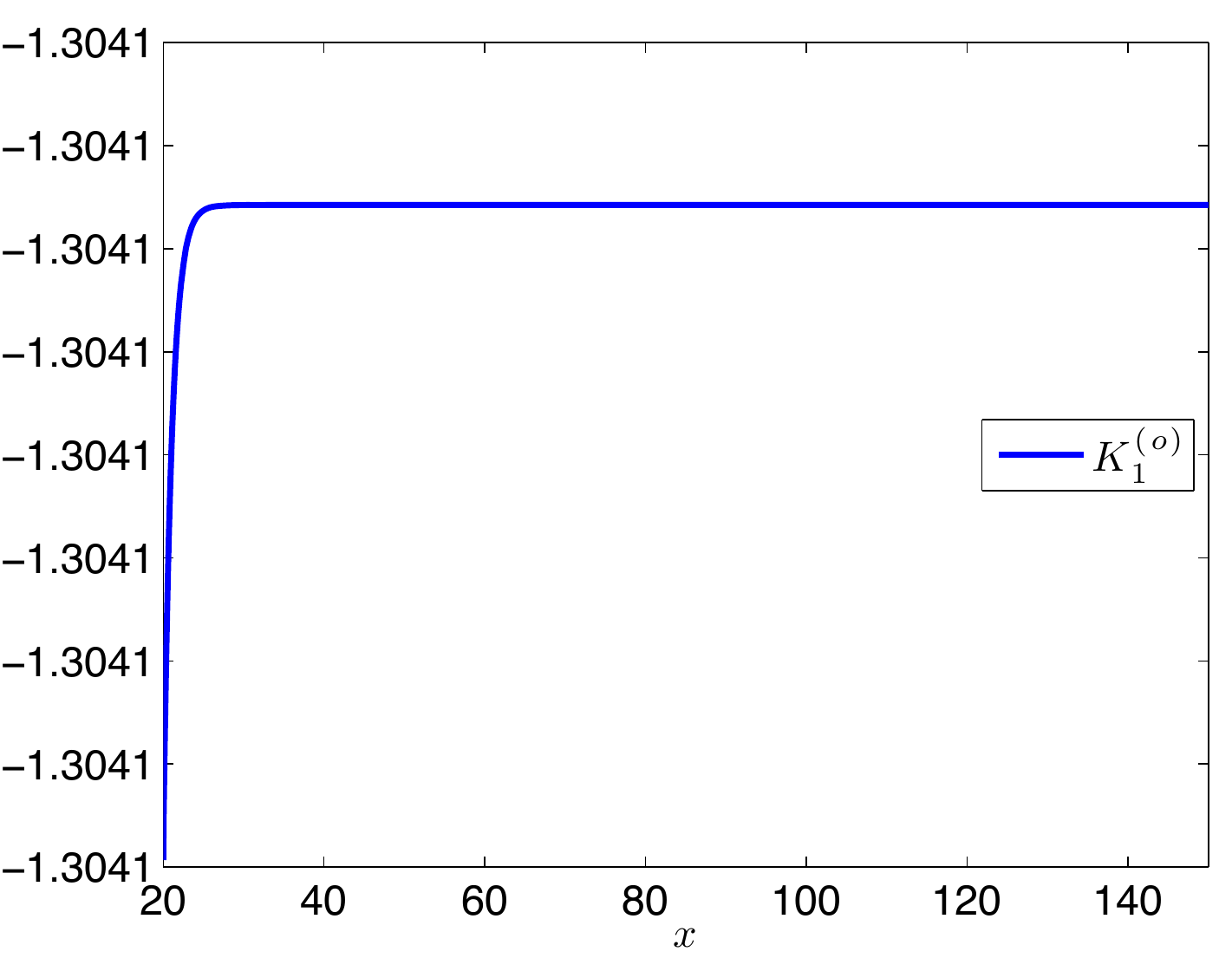}}

    \subfigure[$\calL_-^{(e)}$ in the critical case with the
    alternative orthogonality
    condition.]{\includegraphics[width=2.1in]{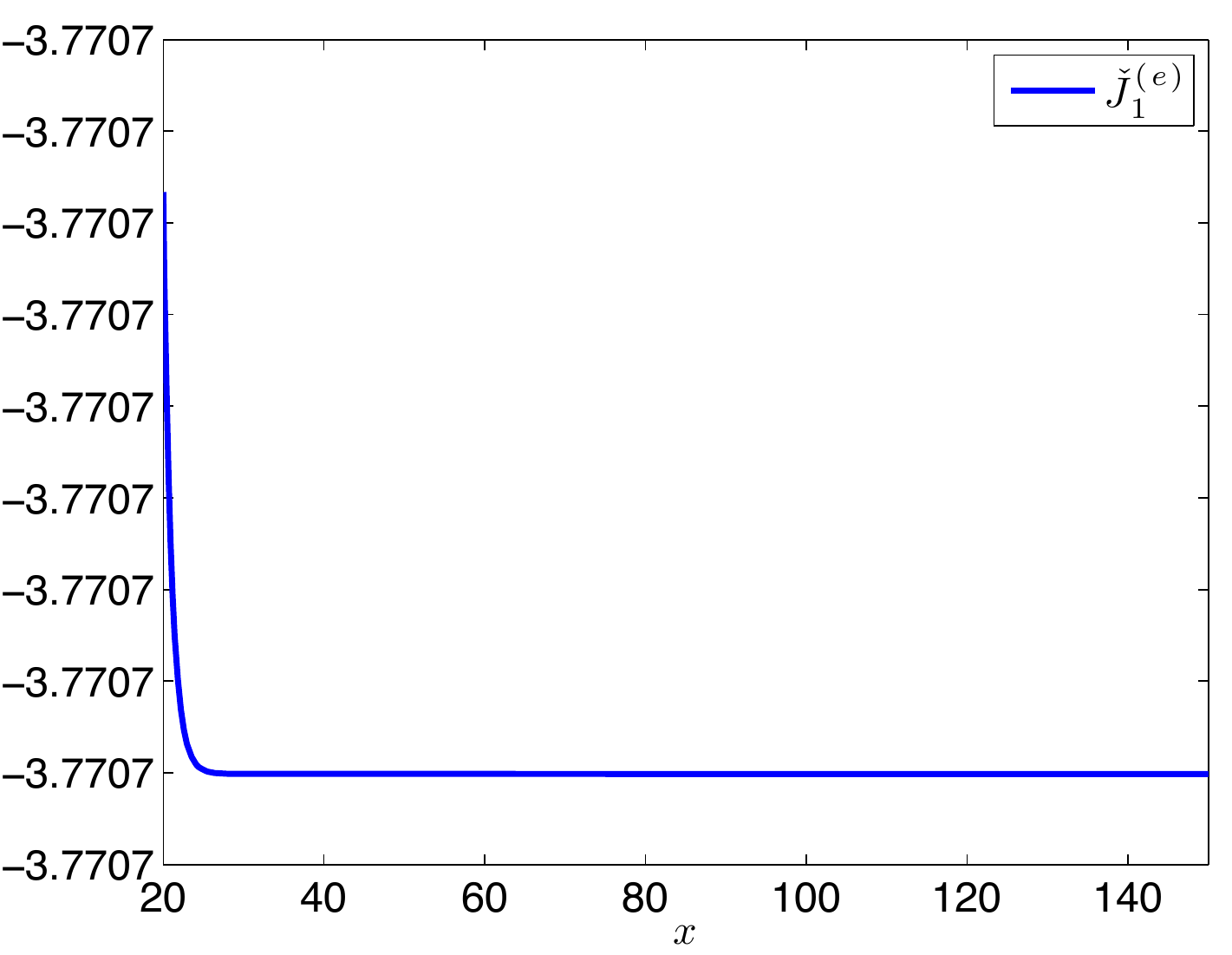}}
    \subfigure[$\calL_-^{(e)}$ in the critical case with the FMR
    orthogonality
    conditions.]{\includegraphics[width=2.1in]{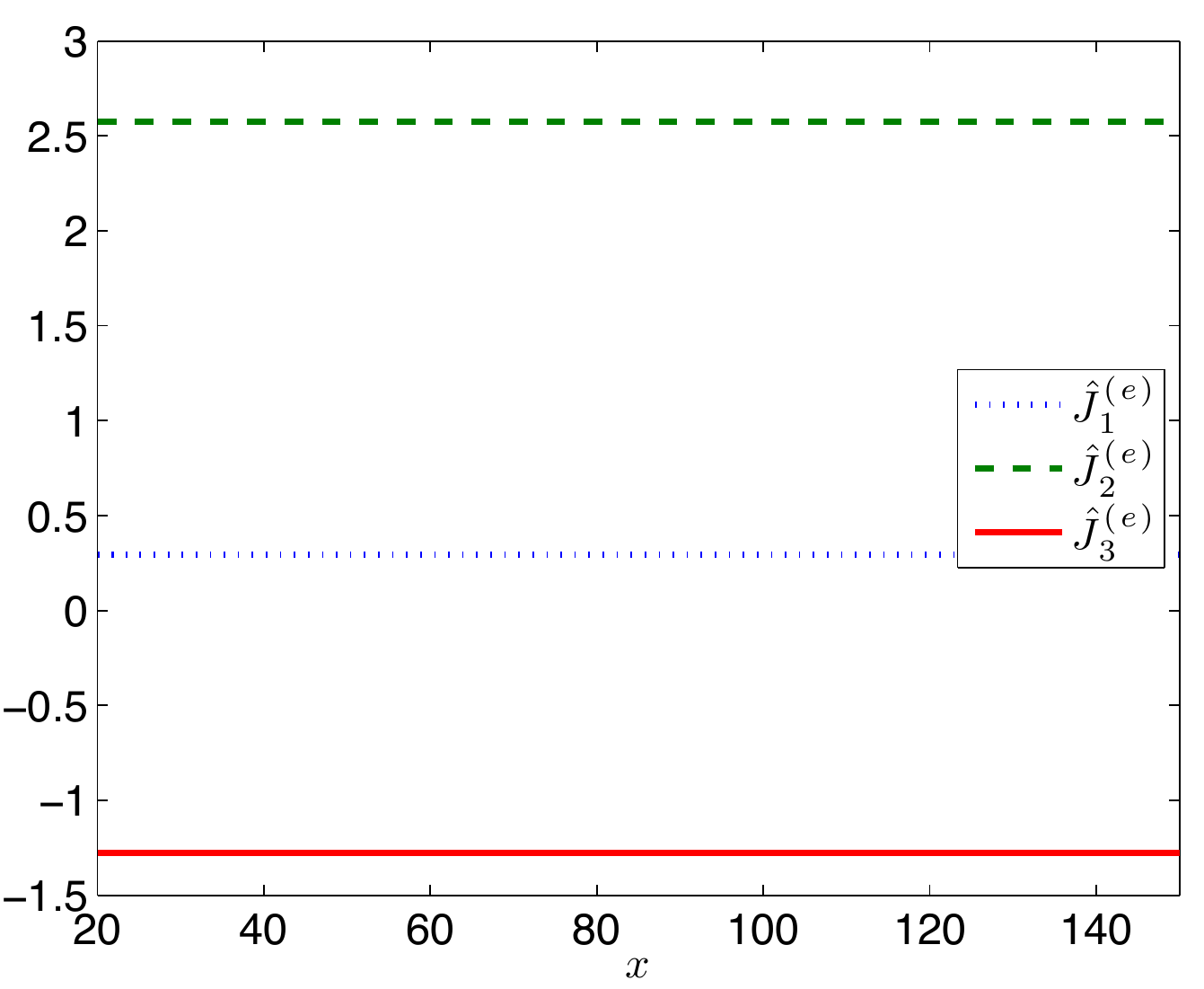}}
  \end{center}
  \caption{Inner products for the 1d critical problem with various
    orthogonality conditions.}
  \label{f:ip_1d_crit}
\end{figure}

\newpage \bibliography{ms-bib}

\end{document}

%% file: grsmacros.tex
\newcommand{\abs}[1]{\lvert#1\rvert}
\newcommand{\norm}[1]{\lVert#1\rVert}
\newcommand{\paren}[1]{\left(#1\right)}

\newcommand{\set}[1]{\left\{#1\right\}}

\newcommand{\bigo}{\mathrm{O} }
\newcommand{\matlab}{\textsc{Matlab} }

\newcommand{\ddx}{\frac{d}{dx}}
\newcommand{\ddr}{\frac{d}{dr}}

\newcommand{\cont}{\mathrm{cont.}}

\newcommand{\ind}{{\mathrm{ind}}}
\newcommand{\rad}{{\mathrm{rad}}}
\newcommand{\embd}{{\mathrm{em}}}
\newcommand{\rmax}{r_{\max}}
\newcommand{\xmax}{x_{\max}}
\newcommand{\grad}{\nabla}

\newcommand{\kerg}{\ker_{\mathrm{g}}}

\newcommand{\spn}{\mathrm{span}}

\newcommand{\bx}{\mathbf{x}}

\newcommand{\inner}[2]{\left\langle #1,#2\right\rangle}

\newcommand{\R}{\mathbb{R}}

\newcommand{\calL}{\mathcal{L}}
\newcommand{\calB}{\mathcal{B}}
\newcommand{\calV}{\mathcal{V}}
\newcommand{\calH}{\mathcal{H}}

\newcommand{\calU}{\mathcal{U}}

\newcommand{\by}{\bold{y}}
\newcommand{\bz}{\bold{z}}

\newtheorem{thm}{Theorem}
\newtheorem{prop}{Proposition}[section]
\newtheorem{cor}[prop]{Corollary}
\newtheorem{lem}[prop]{Lemma}

\newtheorem{rem}[prop]{Remark}

\newtheorem{defn}[prop]{Definition}

%% file: jlmmacros.tex
\newcommand{\RR}{{\mathbb R}}

\def\squarebox#1{\hbox to #1{\hfill\vbox to #1{\vfill}}}

\def\R{\mathbb R}
\def\reals{\mathbb R}

\def\cal{\mathcal}
\def\exp{e^}

\def\reals{{\mathbb R}}

\def\be{\begin{eqnarray*}}
\def\ee{\end{eqnarray*}}
\def\ben{\begin{eqnarray}}
\def\een{\end{eqnarray}}